\documentclass[a4paper]{amsart}
\usepackage{amsmath,amsthm,amssymb,mathtools,mathrsfs}
\usepackage{bm}
\usepackage{xcolor}
\usepackage{graphicx}
\usepackage{tikz-cd}
\usepackage[all]{xy}
\usepackage{hyperref}
    \definecolor{urlcolor}{rgb}{0,0,0}
    \definecolor{linkcolor}{rgb}{.7,0.10,0.2}
    \definecolor{citecolor}{rgb}{.12,.54,.11}
\hypersetup{
      breaklinks=true, 
      colorlinks=true,
      urlcolor=urlcolor,
      linkcolor=linkcolor,
      citecolor=citecolor,
      pdftitle={Cohomological Mackey formula for quotient stacks},
      pdfauthor={Lucien Hennecart},
}

\usepackage{setspace}
\usepackage[shortlabels]{enumitem}
\usepackage[utf8]{inputenc}
\numberwithin{equation}{section}
\usepackage[top=2.3cm,
bottom=2.3cm,
left=2.3cm,
right=2.3cm]{geometry}
\usepackage[
    backend=biber,
    style=numeric,
    doi=false,
    isbn=false,
    url=false,
    eprint=false,
    block=space,
    maxcitenames=100,
    maxbibnames=100,
    sorting=anyt
]{biblatex}

\newtheorem{theorem}{Theorem}[section]
\newtheorem{corollary}[theorem]{Corollary}
\newtheorem{proposition}[theorem]{Proposition}
\newtheorem{proposition-definition}[theorem]{Proposition-Definition}
\newtheorem{lemma}[theorem]{Lemma}

\theoremstyle{definition} 
\newtheorem{definition}[theorem]{Definition}

\newtheorem{theorem-definition}[theorem]{Theorem-Definition}

\newtheorem{question}[theorem]{Question}

\theoremstyle{remark} 
\newtheorem{remark}[theorem]{Remark}
\newtheorem{example}[theorem]{Example}

\renewcommand{\emptyset}{\varnothing}
\renewcommand{\geq}{\geqslant}
\renewcommand{\leq}{\leqslant}
\renewcommand{\subset}{\subseteq}
\renewcommand{\setminus}{\smallsetminus}
\renewcommand{\tilde}{\widetilde}

\newcommand{\BD}{\mathbb{D}}

\newcommand{\BoA}{\mathbf{A}}
\newcommand{\BoC}{\mathbf{C}}

\newcommand{\BoG}{\mathbf{G}}
\newcommand{\bfH}{\mathbf{\mathsf{H}}}

\newcommand{\BoQ}{\mathbf{Q}}
\newcommand{\BoR}{\mathbf{R}}
\newcommand{\BoZ}{\mathbf{Z}}

\newcommand{\CC}{\mathcal{C}}
\newcommand{\CD}{\mathcal{D}}

\newcommand{\CH}{\mathcal{H}}

\newcommand{\CM}{\mathcal{M}}

\newcommand{\CO}{\mathcal{O}}

\newcommand{\CV}{\mathcal{V}}
\newcommand{\CW}{\mathcal{W}}

\newcommand{\FC}{\mathfrak{C}}

\newcommand{\FF}{\mathfrak{F}}

\newcommand{\FH}{\mathfrak{H}}

\newcommand{\FM}{\mathfrak{M}}

\newcommand{\FS}{\mathfrak{S}}

\newcommand{\FX}{\mathfrak{X}}
\newcommand{\FY}{\mathfrak{Y}}
\newcommand{\FZ}{\mathfrak{Z}}

\newcommand{\SL}{\mathscr{L}}

\newcommand{\rmBPS}{\mathrm{BPS}}

\newcommand{\rmc}{\mathrm{c}}

\newcommand{\crit}{\mathrm{crit}}
\newcommand{\cms}{/\!\!/}

\newcommand{\Eu}{\mathfrak{Eu}}

\newcommand{\Fg}{\mathfrak{g}}
\newcommand{\GL}{\mathrm{GL}}

\newcommand{\Fh}{\mathfrak{h}}

\newcommand{\Frac}{\mathrm{Frac}}

\newcommand{\HO}{\mathrm{H}}

\newcommand{\id}{\mathrm{id}}
\newcommand{\IH}{\mathrm{IH}}

\newcommand{\Ind}{\mathrm{Ind}}

\newcommand{\Lie}{\mathrm{Lie}}
\newcommand{\Fl}{\mathfrak{l}}
\newcommand{\MHM}{\mathrm{MHM}}
\newcommand{\MMHM}{\mathrm{MMHM}}
\newcommand{\nil}{\mathrm{nil}}
\newcommand{\rmN}{\mathrm{N}}

\newcommand{\loc}{\mathrm{loc}}
\newcommand{\rmm}{\mathrm{m}}
\newcommand{\Fp}{\mathfrak{p}}
\newcommand{\Perv}{\mathrm{Perv}}

\newcommand{\pt}{\mathrm{pt}}
\newcommand{\rank}{\mathrm{rank}}

\newcommand{\Res}{\mathrm{Res}}

\newcommand{\rmBM}{\mathrm{BM}}

\newcommand{\rat}{\mathbf{rat}}

\newcommand{\Sym}{\mathrm{Sym}}

\newcommand{\Tan}{\mathrm{T}}

\newcommand{\Ft}{\mathfrak{t}}
\newcommand{\vir}{\mathrm{vir}}
\newcommand{\sfW}{\mathsf{W}}
\newcommand{\rmX}{\mathrm{X}}

\newcommand{\dd}{\mathbf{d}}

\title{Cohomological Mackey formula for quotient stacks}
\date{\today}

\author{Lucien Hennecart}

\address{Laboratoire Ami\'enois de Math\'ematique Fondamentale et Appliqu\'ee, CNRS UMR 7352, Universit\'e de Picardie Jules Verne, 33 rue Saint Leu, 80000 Amiens, France}
\email{lucien.hennecart@u-picardie.fr}
\begin{document}
\raggedbottom
\begin{abstract}
In this paper, we construct a restriction morphism on the critical cohomology of an equivariant Landau--Ginzburg model associated to a representation of a reductive group equipped with an invariant function. We show a compatibility formula between this restriction map and the induction map considered in earlier work in the form of a Mackey-type formula, thereby giving the critical cohomology the structure of a localized induction-restriction system. In the case of quivers, our formulas give the localized coproduct of Davison for cohomological Hall algebras of quivers with potential.
\end{abstract}

\maketitle

\setcounter{tocdepth}{1}
\tableofcontents

\section{Introduction}

\subsection{Donaldson--Thomas theory of equivariant Landau--Ginzburg models}
In \cite{hennecart2024cohomological}, we started the study of the enumerative geometry of the quotient $V/G$ of a representation $V$ of a reductive group $G$. We established a \emph{cohomological integrality isomorphism} for the cohomology vector space $\HO^*_G(V)$ when $V$ is weakly symmetric (for example, self-dual). This space is a polynomial ring (and hence infinite-dimensional when $G$ has positive rank), and the cohomological integrality gives a decomposition involving finitely many finite-dimensional subspaces. It involves parabolic induction. One of the finite-dimensional pieces recovers in many cases the intersection cohomology $\IH^*(V\cms G)$ of the affine GIT quotient (\cite{meinhardt2019donaldson, bu2025cohomology}, \cite[Appendix~B]{hennecart2024cohomological2}, and \cite{hennecart2025bps} for the most general statement), which is usually challenging to compute. Building on this, we extended this study to equivariant Landau--Ginzburg models in \cite{hennecart2024cohomological2}. The input is a triple $(G,V,f)$ of a connected reductive group $G$, a finite-dimensional representation $V$ of $G$ and a $G$-invariant function $f$ on $V$. The goal is to understand the critical cohomology $\HO^*_{\crit}(V/G,f)$ of $V/G$ with respect to $f$. By definition, this is the $G$-equivariant cohomology $\HO^*_{G}(V,\varphi_f)$ of $V$ with coefficients in the $G$-equivariant complex of vanishing cycles $\varphi_f\BoQ_V$.  For symmetric representations, cohomological integrality extends to this situation, providing finite-dimensional subspaces $\rmBPS_{V,G,f,\lambda}$ for any cocharacter $\lambda\in\rmX_*(T)$ of a maximal torus $T$ of $G$. This can be seen as an extension of cohomological Donaldson--Thomas theory from the study of moduli stacks of objects in $3$-Calabi--Yau categories (in particular, stacks of representations of quivers with potentials) to $(-1)$-shifted symplectic stacks, in the local case. The goal is to understand the associated enumerative invariants and the structure of the finite-dimensional spaces $\rmBPS_{V,G,f,\lambda}$.

In \cite{hennecart2024cohomological}, we constructed induction morphisms for $\HO^*(V/G)$. The same construction also gives induction morphisms on $\HO^*_{\crit}(V/G,f)$. In this paper, we recall the definition of the induction morphisms for critical cohomology and we explain how to construct \emph{restriction morphisms}  on $\HO^*_{\crit}(V/G,f)$. The definition of the restriction morphisms in many cases requires us to \emph{localize} the cohomology with respect to some Euler classes. We show that the induction and restriction morphisms satisfy a compatibility that we call \emph{Mackey formula} by analogy with the restriction formula valid for representations of finite groups, \cite[Theorem~5.2.1]{webb2016course} (see \S\ref{subsubsection:mackeyformula} for a short account). The situation of a $G$-equivariant LG-model is referred to as the \emph{$3d$-situation}. We also study \emph{$2d$-situations}. A $2d$-situation is given by a pair $V,W$ of representations of a reductive group $G$ with a $G$-equivariant function $\mu\colon V\rightarrow W^*$ to the dual representation. We may relate the $2d$-situation to a $3d$-situation via cohomological dimensional reduction following \cite{davison2017critical,kinjo2022dimensional}.

\subsection{Induction and restriction morphisms and functors}
\label{subsection:indresfunctors}
We present a selection of occurrences of induction and restriction morphisms in the literature, in order for the reader to become familiar with the induction-restriction formalism and appreciate the similarities and differences with the formalism presented here (in particular, Theorem~\ref{theorem:mainintro}).

\subsubsection{Mackey formula for representations of finite groups}
\label{subsubsection:mackeyformula}
The original Mackey formula appears in the context of representations of finite groups (see for example \cite[Theorem~5.2.1]{webb2016course}). Let $G$ be a finite group and $H\subset G$ a subgroup. Then, a representation $V$ of $G$ may be \emph{restricted} to a representation of $H$. We denote the restriction by $\Res_{G}^HV$. A representation $V$ of $H$ may be \emph{induced} to a representation of $G$ via $\Ind_H^GV=\BoC[G]\otimes_{\BoC[H]} V$, where $\BoC[G], \BoC[H]$ are the corresponding group algebras. The Mackey formula expresses how these two operations commute with each other. Let $H,K\subset G$ be two subgroups and $V$ a representation of $H$. Then,
\[
 \Res_{G}^K\Ind_H^GV\cong \bigoplus_{x\in K\backslash G/H}\Ind_{xHx^{-1}\cap K}^K\Res^{xHx^{-1}\cap K}_{xHx^{-1}}x\cdot V
\]
where $x\cdot V$ is the representation of $xHx^{-1}\subset G$ with underlying vector space $V$ given by $xhx^{-1}\cdot v\coloneqq h\cdot v$ for $v\in V$ and $h\in H$.

\subsubsection{Restriction morphism and cohomological integrality}
In \cite{hennecart2024cohomological}, in order to prove the cohomological integrality isomorphism for weakly symmetric representations of reductive groups, we used an avatar of the restriction morphism we define in this paper for arbitrary representations of reductive groups. The restriction defined in \cite{hennecart2024cohomological} is crucial to prove the injectivity of the cohomological integrality morphism. It is defined as follows (we refer to \cite{hennecart2024cohomological} for the precise notations and definitions). Let $G$ be a reductive group and $V$ a representation of $G$. For every cocharacter $\lambda\in \rmX_*(T)$, there is a Levi subgroup $G^{\lambda}$ and a representation $V^{\lambda}$ of $G^{\lambda}$ defined as fixed loci for the $\BoG_{\rmm}$-action induced by $\lambda$. There are natural isomorphisms $\HO^*(V/G)\cong \HO^*_{T}(\pt)^{\sfW}$ and $\HO^*(V^{\lambda}/G^{\lambda})\cong \HO^*_T(\pt)^{\sfW^{\lambda}}$ where $\sfW$ (resp. $\sfW^{\lambda}$) is the Weyl group of $G$ (resp. $G^{\lambda}$). In particular, we have a natural inclusion, which we may see as a restriction morphism
\[
 \Res_{\lambda}\colon \HO^*(V/G)\rightarrow\HO^*(V^{\lambda}/G^{\lambda})\,.
\]
This restriction is sufficient to show the injectivity of the cohomological integrality map in \cite{hennecart2024cohomological}, but it is \emph{not} compatible in a nice way with the induction morphisms of \emph{loc. cit.}. The restriction morphisms are defined in  the current paper using a slight variation (i.e. division by the Euler class of some vector bundle, which is where we require \emph{localization}) of the restriction morphisms $\Res_{\lambda}$ and correct this incompatibility: they satisfy the Mackey formula, Theorem~\ref{theorem:mainintro}.

\subsubsection{Localized coproduct on cohomological Hall algebras}
In \cite{davison2017critical}, Davison defined a \emph{localized} coproduct on the cohomological Hall algebra of a quiver with potential. His motivation comes from the fact that cohomological Hall algebras are expected to give geometric constructions of quantum groups, and should therefore carry a natural coproduct. This expectation is realized in various works \cite{schiffmann2017cohomological,davison2023bps,botta2023okounkov,schiffmann2023cohomological}. The localized coproduct is an important structure, featured in the proof of Okounkov's conjecture, in particular for the comparison between the cohomological Hall algebra and the Maulik--Okounkov Yangian \cite{botta2023okounkov}; see also \cite{schiffmann2023cohomological} for a different point of view on this conjecture.

More recently, an actual coproduct was constructed on CoHAs of coherent sheaves on quasi-projective symplectic surfaces in \cite{davison2026hecke}.

\subsubsection{Eisenstein series and constant term functors}
In the Langlands programme, the Eisenstein series can be understood as an induction while the constant term is analogous to a restriction. The Langlands constant term formula \cite[Chapter 8]{fleig2018eisenstein} gives an expression for the restriction of the Eisenstein series in terms of a sum over a Weyl group. This is a form of the Mackey formula.

\subsubsection{Springer theory}
Generalized Springer theory for the Lie algebra $\Fg$ of a reductive algebraic group can be understood in terms of the induction/restriction diagrams of quotient stacks
\[
 \Fl/L\leftarrow\Fp/P\rightarrow \Fg/G
\]
obtained for each choice of a parabolic subgroup $P\subset G$ with Lie algebra $\Fp$, with Levi subgroup $L$ and Levi subalgebra $\Fl$ \cite{gunningham2018generalized}. Reading the diagram from left to right gives the induction functor $\Ind_L^G$ while reading the diagram from right to left gives the restriction functor $\Res_G^L$ on the categories of perverse sheaves or $D$-modules on the stacks involved. On the Abelian categories considered in \cite[Theorem B]{gunningham2018generalized}, these functors do not depend on the parabolic subgroup with given Levi $L$. The induction and restriction functors satisfy a Mackey formula
\[
 \Res^{M}_G\Ind_L^G\cong\bigoplus_{w\in \sfW_M\backslash \sfW/\sfW_L}\Ind_{M\cap (\dot{w}\cdot L)}^M\Res^{M\cap (\dot{w}\cdot L)}_{\dot{w}\cdot L}\dot{w},
\]
where $\dot w\in G$ is any lift of $w$ \cite[Theorem B]{gunningham2018generalized}.

The split formula just cited concerns the Abelian categories; a derived formulation must keep track of the Mackey filtration and its extensions. In work in progress, we study the corresponding induction and restriction formalism for constructible complexes and complexes of $D$-modules on a representation $V$ of a reductive group $G$. Via the characteristic cycle map, this category of perverse sheaves is important to understand the Borel--Moore homology $\HO^{\rmBM}_*(\Tan^*(V/G))$ of the cotangent stack of $V/G$ following \cite{hennecart2024geometric} in the case of quivers. In the case of quivers, this was the cornerstone to determine the structure of the whole BPS Lie algebra in \cite{davison2023bps} and we expect the study of constructible complexes on $V/G$ to be crucial to understand the BPS sheaf of $0$-shifted cotangent stacks of smooth stacks or also more general $0$-shifted symplectic stacks.

\subsubsection{Perverse sheaves on hyperplane arrangements}
Induction and restriction systems appear in the description of the category of perverse sheaves on the quotient $\Fh/\sfW$ for a Cartan subalgebra $\Fh$ of a reductive Lie algebra $\Fg$, for some natural stratification \cite{kapranov2024langlands}. The description is in terms of the \emph{perverse Coxeter category}, which is built out of a formalization of induction and restriction morphisms. Then, the Abelian category of perverse sheaves on $\Fh/\sfW$ with respect to a natural stratification admits a description in terms of the Abelian category of functors from the perverse Coxeter category to the category of vector spaces. One of the motivations of the authors of \emph{op. cit.} is to categorify perverse sheaves in terms of perverse schobers \cite{kapranov2014perverse} on hyperplane arrangements. The induction and restriction satisfy a Mackey formula \cite[Proposition~4.6.1]{kapranov2024langlands} (called ``Langlands formula'').

\subsection{Perverse sheaves on smooth stacks}The author together with Chenjing Bu \cite{bu2026mackey} developped a Mackey formula in any $6$-functor formalism for smooth stacks. In that framework, the component lattice \cite{bu2025intrinsic} replaces the Coxeter system. We refer to \emph{loc.cit.} for the details.

\subsection{Main results}
In this paper, we prove a cohomological version of Mackey's formula associated to any triple $(G,V,f)$ of a reductive group $G$, a representation $V$ of $G$ and a $G$-invariant regular function $f\colon V\rightarrow \BoC$ (Theorem~\ref{theorem:mainintro}). This is the commutation rule between the induction and restriction morphisms on the critical cohomology of the stack $V/G$ with respect to the function $f$. To express the induction and restriction functors conveniently, we first introduce the \emph{Coxeter complex} associated to $(G,V)$.

\subsubsection{Coxeter complex}
\label{subsubsection:Coxeter-Complex}
The Coxeter complex is defined in more detail in \S\ref{section:Coxetercomplex}. Let $G$ be a reductive group and $V$ a finite-dimensional representation of $G$. We let $T\subset G$ be a maximal torus and $\Fh\coloneqq\Lie(T)$. We let $\CW(V)\subset\rmX^*(T)$ be the collection of weights of $V$ (with multiplicities taken into account) and $\CW(\Fg)$ the collection of weights of $\Fg\coloneqq\Lie(G)$ which is the adjoint representation of $G$ (each non-zero weight of the adjoint representation has multiplicity $1$). We let $\Fh_{\BoR}\coloneqq\rmX_*(T)\otimes_{\BoZ}\BoR$ be the real form of the Cartan subalgebra $\Fh\subset\Fg$. We may identify the dual $\Fh_{\BoR}^*$ with $\rmX^*(T)\otimes_{\BoZ}\BoR$ (where $\rmX^*(T)$ denotes the lattice of characters of $T$) so that the natural pairings $\rmX^*(T)\times\rmX_*(T)\rightarrow\BoZ$ and $\Fh_{\BoR}^*\times\Fh_{\BoR}\rightarrow\BoR$ are compatible.

 The \emph{Coxeter complex} associated to the pair $(G,V)$ of a representation $V$ of a reductive group $G$ is the pair $(\Fh_{\BoR},\bfH)$ of the real form $\Fh_{\BoR}$ of a Cartan subalgebra of $\Fg$ and the hyperplane arrangement $\bfH\coloneqq \bigcup_{\alpha\in(\CW(V)\cup\CW(\Fg))\setminus\{0\}}\{\alpha=0\}\subset \Fh_{\BoR}$. We let $\FH\coloneqq\{H_{\alpha}\colon \alpha\in(\CW(V)\cup\CW(\Fg))\setminus\{0\}\}$ be the set of hyperplanes of $\bfH$. The \emph{cells} of the Coxeter complex are the connected components of the subsets
 \[
  \bigcap_{H\in\FH_1}H\setminus\bigcup_{H\in\FH_2}H,
 \]
for some decomposition $\FH=\FH_1\sqcup\FH_2$. We denote by $\FC$ the set of cells of the Coxeter complex. Cells give a partition of $\Fh_{\BoR}$ by locally closed convex subsets (for the analytic topology): $\Fh_{\BoR}=\bigsqcup_{C\in\FC}C$.

The \emph{flats} of the Coxeter complex are intersections of hyperplanes: $\bigcap_{\alpha\in\CW}H_{\alpha}$ for $\CW\subset(\CW(V)\cup\CW(\Fg))\setminus\{0\}$. If $C\in\FC$, we denote by $\langle C\rangle$ the linear span of $C$. It is a flat of the Coxeter complex. We denote by $\FF$ the set of flats.

The set $\FC$ of cells of the Coxeter complex may be ordered by reverse inclusion of the cell closures (see \S\ref{subsection:partialorder}): $C\preceq C'\iff C'\subset\overline{C}$. We similarly order the set of flats. We may compare cells and flats as follows. For any $C\in \FC$ and $F\in\FF$, we write $C\preceq F$ if $F\subset\langle C\rangle$ and $F\preceq C$ if $C\subset F$.

Given two cells $C,C'\in\FC$, the \emph{Tits product} of $C$ and $C'$ is a third cell $C\circ C'$ such that $C\circ C'\preceq C$ (see \S\ref{subsection:Titsproduct}). It is non-commutative: $C\circ C'\neq C'\circ C$ in general, but it is associative. Since the Weyl group $\sfW$ of $G$ acts on the set of weights of any representation of $G$, it acts on the hyperplane arrangement $\bfH$ and on the set of cells $\FC$. The Tits product is $\sfW$-equivariant: $w\cdot(C\circ C')=(w\cdot C)\circ(w\cdot C')$ for any $C,C'\in\FC$ and $w\in \sfW$ (Proposition~\ref{proposition:equivarianceTits}).

One may generalize the definition of the Coxeter complex to an arbitrary number of representations $V_1,\hdots,V_r$ of $G$ by considering $V=\bigoplus_{i=1}^rV_i$. We will use this in the case of two representations $V,W$ of $G$ in the framework of the 2d induction systems.

\subsubsection{The critical cohomological system}
Let $(G,V,f)$ be a triple of a reductive group $G$, a representation $V$ of $G$ and a $G$-invariant function $f\colon V\rightarrow\BoC$. We let $T'$ be an auxiliary torus acting on $V$, preserving $f$, and whose action commutes with that of $G$. The group $T'$ encodes the deformation parameters. Let $F\in\FF$ be a flat of the Coxeter complex of $(G,V)$. Let $C\in\FC$ be an open cell in $F$ and $\lambda\in\rmX_*(T)$ be a cocharacter of the maximal torus of $G$ that lies in $C$ (considering the natural inclusion $\rmX_*(T)\subset \Fh_{\BoR}$). We let $f_{\lambda}\colon V^{\lambda}\rightarrow\BoC$ be the $G^{\lambda}\times T'$-invariant function induced by $f$. We define $\CH^{T'}_{G,V,f,F}\coloneqq \HO^*_{G^{\lambda}\times T'}(V^{\lambda},\varphi_{f_{\lambda}})$. It does not depend on the choice of $C$ and $\lambda$. The data $(\CH_{G,V,f,F}^{T'})_{F\in\FF}$ is called \emph{the critical cohomological system}. When $f$ is the identically vanishing function, we drop $f$ from the notation and write $\CH_{G,V,F}^{T'}$. If in addition $T'=1$ is the trivial torus, we recover the cohomological induction system studied in \cite{hennecart2024cohomological} when the auxiliary torus is omitted.

\subsubsection{Critical induction morphisms}

Let $C\in\FC$ and $F\in\FF$ be such that $C\preceq F$. There is an induction morphism
\[
 \Ind_{C}^{F}\colon\CH_{G,V,f,\langle C\rangle}^{T'}\rightarrow\CH_{G,V,f,F}^{T'}\,,
\]
\S\ref{subsection:parabolic_induction}. The following proposition describes the associativity in this context.
\begin{proposition}[=Proposition~\ref{proposition:associativity3d}]
 For any cells $C,C'\in\FC$ and flat $F\in\FF$ such that $C\preceq C'\preceq F$, we have
 \[
  \Ind_{C}^{F}=\Ind_{C'}^{F}\circ\Ind_{C}^{\langle C'\rangle}\,.
 \]
\end{proposition}
When the function $f$ is identically $0$, one recovers the associativity of \cite[Proposition~2.9]{hennecart2024cohomological}.

\subsubsection{$2d$ induction morphisms}
We let $G$ be a reductive group and $V,W$ two representations of $G$. We let $\mu\colon V\rightarrow W^*$ be a $G$-equivariant regular map to the linear dual $W^*$ of $W$. We let $T'$ be an auxiliary torus acting on $V,W$ such that the actions of $T'$ and $G$ commute and $\mu$ is $T'$-equivariant. We let $\bfH$ be the hyperplane arrangement in $\Fh_{\BoR}=\rmX_*(T)\otimes_{\BoZ}\BoR$ associated to the set of weights $\CW(\Fg)\cup\CW(V)\cup\CW(W)$. We let $\FC$ be the corresponding set of cells and $\FF$ the set of flats as introduced in \S\ref{subsubsection:Coxeter-Complex}. For $F\in \FF$, we define
\[
 V^F\coloneqq V^{\lambda},\quad W^{F}\coloneqq W^{\lambda},\quad G^{F}\coloneqq G^{\lambda}\,.
\]
where $\lambda\in\rmX_*(T)$ is a cocharacter which belongs to an open cell of $\FC$ in $F$ under the inclusion $\rmX_*(T)\subset\Fh_{\BoR}$. We let $\mu_{\lambda}\colon V^{\lambda}\rightarrow W^{*,\lambda}$ be the $G^{\lambda}$-equivariant function induced by $\mu$. These spaces and the induced map depend only on $F$. We write $\mu_F\colon V^{F}\rightarrow W^{*,F}$ for this $G^{F}\times T'$-equivariant map.

For $C\in\FC$ and $F\in\FF$ such that $C\preceq F$, we choose $\lambda\in\rmX_*(T)\cap C$ and $\nu\in\rmX_*(T)$ a cocharacter in an open cell of $F$ and we define
\[
 V^{C\geq 0,F}\coloneqq V^{\lambda\geq 0}\cap V^{\nu}, \quad W^{C\geq 0,F}\coloneqq W^{\lambda\geq 0}\cap W^{\nu},\quad \mu_{C\geq 0,F}\coloneqq \mu_{\lambda\geq 0,\nu}\colon V^{C\geq 0,F}\rightarrow W^{*,C\geq 0,F}\,.
\]
We similarly define the variants $V^{C\leq 0,F}, V^{C>0,F}$, and $V^{C<0,F}$ and the analogous spaces for $W$ and $W^*$.

We define the cohomological system $(\CH_{G,V,W,\mu,F}^{T'})_{F\in\FF}$ by
\[
 \CH_{G,V,W,\mu,F}^{T'}\coloneqq \HO^{\rmBM}_*(\mu_{F}^{-1}(0)/(G^{F}\times T'),\BoQ)\,.
\]
This is the $G^F\times T'$-equivariant Borel--Moore homology of the zero-locus $\mu_F^{-1}(0)$.
\begin{theorem}[=\S\ref{subsection:parabolicinduction2d}+Proposition~\ref{proposition:associativity2d}]
We let $\mu\colon V\rightarrow W^*$ be a $G\times T'$-equivariant function. For $C\preceq F$, put $L=G^{\langle C\rangle}$, $P=G^{C\geq0,F}$ and $Z_+=\mu_F^{-1}(0)\cap V^{C\geq0,F}$. The Levi-equivariant ambient map
\[
 h_{C,F}\colon V^{C\geq0,F}\longrightarrow V^{\langle C\rangle}\times(W^*)^{C>0,F},
 \qquad v\longmapsto(q_{C,F}(v),\mu_{C>0,F}(v)),
\]
defines a refined pullback on zero loci. Composing it with the inverse smooth pullback for
$r_{C,F}\colon Z_+/(L\times T')\to Z_+/(P\times T')$ and then the proper pushforward
$p_{C,F}\colon Z_+/(P\times T')\to\mu_F^{-1}(0)/(G^F\times T')$ defines induction morphisms
\[
 \Ind_{C}^F\colon\CH_{G,V,W,\mu,\langle C\rangle}^{T'}\rightarrow\CH_{G,V,W,\mu,F}^{T'}
\]
which satisfy the associativity constraint: for any cells $C,C'\in\FC$ and flat $F\in\FF$ such that $C\preceq C'\preceq F$, one has
\[
 \Ind_{C}^F=\Ind_{C'}^F\circ\Ind_{C}^{\langle C'\rangle}\,.
\]

\end{theorem}

\begin{theorem}[Cohomological dimensional reduction and induction = Proposition~\ref{proposition:comparison2d3dmultiplications}]
We let $f\colon V\oplus W\rightarrow \BoC$ be the regular function obtained by contracting $\mu$: $f(v,w)\coloneqq \mu(v)(w)$. It is $G\times T'$-invariant. Then, for $F\in\FF$, the dimensional reduction isomorphisms (see Corollary~\ref{corollary:cohdimrediso})
\[
\CH_{G,V,W,\mu,F}^{T'}\rightarrow\CH_{G,V\oplus W,f,F}^{T'}
\]
commute with the induction morphism $\Ind_{C}^F$ ($C\in\FC$, $F\in\FF$ such that $C\preceq F$) up to the sign $(-1)^{\dim W^{*,C>0,F}}$.
\end{theorem}

A typical $2d$-situation is when $V=\Tan^*V'$ is the cotangent representation of a representation $V'$ of $G$, or more generally a symplectic representation of $G$, and $\mu\colon V\rightarrow\Fg^*$ is the moment map for the $G$-action.

\subsubsection{Restriction morphism}
We define the \emph{localized cohomological system} as $(\tilde{\CH}_{G,V,f,F',F}^{T'})_{F'\preceq F}$, where
\[
 \tilde{\CH}_{G,V,f,F',F}^{T'}\coloneqq
 \CH_{G,V,f,F'}^{T'}[(\Eu_{V,F',F}^{T'})^{-1}],\qquad
 \Eu_{V,F',F}^{T'}\coloneqq
 \prod_{\substack{\alpha\in\rmX^*(T\times T')\\
 \langle\alpha,\lambda\rangle\neq0\\\langle\alpha,\mu\rangle=0}}
 \alpha^{\dim V_\alpha}.
\]
Here $\lambda$ lies in an open cell of $F'$, $\mu$ lies in an open cell of $F$, and $V=\bigoplus_\alpha V_\alpha$ is the $T\times T'$-weight decomposition. For any $C\in\FC$ and $F\in\FF$ such that $C\preceq F$, there is a restriction morphism
\[
 \Res_{F}^C\colon\CH_{G,V,f,F}^{T'}\rightarrow\tilde{\CH}_{G,V,f,\langle C\rangle,F}^{T'}\,.
\]
For any $C\preceq C'\preceq F$, the morphism $\Res_{\langle C'\rangle}^C$ commutes with the multiplication by $\Eu_{V,\langle C'\rangle,F}^{T'}$ and therefore induces a restriction morphism
\[
 \Res_{\langle C'\rangle}^C\colon \tilde{\CH}_{G,V,f,\langle C'\rangle,F}^{T'}\rightarrow\tilde{\CH}_{G,V,f,\langle C\rangle,F}^{T'}\,.
\]
This extension is important for the following proposition, which is the coassociativity for the restriction morphisms, since to write compositions of restriction morphisms, one needs to extend the definition of the restriction to the localized cohomology.

\begin{proposition}[=Proposition~\ref{proposition:coassociativityrestriction}]
 Let $C,C'\in\FC$ and $F\in\FF$ be such that $C\preceq C'\preceq F$. Then, one has
 \[
  \Res_{F}^C=\Res_{\langle C'\rangle}^{C}\circ\Res_{F}^{C'}\,.
 \]
\end{proposition}

\subsubsection{Mackey formula}
The Mackey formula expresses the commutation rule between the induction and restriction functors. In order to write down the formula, we define the braiding morphisms.

Let $C\in\FC$ and $F\in\FF$. We define the kernel
\[
 k_{C,F}\coloneqq \frac{\prod_{\substack{\alpha\in\rmX^*(T\times T')\\\langle\alpha,\lambda\rangle<0\\\langle\alpha,\mu\rangle=0}}\alpha^{\dim V_{\alpha}}}{\prod_{\substack{\alpha\in\rmX^*(T\times T')\\\langle\alpha,\lambda\rangle<0\\\langle\alpha,\mu\rangle=0}}\alpha^{\dim \Fg_{\alpha}}}\in\Frac(\HO^*_{T}(\pt))\otimes_{\BoQ}\HO^*_{T'}(\pt)
\]
for $\lambda\in C$ and $\mu\in F\setminus \bigcup_{\substack{F'\in\FF\\F\not\subset F'}}F'$. When we drop $F$ from the notation, this means that we take $F=\cap_{H\in\FH}H$, the smallest flat for inclusion (and hence the maximal flat for $\preceq$). Note that the denominator belongs to $\HO^*_T(\pt)$ and there is no need to invert $\HO^*_{T'}(\pt)$ as the action of $T'$ on $\Fg$ is trivial.

\begin{definition}[=Definition~\ref{definition:braidingoperators}]
 Let $C,C'\in\FC$ be two cells generating the same flat: $\langle C\rangle=\langle C'\rangle$. Then, the braiding operator is defined as
 \[
 \begin{matrix}
  \tau_{C}^{C'}&\colon& \tilde{\CH}_{G,V,f,\langle C\rangle}^{T'}&\rightarrow&\tilde{\CH}_{G,V,f,\langle C'\rangle}^{T'}\\
  &&g&\mapsto& \frac{k_{C}}{k_{C'}}g\,.
 \end{matrix}
 \]
It is an isomorphism.

For a flat $F\in\FF$ such that $C,C'\preceq F$, we also define
\[
 \begin{matrix}
  \tau_{C,F}^{C',F}&\colon& \tilde{\CH}_{G,V,f,\langle C\rangle,F}^{T'}&\rightarrow&\tilde{\CH}_{G,V,f,\langle C'\rangle,F}^{T'}\\
  &&g&\mapsto& \frac{k_{C,F}}{k_{C',F}}g\,.
 \end{matrix}
\]
It is an isomorphism.
\end{definition}

We let $C,C'\preceq C''$ be three cells of the Coxeter complex of $(G,V)$. We fix a set of representatives $\{\dot{w}\colon w\in \sfW^{\langle C'\rangle}\backslash \sfW^{\langle C''\rangle}/\sfW^{\langle C\rangle}\}\subset \sfW^{\langle C''\rangle}$ of the double cosets $\sfW^{\langle C'\rangle}\backslash \sfW^{\langle C''\rangle}/\sfW^{\langle C\rangle}$. We can now state our main theorem: the Mackey formula.

\begin{theorem}[Mackey formula=Theorem~\ref{theorem:mackeyformula} and Corollary~\ref{corollary:Mackey3cells}]
\label{theorem:mainintro}
Let $(G,V,f)$ be a $G$-equivariant LG-model. Let $C,C'\preceq C''$ be three cells of the Coxeter complex of $(G,V)$. Then, we have an identification of the morphisms
 \[
  \Res_{\langle C''\rangle}^{C'}\circ \Ind_{C}^{\langle C''\rangle}=\sum_{w\in \sfW^{\langle C'\rangle} \backslash \sfW^{\langle C''\rangle}/\sfW^{\langle C\rangle}}\Ind_{C'\circ(\dot{w}\cdot C)}^{\langle C'\rangle}\circ\tau_{(\dot{w}\cdot C)\circ C'}^{C'\circ(\dot{w}\cdot C)}\circ\Res_{\langle \dot{w}\cdot C\rangle}^{(\dot{w}\cdot C)\circ C'}\circ(\dot{w}\cdot-)\,.
 \]
 from $\CH_{G,V,f,\langle C\rangle}\rightarrow \tilde{\CH}_{G,V,f,\langle C'\rangle,\langle C''\rangle}$.
\end{theorem}

Note that the target of $\Res_{\langle C''\rangle}^{C'}$ is $\tilde{\CH}_{G,V,f,\langle C'\rangle,\langle C''\rangle}$. The summand corresponding to $w$ in the r.h.s. of the formula in Theorem~\ref{theorem:mainintro} can be interpreted as an element of $\tilde{\CH}_{G,V,f,\langle C'\rangle,\langle C''\rangle}$. Indeed, $\Res_{\langle \dot{w}\cdot C\rangle}^{(\dot{w}\cdot C)\circ C'}$ requires the localization at $\prod_{\substack{\alpha\in\rmX^*(T\times T')\\\langle\alpha,C'\rangle\neq 0\\\langle\alpha,\dot{w}\cdot C\rangle=0}}\alpha^{\dim V_{\alpha}}$ which divides $\prod_{\substack{\alpha\in\rmX^*(T\times T')\\\langle\alpha,C'\rangle\neq 0\\\langle\alpha,C''\rangle=0}}\alpha^{\dim V_{\alpha}}$ since $\dot{w}\cdot C\preceq C''$ and $\tau_{(\dot{w}\cdot C)\circ C'}^{C'\circ\dot{w}\cdot C}$ has a denominator that divides $\prod_{\substack{\alpha\in\rmX^*(T\times T')\\\langle\alpha,C'\rangle<0\\\langle\alpha,C''\rangle=0}}\alpha^{\dim V_{\alpha}}$.

Implicit in Theorem~\ref{theorem:mainintro} is that the composition
\[
 \Ind_{C'\circ(\dot{w}\cdot C)}^{C'}\circ\tau_{(\dot{w}\cdot C)\circ C'}^{C'\circ(\dot{w}\cdot C)}\circ\Res_{\langle\dot{w}\cdot C\rangle}^{(\dot{w}\cdot C)\circ C'}\circ(\dot{w}\cdot-)
\]
does not depend on the representative $\dot{w}\in \sfW^{C''}$ of $w\in  \sfW^{C'} \backslash \sfW^{C''}/\sfW^{C}$. It is also implicit that the induction $\Ind_{C'\circ (\dot{w}\cdot C)}^{\langle C'\rangle}$ extends canonically to the image of $\tau_{(\dot{w}\cdot C)\circ C'}^{C'\circ(\dot{w}\cdot C)}\circ\Res_{\langle\dot{w}\cdot C\rangle}^{(\dot{w}\cdot C)\circ C'}\circ(\dot{w}\cdot-)$ for the Mackey formula to be well-defined. This is explained in Lemma~\ref{lemma:composition_only_depends_on_class}.

When the potential $f$ on $V$ vanishes, the critical cohomology appearing is simply the singular cohomology of the stack $V/G$. The cohomology of $V/G$ can be identified with the Weyl group invariant polynomial functions on $\Fh=\Lie(T)$, where $T$ is a maximal torus of $G$. This case is very favorable to computations in terms of shuffle formulas. We provide several examples in \S\ref{section:examples}. Computations show that most of the theory of cohomological integrality for $V/G$ with trivial potential only depends on the Weyl group $\sfW$ of $G$ and the set of weights of the representation $V$. Indeed, as we show in \cite{hennecart2025integrality}, one can show a cohomological integrality isomorphism and a Mackey formula in a purely combinatorial situation given by a reflection group $\sfW$ with a reflection representation $\Fh$ and a finite collection of ``weights'' $\CW\subset \Fh^*$ (i.e. a finite $\sfW$-invariant collection of elements of $\Fh^*$).

\subsection{Connection to other works}
The localized coproduct for cohomological Hall algebras of quivers has been defined in \cite{davison2017critical}. Let $Q$ be a quiver and let $V=X_{Q,\dd}$ be the space of $\dd$-dimensional representations of $Q$. It is a representation of the product of general linear groups $G_{\dd}$. In this case, the restriction morphisms we define in the present paper coincide with the coproduct of Davison. More recently, the localized coproduct has been defined by Samuel DeHority and Alexei Latyntsev \cite{dehority2025orthosymplectic} for orthosymplectic cohomological Hall algebras, which in our language would correspond to a certain class of representations of products of orthogonal and symplectic groups. The disconnected orthogonal factors require a separate treatment beyond our connected-group convention.

\subsection{Further directions}
Recent work of Bu, Davison, Ib{\'a}{\~n}ez N{\'u}{\~n}ez, Kinjo and P{\u{a}}durariu \cite{bu2025cohomology} introduces a global way to think of the induction morphisms for the cohomology of smooth stacks, the critical cohomology of $(-1)$-shifted symplectic stacks and the Borel--Moore homology of $0$-shifted symplectic stacks. It would be a very interesting direction to define restriction morphisms for this class of stacks and write down explicitly the Mackey formula ruling the commutation between the induction and restriction morphisms in that situation. The present paper, by considering the local case of quotient stacks of the form $V/G$ for $V$ a representation of a reductive group $G$, should give insight into how to generalize the Mackey formula to global stacks, using the component lattice of stacks as defined by Bu, Halpern-Leistner, Ib{\'a}{\~n}ez N{\'u}{\~n}ez and Kinjo in \cite{bu2025intrinsic}, see also \cite{bu2025cohomology} for the definitions.

The localized coproduct for cohomological Hall algebras is used in a crucial way in the proof of Okounkov's conjecture \cite{botta2023okounkov}, also proved in \cite{schiffmann2023cohomological}, to identify the BPS cohomology. We expect the same kind of phenomenon for more general stacks of the form $\Tan^*V/G$ where $V$ is a representation of a reductive group $G$. Moreover, the Mackey formula should play a role in showing that a certain induction procedure preserves the BPS cohomology and, more generally, in studying the BPS cohomology of stacks defined in \cite{bu2025cohomology}.

\subsection{Organization of the paper}
In \S\ref{section:equcoh-mhm}, we recall the basics regarding monodromic mixed Hodge modules and vanishing cycle sheaf functors. In \S\ref{section:Coxetercomplex}, we present the Coxeter complex and some useful operations (Weyl group action and Tits product). In \S\ref{section:inductionsystem}, we define the induction morphisms in critical equivariant cohomology in two different ways, directly and via the torus-equivariant cohomology. In \S\ref{section:2dinductionsystem}, we define an alternative induction system associated to a $G$-equivariant map between two representations of $G$. This alternative situation is called the ``$2d$-situation''. We relate it to critical induction systems via dimensional reduction. \S\ref{section:restrictionmorphism} is devoted to the definition of the localized restriction morphisms on the critical cohomological system. There are again various approaches to this definition, one using the torus-equivariant cohomology and one direct definition. The heart of this paper lies in \S\ref{section:cohomologicalMackeyformula}, where a proof of the cohomological Mackey formula is given. The proof localizes the flag variety with respect to the central cocharacter of the target Levi, computes each fixed-component contribution by critical excess intersection, and descends using injectivity of the moving-root Euler class. In \S\ref{section:comparisoninductionsystems}, we compare the induction systems one obtains on a representation $V$ for various $G$-invariant functions on $V$. For cotangent representations, we also study the comparison from the moment-map $2d$-induction system to the shuffle system and the necessary relations satisfied by its image. Finally, \S\ref{section:examples} illustrates the Mackey formula by explicit calculations for vanishing potential and by a quadratic-potential example that also checks the dimensional-reduction sign.

\subsection{Acknowledgements}
I would like to thank Andr{\'e}s Ib{\'a}{\~n}ez N{\'u}{\~n}ez and Chenjing Bu for discussions related to the extension to general stacks of the Mackey formula presented here, and also for pointing out the reference \cite{dehority2025orthosymplectic} in which the localized restriction for orthosymplectic cohomological Hall algebras is defined. I would also like to thank IPMU for the excellent working conditions during a visit in March 2025. I would like to thank Olivier Schiffmann for discussions related to the Mackey formula in Lyon in November 2024. It will be clear to the reader that this paper draws much of its inspiration from \cite{davison2017critical} and I would like to thank Ben Davison for his support for many years and the exceptional working conditions in Edinburgh between 2021 and 2024 as a postdoctoral researcher.

\subsection{Conventions and Notations}
\label{subsection:conventions}
\begin{enumerate}
 \item All algebraic groups and varieties are defined over $\BoC$. Unless stated otherwise, reductive groups are connected, representations are finite-dimensional, and functions are regular. Whenever an auxiliary torus $T'$ is present in critical cohomology, the potential is assumed to be $T'$-invariant.
 \item We denote by $\HO^*_{\BoC^*}=\HO^*_{\BoC^*}(\pt,\BoQ)$ the $\BoC^*$-equivariant cohomology of the point.
 \item When $X$ is a $G$-variety and $f\colon X\rightarrow\BoC$, we only use the perverse $t$-exact shifted vanishing cycle functor $\varphi_f$.
 \item The letters $V,W$ will usually denote representations of a reductive group $G$.
 \item Given a reductive group $G$ and a maximal torus $T\subset G$, we denote by $\sfW$ the Weyl group, defined by $\sfW\coloneqq \rmN_G(T)/T$.
 \item All cohomology spaces are considered with rational coefficients and come with natural mixed Hodge structures. In polynomial descriptions such as $\HO^*_T(\pt)\cong\Sym(\Fh^*)$, the symmetric algebra is taken on the rational form $\rmX^*(T)\otimes_{\BoZ}\BoQ$ of $\Fh^*$; after extension of scalars to $\BoC$, it is the usual algebra of polynomial functions on $\Fh$. When taking derived global sections, we make the abuse of hiding the Tate twists that appear, for example in the induction morphisms, to lighten the notation. When working at the level of constructible sheaves, this does not cause any harm. To understand the morphisms at the level of mixed Hodge structures, there are shifts to incorporate. We illustrate this on the induction morphisms, hoping this will allow the reader to work out the various Tate twists, in \S\ref{subsection:inductionmhs}.
\end{enumerate}

\section{Equivariant cohomology and monodromic mixed Hodge modules}
\label{section:equcoh-mhm}

\subsection{Cohomology of flag varieties}
\label{subsection:cohflagvar}
We recall the following classical proposition regarding the cohomology of flag varieties; see \cite{borel1953cohomologie}.
\begin{proposition}
 \label{proposition:cohflagvarieties}
 Let $G$ be a reductive algebraic group, $B$ a Borel subgroup and $T\subset B$ a maximal torus. We let $\sfW\coloneqq N_G(T)/T$ be the Weyl group of $G$. Then, there are actions of $\sfW$ on $\HO^*(G/B,\BoQ)$ and $\HO^*(G/T,\BoQ)$ compatible with the isomorphism $\HO^*(G/B,\BoQ)\rightarrow\HO^*(G/T,\BoQ)$ given by the pullback map and
 \[
  \HO^*(G/T,\BoQ)^{\sfW}\cong \BoQ\cong \HO^*(G/B,\BoQ)^{\sfW}\,.
 \]
\end{proposition}
We will not use the following in this paper, but the cohomology $\HO^*(G/B,\BoQ)$ with the action of $\sfW$ is the regular representation of $\sfW$ (see \cite[Lemma~27.1]{borel1953cohomologie}), which generalizes Proposition~\ref{proposition:cohflagvarieties}. Note that the morphism $G/T\rightarrow G/B$ is a locally trivial fibration with affine-space fibers, and so the pullback $\HO^*(G/B)\rightarrow\HO^*(G/T)$ is an isomorphism. Moreover, $\sfW$ acts on $G/T$ and so on $\HO^*(G/T)$ and also on $\HO^*(G/B)$, but it does not act directly on the flag variety $G/B$.

\subsection{Equivariant cohomology and homology}
\label{subsection:equcoh}
We recall the comparison between equivariant Borel--Moore homology for a reductive group $G$ and its maximal torus $T$; compare \cite[Proposition~6]{edidin1998equivariant} for the corresponding statement in equivariant intersection theory.
\begin{proposition}
 \label{proposition:weylgroupBorelMoore}
Let $G$ be a reductive algebraic group acting on a complex algebraic variety $X$. Then, $\HO^{\rmBM}_{*,G}(X,\BoQ)\cong\HO^{\rmBM}_{*,T}(X,\BoQ)^{\sfW}$, and this is an isomorphism of mixed Hodge structures (when appropriate Tate twists are applied).
\end{proposition}
\begin{proof}[Sketch of proof]
The morphism of stacks $p\colon X/T\rightarrow X/G$ is a $G/T$ fibration and in particular is smooth. Therefore, the dualizing sheaf $\BD\BoQ_{X/G}$ is constant along the fibers of the projection, whose cohomology has $\sfW$-invariant part $\BoQ$ by Proposition~\ref{proposition:cohflagvarieties}.
\end{proof}
The ``appropriate Tate twists'' of Proposition~\ref{proposition:weylgroupBorelMoore} can be deduced from the isomorphism $p^*\BD\BoQ_{X/G}\cong \BD\BoQ_{X/T}\otimes\SL^{\dim G/T}$ where the Tate twist is recalled in \S\ref{subsection:Tatetwist}.

\subsection{Monodromic mixed Hodge modules}
\label{subsection:monmhm}
In this paper, we work with (complexes of) monodromic mixed Hodge modules, which one might understand, as a first approximation, as an enhancement of constructible sheaves on a complex algebraic variety. We refer to \cite{davison2020cohomological} for more details regarding the formalism of (equivariant) monodromic mixed Hodge modules that we use.

Equivariant operations are interpreted on finite-dimensional approximations to the relevant quotient stacks, in every fixed cohomological range. We allow the resulting locally bounded-below complexes and their degreewise stabilized cohomology: equivariant cohomology need not be bounded above. The notation for derived categories below includes these degreewise extensions when a quotient-stack pushforward requires them. Our graded cohomology modules are direct sums of their homogeneous pieces, and localizations are algebraic graded localizations. We do not replace these modules by completions or use infinite Laurent expansions. Dimension twists from smooth approximations and changes of group are included before passing to the stable range. Throughout, $\BD$ denotes absolute Verdier duality on the relevant variety or quotient stack.

Let $X$ be a complex algebraic variety and $G$ a linear algebraic group acting on $X$. We denote by $\MHM_G(X)$ the Abelian category of $G$-equivariant mixed Hodge modules on $X$. We denote by $\MMHM_G(X)$ the Abelian category of $G$-equivariant monodromic mixed Hodge modules on $X$. We let $\CD_G(\MHM(X))$ be the derived category of $G$-equivariant mixed Hodge modules and $\CD_G(\MMHM(X))$ the derived category of $G$-equivariant monodromic mixed Hodge modules, which may be constructed as in \cite{achar2013equivariant}. We denote by $\CD_{\rmc,G}(X,\BoQ)$ the category of $G$-equivariant constructible complexes of $\BoQ$-vector spaces on $X$ and $\Perv_G(X)$ the Abelian category of $G$-equivariant perverse sheaves. There is a triangulated realization functor
\[
 \rat\colon \CD_G(\MMHM(X))\rightarrow\CD_{\rmc,G}(X,\BoQ)
\]
which restricts to a functor $\rat\colon\MMHM_G(X)\rightarrow\Perv_G(X)$ between Abelian categories.

\subsection{Vanishing cycles functors}
\label{subsection:vanishingcycle}
Given a $G$-invariant regular function $f\colon X\rightarrow \BoC$ on an algebraic variety $X$, there is an exact \emph{vanishing cycle} functor $\varphi_f\colon\CD_G(\MHM(X))\rightarrow\CD_G(\MMHM(X))$. We let $\imath\colon f^{-1}(0)\rightarrow X$ be the inclusion of the zero locus of $f$. There is a distinguished triangle
\begin{equation}
\label{equation:distinguishedtrianglevanishingnearby}
 \varphi_f\rightarrow \imath_*\imath^*\rightarrow\psi_f\rightarrow
\end{equation}
where $\psi_f$ is the unshifted \emph{nearby cycle functor}. We emphasize the fact that the vanishing cycle functor is considered with the shift $[-1]$ relative to the unshifted vanishing cycle functor, so that it sends mixed Hodge modules to monodromic mixed Hodge modules, and perverse sheaves to perverse sheaves.

\subsection{Functoriality of vanishing cycles}
\label{subsection:vanishingcyclesfunctoriality}
Let $h\colon\FX\rightarrow\FY$ be a morphism between algebraic stacks and $f\colon\FY\rightarrow\BoA^1$ a regular function. For the purposes of this paper, one may assume that $h$ is a morphism between quotient stacks $h\colon X/G\rightarrow Y/H$ coming from a $(G,H)$-equivariant morphism $X\rightarrow Y$. Then, there are natural transformations $\varphi_f\circ h_*\rightarrow h_*\varphi_{f\circ h}$, $h_!\varphi_{f\circ h}\rightarrow \varphi_{f}h_!$ and their duals involving pullbacks instead of pushforwards. The composition $\varphi_{f}\rightarrow \varphi_{f}h_*h^*\rightarrow h_*\varphi_{f\circ h} h^*$ gives, when applied to the constant sheaf and after taking derived global sections, the pullback morphism in vanishing cycle cohomology
\begin{equation}
\label{equation:pbvanishingcyclecoh}
 \HO^*(\FY,\varphi_f\BoQ_{\FY})\rightarrow\HO^*(\FX,\varphi_{f\circ h}\BoQ_{\FX})\,.
\end{equation}
The proper pushforward in vanishing cycle cohomology is defined for $h$ proper and representable by taking the derived global sections of the morphism of complexes
\[
 h_!\varphi_{f\circ h}h^!\BD\BoQ_{\FY}\rightarrow\varphi_fh_!h^!\BD\BoQ_{\FY}\rightarrow\varphi_f\BD\BoQ_{\FY}
\]
which gives a morphism
\[
 h_*\varphi_{f\circ h}\BD\BoQ_{\FX}\rightarrow \varphi_f\BD\BoQ_\FY
\]
by properness of $h$, and so $\HO^*(\FX,\varphi_{f\circ h}\BD\BoQ_{\FX})\rightarrow\HO^*(\FY,\varphi_f\BD\BoQ_{\FY})$, the pushforward in vanishing cycle cohomology.

\begin{lemma}[Exchange and the zero-fiber morphisms]
\label{lemma:exchange-zero-fiber}
Let $h\colon X\to Y$ be a morphism, let $f\colon Y\to\BoA^1$ be a regular function, and let $u\colon Z\hookrightarrow Y$ be a closed immersion factoring through $f^{-1}(0)$. Set $Z'=X\times_Y Z$, and denote its inclusion in $X$ by $u'$. We consider the morphisms
\[
 \epsilon_Y\colon\varphi_f\longrightarrow u_*u^*,
 \qquad \eta_Y\colon u_*u^!\longrightarrow\varphi_f,
\]
obtained from the zero-fiber triangle and its Verdier dual, and similarly $\epsilon_X,\eta_X$ for $f\circ h$. For the canonical exchange transformations
\[
 a\colon\varphi_f h_*\longrightarrow h_*\varphi_{f\circ h},
 \qquad b\colon u_*u^!h_*\xrightarrow{\sim}h_*u'_*u'^!,
\]
one has
\begin{equation}
\label{equation:eta-exchange}
 a\circ\eta_Yh_*=h_*\eta_X\circ b.
\end{equation}
The analogous compatibility for $\epsilon$, ordinary restriction to the zero fiber and the exchange transformation $a$ also holds. These identities are compatible with composition of morphisms, units and counits of adjunction, and smooth equivariant descent.
\end{lemma}
\begin{proof}
Write $i\colon Y_0=f^{-1}(0)\hookrightarrow Y$ and $i'\colon X_0\hookrightarrow X$. We use the canonical functorial triangles defining nearby and vanishing cycles. The exchange morphism $h_!\varphi_{f\circ h}\to\varphi_fh_!$ is the morphism on the vanishing-cycle terms induced by the exchange morphisms on the other two terms of
\[
 \varphi_f\longrightarrow i_*i^*\longrightarrow\psi_f\longrightarrow.
\]
Here the zero-fiber exchange is the base-change isomorphism
$h_!i'_*i'^*\cong i_*i^*h_!$. The nearby-cycle exchange is constructed from base change and adjunction on the punctured fibers (or their universal covers). It intertwines the specialization morphisms $i^*\to\psi_f$ by naturality of the unit of adjunction. Thus the exchange morphisms form a morphism of the canonical triangles. This proves the compatibility with $\varphi\to i_*i^*$ for $h_!$. Applying Verdier duality proves the compatibility of $i_*i^!\to\varphi$ with $a$ for $h_*$.

For a general $Z\subset Y_0$, write $u=i\circ v$. The map $u_*u^!\to i_*i^!$ is obtained by applying $i_*(-)i^!$ to the counit $v_*v^!\to\id$. Exceptional base change commutes with this counit: equivalently, take right adjoints of the ordinary proper-base-change identity for the closed immersion $v$. Composing with the preceding zero-fiber identity proves \eqref{equation:eta-exchange}. The ordinary-restriction identity follows in the same way from the $h_*$ exchange morphism of the nearby-cycle triangles and the unit $\id\to v_*v^*$.

All the exchanges used here are the canonical mates of base-change morphisms. Their compatibility with composition follows from the composition rule for mates, and their compatibility with units and counits follows from the triangular identities of adjunction. The argument applies to equivariant complexes on smooth presentations, and hence to the quotient stacks used below. For monodromic mixed Hodge modules, use the same functorial constructions and their monodromic versions.
\end{proof}

\subsection{Tate twist}
\label{subsection:Tatetwist}
We let $\SL=\BoQ(-1)[-2]\in\CD(\MHM(\pt))$ be the Tate complex. It is given by the mixed Hodge structure $\HO^*_{\rmc}(\BoA^1,\BoQ)$. It is a pure weight zero complex of mixed Hodge structures. It is concentrated in cohomological degree $2$. Working with monodromic mixed Hodge modules gives a square-root of $\SL$, given by $\SL^{1/2}\coloneqq\HO^*_{\rmc}(\BoA^1,\varphi_{x^2}\BoQ_{\BoA^1})$, where $x^2\colon\BoA^1\rightarrow\BoA^1$ is the square function. It is pure, concentrated in cohomological degree $1$. We refer to \cite{davison2020cohomological} for more details.

\subsection{Virtual pullback in Borel--Moore homology}
\label{subsection:virtualpullbackBM}
We recall a particular instance of virtual pullback in Borel--Moore homology that we will use later to define the $2d$-induction system in \S\ref{section:2dinductionsystem}.

Let $g\colon X\rightarrow Y$ be a morphism between smooth algebraic varieties, and $f\colon T\rightarrow Y$ a morphism. We form the pullback square
\[
 \begin{tikzcd}
  {U}&{T}\\
  {X}&{Y}
  \arrow["g'",from=1-1, to=1-2]
  \arrow[from=1-1, to=2-1]
  \arrow["g",from=2-1, to=2-2]
  \arrow["f",from=1-2, to=2-2]
  \arrow["\ulcorner"{anchor=center, pos=0.125, rotate=180}, draw=none,from=1-1, to=2-2]
 \end{tikzcd}
\]
The pullback in cohomology $\HO^*(Y,\BoQ)\rightarrow\HO^*(X,\BoQ)$ is induced by the adjunction map
\[
 \BoQ_Y\rightarrow g_*\BoQ_X\,.
\]
By applying $f^!$ and base-change, we obtain the morphism of complexes
\begin{equation}
\label{equation:virtualpullback}
 \BD\BoQ_{T}\otimes\SL^{\dim Y}\rightarrow (g')_*\BD\BoQ_U\otimes\SL^{\dim X}
\end{equation}

\begin{definition}
\label{definition:virtualpullbackBMhomology}
The virtual pullback $g^!\colon\HO^{\rmBM}(T,\BoQ)\rightarrow\HO^{\rmBM}(U,\BoQ)$ is defined as the derived global sections functor applied to the morphism of complexes \eqref{equation:virtualpullback}. It shifts cohomological degrees by twice the codimension of $g$.
\end{definition}
The virtual pullback depends on the ambient smooth schemes $X,Y$. It does not respect the homological degrees due to the presence of Tate twists.

\begin{proposition}[Comparison of virtual pullbacks -- no excess intersection bundle]
Let
\begin{equation}
\label{equation:comparisonvirtualpullbacks}
  \begin{tikzcd}
  {U}&{T}\\
  {X}&{Y}\\
  {\tilde{X}}&{\tilde{Y}}
  \arrow["g'",from=1-1, to=1-2]
  \arrow[from=1-1, to=2-1]
  \arrow["g",from=2-1, to=2-2]
  \arrow["f",from=1-2, to=2-2]
  \arrow["\ulcorner"{anchor=center, pos=0.125, rotate=180}, draw=none,from=1-1, to=2-2]
  \arrow["\tilde{g}",from = 3-1, to = 3-2]
  \arrow[from = 2-1, to=3-1]
  \arrow["f'",from = 2-2, to = 3-2]
 \end{tikzcd}
\end{equation}
be a commutative diagram with Cartesian squares, with $X,Y,\tilde X,\tilde Y$ smooth, such that $g$ and $\tilde{g}$ have the same codimension, i.e. $\dim Y-\dim X=\dim \tilde{Y}-\dim \tilde{X}$ (we say that the bottom square of \eqref{equation:comparisonvirtualpullbacks} \emph {has no excess intersection bundle}). Then, the virtual pullback morphisms \eqref{equation:virtualpullback}
\[
 \BD\BoQ_T\rightarrow (g')_*\BD\BoQ_U\otimes\SL^{\dim X-\dim Y}
\]
obtained from $g$ and $\tilde{g}$ via Definition~\ref{definition:virtualpullbackBMhomology} coincide.
\end{proposition}
\begin{proof}
 It suffices to show that $f'^!$ applied to the pullback morphism $\BoQ_{\tilde{Y}}\rightarrow \tilde{g}_*\BoQ_{\tilde{X}}$ gives the pullback morphism $\BoQ_{Y}\rightarrow g_*\BoQ_X$ up to a Tate twist. This follows by base-change and the fact that $X,Y,\tilde{X},\tilde{Y}$ are smooth and so their dualizing complexes are Tate twists of the constant sheaves.
\end{proof}

\subsection{Cohomological dimensional reduction isomorphism}
\label{subsection:dimrediso}

\subsubsection{Dimensional reduction}
We recall the dimensional reduction isomorphism in its classical version, which relates the vanishing cycles sheaf for a given function to the dualizing sheaf of another algebraic variety. We refer to \cite[Appendix~A]{davison2017critical} for the cohomological dimensional reduction isomorphism we recall here. A deformed version of dimensional reduction has been given in \cite{davison2022deformed} while a global version of dimensional reduction for shifted cotangent bundles is proven in \cite{kinjo2022dimensional}. The closely related dimensional reduction for rapid decay cohomology is introduced in \cite{kontsevich2011cohomological}.

\begin{theorem}[{\cite[Theorem~A.1 and Corollary~A.9]{davison2017critical}}]
\label{theorem:cohdimrediso}
Let $X$ be an algebraic variety, $\pi\colon V\rightarrow X$ (the total space of) a vector bundle on $X$, $s\colon X\rightarrow V^{\vee}$ a section of the dual vector bundle $\pi^{\vee}\colon V^{\vee}\rightarrow X$ and $f\colon V\rightarrow\BoC$ the regular function on $V$ obtained by contraction with the section $s$, $f(v)=s(\pi(v))(v)$. We let $Z\coloneqq s^{-1}(0)\subset X$,  $\overline{Z}\coloneqq \pi^{-1}(Z)\subset V$. We let $V_0\coloneqq f^{-1}(0)\subset V$ and $\imath\colon V_0\rightarrow V$ be the natural closed immersion. We have $\overline{Z}\subset V_0$ and we denote by $\imath'\colon \overline{Z}\rightarrow V_0$ the closed immersion. We let $\overline{\imath}\coloneqq \imath\circ\imath'$. There is a morphism of functors obtained from the adjunction $(\overline{\imath}^*,\overline{\imath}_*)$.
\[
 \id\rightarrow \overline{\imath}_*\overline{\imath}^*\,.
\]
Then, the morphism of functors
\[
 \pi_!\varphi_f(\id\rightarrow\overline{\imath}_*\overline{\imath}^*)\pi^*
\]
is an isomorphism.
\end{theorem}
Assuming that $X$ is smooth, applying Theorem~\ref{theorem:cohdimrediso} to the constant sheaf and taking Verdier duals, one gets the following.
\begin{corollary}
\label{corollary:cohdimrediso}
Assume that $X$ is smooth, and let $j\colon Z\hookrightarrow X$ be the inclusion. There is an isomorphism of complexes of monodromic mixed Hodge modules on $X$
\[
 j_*\BD(\BoQ_{Z}^{\vir})\xrightarrow{\mathsf{dr}}\pi_*\varphi_f\BoQ_{V}^{\vir}\,,
\]
where $\BoQ_{V}^{\vir}\coloneqq \BoQ\otimes\SL^{-\dim V/2}$ and $\BoQ_{Z}^{\vir}\coloneqq\BoQ_{Z}\otimes\SL^{-\frac{\dim X-\rank V}{2}}$ are the constant sheaves shifted by the virtual dimension. In particular, by taking derived global sections, we obtain an isomorphism
\[
 \mathsf{dr}\colon\HO^{\rmBM}_*(Z,\BoQ^{\vir})\rightarrow\HO^*(V,\varphi_f\BoQ^{\vir})
\]
preserving cohomological degrees. This is the \emph{dimensional reduction isomorphism}.
\end{corollary}

\subsubsection{Setup for the functoriality of dimensional reduction}
\label{subsubsection:setupfunctorialitydimred}
We will explain in \S\S\ref{subsubsection:dimredproperpushforward}, \ref{subsubsection:dimensionalreduction-virtualpullback} how the dimensional reduction interacts with pullback/proper pushforward in vanishing cycle cohomology and virtual pullback/proper pushforward in Borel--Moore homology. We could not locate these functorialities explicitly in the existing literature. First, we introduce the setup. Let $X,Y$ be smooth algebraic varieties and $g\colon X\rightarrow Y$ a morphism of algebraic varieties, $V$ a vector bundle over $Y$ and $g^*V$ the pullback to $X$. We let $\overline{Y}$ be the total space of $V$ and $\pi\colon \overline{Y}\rightarrow Y$ the projection, $\overline{X}$ the total space of $g^*V$ with $\pi\colon\overline{X}\rightarrow X$ the projection, and $\overline{g}\colon\overline{X}\rightarrow\overline{Y}$ the induced morphism. We let $V^*$ be the dual vector bundle on $Y$, $g^*V^*\cong (g^*V)^*$, and $\overline{X}^*$, $\overline{Y}^*$ the corresponding total spaces. We let $s\colon Y\rightarrow \overline{Y}^*$ be a section of $V^*$, $g^*s\colon X\rightarrow\overline{X}^*$ the corresponding section of $g^*V^*$. We let $f\colon \overline{Y}\rightarrow \BoC$ be the regular function on $\overline{Y}$ obtained by contraction with $s$. Then, $f\circ \overline{g}$ is the regular function on $\overline{X}$ obtained by contraction with $g^*s$. We define $Z_s\coloneqq s^{-1}(0)\subset Y$, $Z_{g^*s}\coloneqq (g^*s)^{-1}(0)\subset X$, $\overline{X}_0\coloneqq (f\circ \overline{g})^{-1}(0)$, $\overline{Y}_0\coloneqq f^{-1}(0)$. We let $\overline{Z}_s\coloneqq \pi^{-1}(Z_s)$ and $\overline{Z}_{g^*s}\coloneqq \pi^{-1}(Z_{g^*s})$. We may organize these algebraic varieties and maps in the following diagram with Cartesian squares indicated by the corresponding symbol:
% https://q.uiver.app/#q=WzAsOCxbMCwxLCJcXG92ZXJsaW5le1h9Il0sWzEsMSwiXFxvdmVybGluZXtZfSJdLFswLDIsIlxcb3ZlcmxpbmV7WH1fMCJdLFsxLDIsIlxcb3ZlcmxpbmV7WX1fMCJdLFswLDMsIlxcb3ZlcmxpbmV7Wn1fe2deKnN9Il0sWzEsMywiXFxvdmVybGluZXtafV9zIl0sWzAsMCwiWCJdLFsxLDAsIlkiXSxbMCwxLCJcXG92ZXJsaW5le2d9Il0sWzIsM10sWzIsMCwiXFxpbWF0aCJdLFszLDEsIlxcaW1hdGgiLDJdLFs0LDVdLFs0LDIsIlxcaW1hdGgnIl0sWzUsMywiXFxpbWF0aCciLDJdLFs2LDcsImciXSxbNiwwLCJnXipzIl0sWzcsMSwicyIsMl0sWzEsNywiXFxwaSIsMix7ImN1cnZlIjoyfV0sWzUsMSwiXFxvdmVybGluZXtcXGltYXRofSIsMix7ImN1cnZlIjozfV0sWzQsMCwiXFxvdmVybGluZXtcXGltYXRofSIsMCx7ImN1cnZlIjotM31dXQ==
\[\begin{tikzcd}
	X & Y \\
	{\overline{X}} & {\overline{Y}} \\
	{\overline{X}_0} & {\overline{Y}_0} \\
	{\overline{Z}_{g^*s}} & {\overline{Z}_s}
	\arrow["g", from=1-1, to=1-2]
	\arrow["{\overline{g}}", from=2-1, to=2-2]
	\arrow["\pi"', bend right=30, from=2-2, to=1-2]
	\arrow["\imath", from=3-1, to=2-1]
	\arrow["\overline{g}_0",from=3-1, to=3-2]
	\arrow["\imath"', from=3-2, to=2-2]
	\arrow["{\overline{\imath}}", bend left=40, from=4-1, to=2-1]
	\arrow["{\imath'}", from=4-1, to=3-1]
	\arrow["\overline{g}_{s}",from=4-1, to=4-2]
	\arrow["{\overline{\imath}}"', bend right=40, from=4-2, to=2-2]
	\arrow["{\imath'}"', from=4-2, to=3-2]
	\arrow["\pi",bend left=30, from=2-1, to=1-1]
	\arrow["\ulcorner"{anchor=center, pos=0.125, rotate=-90},
	draw=none, from=2-1, to=1-2]
\end{tikzcd}\]
where $\imath, \imath', \overline{\imath}$ are the natural closed immersions. Note that several distinct maps are denoted using the same letters and the context determines which one to actually consider.

\subsubsection{Dimensional reduction and proper pushforward}
\label{subsubsection:dimredproperpushforward}
\begin{proposition}
\label{proposition:dimredpropermorphism}
 We assume that $g$ is proper. Then, $\overline{g}$ and $\overline{g}_s$ are proper. Moreover, the pushforward morphism
 \begin{equation}
 \label{equation:pfproperfunctorialitydr}
  \overline{g}_*\BD\BoQ_{\overline{X}}\rightarrow\BD\BoQ_{\overline{Y}}
 \end{equation}
induces by functoriality the commutative diagram
\begin{equation}
\label{equation:functorialitydrproper}
 % https://q.uiver.app/#q=WzAsNixbMCwxLCJcXGltYXRoXyFcXGltYXRoXiFcXG92ZXJsaW5le2d9XypcXEJEXFxCb1Ffe1xcb3ZlcmxpbmV7WH19Il0sWzAsMCwiXFx2YXJwaGlfZlxcb3ZlcmxpbmV7Z31fKlxcQkRcXEJvUV97XFxvdmVybGluZXtYfX0iXSxbMSwwLCJcXHZhcnBoaV9mXFxCRFxcQm9RX3tcXG92ZXJsaW5le1l9fSJdLFsxLDEsIlxcaW1hdGhfIVxcaW1hdGheIVxcQkRcXEJvUV97XFxvdmVybGluZXtZfX0iXSxbMCwyLCJcXG92ZXJsaW5le1xcaW1hdGh9XyFcXG92ZXJsaW5le1xcaW1hdGh9XiFcXG92ZXJsaW5le2d9XypcXEJEXFxCb1Ffe1xcb3ZlcmxpbmV7WH19XFxjb25nIFxcb3ZlcmxpbmV7XFxpbWF0aH1fKihcXG92ZXJsaW5le2d9X3MpXypcXEJEXFxCb1Ffe1xcb3ZlcmxpbmV7Wn1fe2deKnN9fSJdLFsxLDIsIlxcb3ZlcmxpbmV7XFxpbWF0aH1fIVxcb3ZlcmxpbmV7XFxpbWF0aH1eIVxcQkRcXEJvUV97XFxvdmVybGluZXtZfX1cXGNvbmcgXFxvdmVybGluZXtcXGltYXRofV8qXFxCRFxcQm9RX3tcXG92ZXJsaW5le1p9X3N9Il0sWzEsMl0sWzAsM10sWzQsNV0sWzAsMV0sWzMsMl0sWzUsM10sWzQsMF1d
\begin{tikzcd}
	{\varphi_f\overline{g}_*\BD\BoQ_{\overline{X}}\cong \overline{g}_*\varphi_{f\circ\overline{g}}\BD\BoQ_{\overline{X}}} & {\varphi_f\BD\BoQ_{\overline{Y}}} \\
	{\imath_!\imath^!\overline{g}_*\BD\BoQ_{\overline{X}}} & {\imath_!\imath^!\BD\BoQ_{\overline{Y}}} \\
	{\overline{\imath}_!\overline{\imath}^!\overline{g}_*\BD\BoQ_{\overline{X}}\cong \overline{\imath}_*(\overline{g}_s)_*\BD\BoQ_{\overline{Z}_{g^*s}}\cong \overline{g}_*\overline{\imath}_!\overline{\imath}^!\BD\BoQ_{\overline{X}}} & {\overline{\imath}_!\overline{\imath}^!\BD\BoQ_{\overline{Y}}\cong \overline{\imath}_*\BD\BoQ_{\overline{Z}_s}}
	\arrow[from=1-1, to=1-2]
	\arrow[from=2-1, to=1-1]
	\arrow[from=2-1, to=2-2]
	\arrow[from=2-2, to=1-2]
	\arrow[from=3-1, to=2-1]
	\arrow[from=3-1, to=3-2]
	\arrow[from=3-2, to=2-2]
\end{tikzcd}
\end{equation}
Then, by taking derived global sections, we obtain the commutative diagram where the vertical maps are the dimensional reduction isomorphisms, written with dualizing coefficients as in Corollary~\ref{corollary:cohdimrediso}:
\[
 % https://q.uiver.app/#q=WzAsNCxbMCwwLCJcXEhPXiooXFxvdmVybGluZXtYfSxcXHZhcnBoaV97Z1xcY2lyYyBmfSkiXSxbMSwwLCJcXEhPXiooXFxvdmVybGluZXtZfSxcXHZhcnBoaV9mKSJdLFswLDEsIlxcSE9ee1xccm1CTX1fKihcXG92ZXJsaW5le1p9X3tnXipzfSxcXEJvUSkiXSxbMSwxLCJcXEhPXntcXHJtQk19XyooXFxvdmVybGluZXtafV97c30sXFxCb1EpIl0sWzAsMSwiXFxvdmVybGluZXtnfV8qIl0sWzIsMCwiXFxtYXRoc2Z7ZHJ9Il0sWzIsMywiKFxcb3ZlcmxpbmV7Z31fcylfKiIsMl0sWzMsMSwiXFxtYXRoc2Z7ZHJ9IiwyXV0=
\begin{tikzcd}
	{\HO^*(\overline{X},\varphi_{f\circ\overline{g}}\BD\BoQ_{\overline{X}})} & {\HO^*(\overline{Y},\varphi_f\BD\BoQ_{\overline{Y}})} \\
	{\HO^{\rmBM}_*(\overline{Z}_{g^*s},\BoQ)} & {\HO^{\rmBM}_*(\overline{Z}_{s},\BoQ)}
	\arrow["{\overline{g}_*}", from=1-1, to=1-2]
	\arrow["{\mathsf{dr}}", from=2-1, to=1-1]
	\arrow["{(\overline{g}_s)_*}"', from=2-1, to=2-2]
	\arrow["{\mathsf{dr}}"', from=2-2, to=1-2]
\end{tikzcd}
\]

\end{proposition}
\begin{proof}
The proper trace \eqref{equation:pfproperfunctorialitydr} is the counit for the dualizing complexes. Naturality of $\overline\imath_!\overline\imath^!\to\imath_!\imath^!\to\varphi_f$ applied to that trace gives the squares in the target space. Lemma~\ref{lemma:exchange-zero-fiber} identifies their source arrows with the corresponding arrows on $\overline X$ through proper base change. Properness makes the vanishing-cycle exchange an isomorphism, and compatibility with the counit identifies the bottom arrow with $(\overline g_s)_*$ on dualizing complexes. This proves \eqref{equation:functorialitydrproper}. Applying the bundle projections and global sections, and using dimensional reduction, gives the asserted square. The same argument is compatible with the Thom identifications on $Z_s$ and $Z_{g^*s}$.
\end{proof}

\subsubsection{Dimensional reduction and virtual pullbacks}
\label{subsubsection:dimensionalreduction-virtualpullback}
The following proposition is key in the comparison between pullback in critical cohomology and virtual pullback in Borel--Moore homology, via the dimensional reduction (Proposition~\ref{proposition:comparison2d3dmultiplications}). Its goal is to expand the part of \cite[Appendix]{ren2017cohomological} regarding the comparison of the CoHA products for preprojective algebras of quivers and quivers with potential. Indeed, it is not clear how the two diagrams after $(8)$ in \emph{op. cit.} are defined, since there is in general no canonical morphism $\varphi\rightarrow\id$. Instead, there is a zigzag of morphisms, with the restriction to the zero locus of the function $f$ in the middle ($\varphi_f\rightarrow \imath_*\imath^*\leftarrow \id$).
\begin{proposition}
\label{proposition:dimredpullback}
We are in the situation described in \S\ref{subsubsection:setupfunctorialitydimred} and we assume furthermore that $X,Y$ are smooth. The natural pullback morphism
\[
 \BoQ_{\overline{Y}}\rightarrow \overline{g}_*\BoQ_{\overline{X}}
\]
may be rewritten
\begin{equation}
\label{equation:pullbackBMhom}
 (\BD\BoQ_{\overline{Y}})\otimes\SL^{\dim \overline{Y}}\rightarrow \overline{g}_*(\BD\BoQ_{\overline{X}})\otimes\SL^{\dim \overline{X}}\,.
\end{equation}
Set $K_Y\coloneqq\BD\BoQ_{\overline Y}\otimes\SL^{\dim\overline Y}$ and $K_X\coloneqq\BD\BoQ_{\overline X}\otimes\SL^{\dim\overline X}$. Then, we have a commutative diagram
\begin{equation}
\label{equation:commutativediagrampullback}
 % https://q.uiver.app/#q=WzAsMTEsWzAsMCwiXFx2YXJwaGlfe2Z9XFxCRFxcQm9RX3tcXG92ZXJsaW5le1l9fSJdLFsxLDAsIlxcdmFycGhpX2ZcXG92ZXJsaW5le2d9XypcXEJEXFxCb1Ffe1xcb3ZlcmxpbmV7WH19Il0sWzIsMCwiXFxvdmVybGluZXtnfV8qXFx2YXJwaGlfe2ZcXGNpcmNcXG92ZXJsaW5le2d9fVxcQkRcXEJvUV97XFxvdmVybGluZXtYfX0iXSxbMCwxLCJcXGltYXRoXyFcXGltYXRoXiFcXEJEXFxCb1Ffe1xcb3ZlcmxpbmV7WX19Il0sWzEsMSwiXFxpbWF0aF8hXFxpbWF0aF4hXFxvdmVybGluZXtnfV8qXFxCRFxcQm9RX3tcXG92ZXJsaW5le1h9fSJdLFsyLDFdLFswLDIsIlxcQkRcXEJvUV97XFxvdmVybGluZXtZfX0iXSxbMSwyLCJcXG92ZXJsaW5le2d9XypcXEJEXFxCb1Ffe1xcb3ZlcmxpbmV7WH19Il0sWzAsMywiXFxvdmVybGluZXtcXGltYXRofV8qXFxCRFxcQm9RX3tcXG92ZXJsaW5le1p9X3N9XFxjb25nXFxvdmVybGluZXtcXGltYXRofV8hXFxvdmVybGluZXtcXGltYXRofV4hXFxCRFxcQm9RX3tcXG92ZXJsaW5le1l9fSJdLFsxLDMsIlxcb3ZlcmxpbmV7XFxpbWF0aH1fIVxcb3ZlcmxpbmV7XFxpbWF0aH1eIVxcb3ZlcmxpbmV7Z31fKlxcQkRcXEJvUV97XFxvdmVybGluZXtYfX1cXGNvbmdcXG92ZXJsaW5le1xcaW1hdGh9XyooXFxvdmVybGluZXtnfV9zKV8qXFxCRFxcQm9RX3tcXG92ZXJsaW5le1p9X3tnXipzfX0iXSxbMiwzLCJcXG92ZXJsaW5le2d9XypcXG92ZXJsaW5le1xcaW1hdGh9XyFcXG92ZXJsaW5le1xcaW1hdGheIX1cXEJEXFxCb1Ffe1xcb3ZlcmxpbmV7WH19Il0sWzAsMV0sWzEsMl0sWzYsN10sWzMsMF0sWzQsMV0sWzMsNF0sWzgsOV0sWzksN10sWzgsNl0sWzMsNl0sWzQsN10sWzgsMywiIiwwLHsiY3VydmUiOi01fV0sWzksNCwiIiwwLHsiY3VydmUiOjV9XSxbOCwwLCJcXG1hdGhzZntkcn0iLDAseyJjdXJ2ZSI6LTV9XSxbOSwxMCwiXFxzaW0iXSxbMTAsMiwiXFxvdmVybGluZXtnfV8qXFxtYXRoc2Z7ZHJ9IiwyXV0=
\begin{tikzcd}
	{\varphi_{f}K_Y} & {\varphi_f\overline{g}_*K_X} & {\overline{g}_*\varphi_{f\circ\overline{g}}K_X} \\
	{\imath_!\imath^!K_Y} & {\imath_!\imath^!\overline{g}_*K_X} & {} \\
	{K_Y} & {\overline{g}_*K_X} \\
	{\overline{\imath}_!\overline{\imath}^!K_Y} & {\overline{\imath}_!\overline{\imath}^!\overline g_*K_X} & {\overline{g}_*\overline{\imath}_!\overline{\imath}^!K_X}
	\arrow[from=1-1, to=1-2]
	\arrow["a",from=1-2, to=1-3]
	\arrow[from=2-1, to=1-1]
	\arrow[from=2-1, to=2-2]
	\arrow[from=2-1, to=3-1]
	\arrow[from=2-2, to=1-2]
	\arrow[from=2-2, to=3-2]
	\arrow[from=3-1, to=3-2]
	\arrow["{\mathsf{dr}}", bend left=80, from=4-1, to=1-1]
	\arrow[bend left=80, from=4-1, to=2-1]
	\arrow[from=4-1, to=3-1]
	\arrow[from=4-1, to=4-2]
	\arrow[bend right=80, from=4-2, to=2-2]
	\arrow[from=4-2, to=3-2]
	\arrow["\sim","b"', from=4-2, to=4-3]
	\arrow["{\overline{g}_*\mathsf{dr}}"', from=4-3, to=1-3]
\end{tikzcd}
\end{equation}
of adjunction and dimensional reduction morphisms. After taking derived global sections, one obtains a commutative diagram in which the vertical arrows are the dimensional reduction isomorphisms:
\begin{equation}
\label{equation:compatibilitypullbackdr}
 % https://q.uiver.app/#q=WzAsNCxbMCwwLCJcXEhPXiooXFxvdmVybGluZXtZfSxcXHZhcnBoaV9mKSJdLFsxLDAsIlxcSE9eKihcXG92ZXJsaW5le1h9LFxcdmFycGhpX2YpIl0sWzAsMSwiXFxIT157XFxybUJNfV8qKFxcb3ZlcmxpbmV7Wn1fcyxcXEJvUSkiXSxbMSwxLCJcXEhPXntcXHJtQk19XyooXFxvdmVybGluZXtafV97Z14qc30sXFxCb1EpIl0sWzIsMywiXFxvdmVybGluZXtnfV4hIl0sWzMsMSwiXFxtYXRoc2Z7ZHJ9IiwyXSxbMiwwLCJcXG1hdGhzZntkcn0iXSxbMCwxLCJcXG92ZXJsaW5le2d9XioiXV0=
\begin{tikzcd}
	{\HO^*(\overline{Y},\varphi_f\BoQ)} & {\HO^*(\overline{X},\varphi_{f\circ\overline{g}}\BoQ)} \\
	{\HO^{\rmBM}_*(Z_s,\BoQ)\otimes\SL^{\dim Y}} & {\HO^{\rmBM}_*(Z_{g^*s},\BoQ)\otimes\SL^{\dim X}}
	\arrow["{\overline{g}^*}", from=1-1, to=1-2]
	\arrow["{\mathsf{dr}}", from=2-1, to=1-1]
	\arrow["{g^!}", from=2-1, to=2-2]
	\arrow["{\mathsf{dr}}"', from=2-2, to=1-2]
\end{tikzcd}
\end{equation}
where $\overline{g}^*$ is the natural pullback in vanishing cycle cohomology \eqref{equation:pbvanishingcyclecoh} while $g^!$ is the virtual pullback in Borel--Moore homology (Definition~\ref{definition:virtualpullbackBMhomology}). The Tate twists are written explicitly here: if $r=\rank V$, the Thom isomorphism gives $\HO^{\rmBM}_*(\overline Z_s)\cong\HO^{\rmBM}_*(Z_s)\otimes\SL^{-r}$, and similarly on $X$.
\end{proposition}
\begin{proof}
Let $\eta_Y\colon\overline\imath_!\overline\imath^!\to\varphi_f$ and $\eta_X\colon\overline\imath_!\overline\imath^!\to\varphi_{f\circ\overline g}$ denote the dimensional-reduction morphisms before applying the bundle projections. Naturality of these morphisms in the coefficient complex makes the left-hand squares of \eqref{equation:commutativediagrampullback} commute. The right-hand perimeter is the identity
\[
 a\circ\eta_Y\overline g_*
 =\overline g_*\eta_X\circ b,
\]
which is Lemma~\ref{lemma:exchange-zero-fiber} for $h=\overline g$ and $Z=\overline Z_s$. Its exceptional base-change isomorphism is exactly the arrow $b$ of the diagram. Composing with the adjunction map $K_Y\to\overline g_*K_X$ proves commutativity of its entire perimeter.

The bottom morphism is obtained by applying $\overline\imath^!$ to this adjunction map. By Definition~\ref{definition:virtualpullbackBMhomology}, it is the refined pullback for the Cartesian square with ambient smooth total spaces $\overline X,\overline Y$. Under the Thom identifications for their rank-$r$ bundles, smooth base change identifies it with $g^!$ on $Z_s,Z_{g^*s}$, with twists $\SL^{\dim Y}$ and $\SL^{\dim X}$ as displayed in \eqref{equation:compatibilitypullbackdr}. Finally, apply the bundle projection pushforwards and then global sections. The dimensional-reduction theorem makes the two vertical maps isomorphisms at this stage. No isomorphism on the total spaces themselves is asserted.
\end{proof}

\subsection{Vanishing cycles for quadratic functions}

\begin{proposition}
\label{proposition:vanishingcyclesquadraticfunction}
Let $X$ be an algebraic variety, $V$ a vector bundle on $X$. We let $f\colon V\times_X V^*\rightarrow\BoC$ be the canonical regular function. If $V\times_XV^*$ is seen as a vector bundle on $V$, it is induced by the diagonal section of the dual vector bundle $V\rightarrow V\times_XV$. If $V\times_XV^*$ is seen as a vector bundle over $V^*$, then it is induced by the diagonal section $V^*\rightarrow V^*\times_XV^*$. We consider the morphisms of vector bundles
% https://q.uiver.app/#q=WzAsNSxbMywwXSxbMCwwLCJWXFx0aW1lc19YIFZeKiJdLFsxLDAsIlYiXSxbMSwxLCJYIl0sWzAsMSwiVl4qIl0sWzEsNCwiXFx0aWxkZXtcXHBpfSJdLFsxLDIsIlxcdGlsZGV7XFxwaX1ee1xcdmVlfSJdLFs0LDMsIlxccGlee1xcdmVlfSJdLFsyLDMsIlxccGkiLDJdLFsyLDEsIlxcaW1hdGhee1xcdmVlfSIsMix7ImN1cnZlIjoyfV0sWzQsMSwiXFxpbWF0aCIsMCx7ImN1cnZlIjotMn1dXQ==
\[\begin{tikzcd}
	{V\times_X V^*} & V && {} \\
	{V^*} & X
	\arrow["{\tilde{\pi}^{\vee}}", from=1-1, to=1-2]
	\arrow["{\tilde{\pi}}", from=1-1, to=2-1]
	\arrow["{\imath^{\vee}}"', bend right = 30, from=1-2, to=1-1]
	\arrow["\pi"', from=1-2, to=2-2]
	\arrow["\imath", bend left=30, from=2-1, to=1-1]
	\arrow["{\pi^{\vee}}", from=2-1, to=2-2]
\end{tikzcd}\]

Then, the dimensional reduction isomorphisms induce morphisms
\[
 (\pi\circ\tilde{\pi}^{\vee})_!\varphi_f(\pi\circ\tilde{\pi}^{\vee})^*\rightarrow \pi_!\pi^*\cong\id\otimes\SL^{\rank V}
\]
and
\[
  (\pi^{\vee}\circ\tilde{\pi})_!\varphi_f(\pi^{\vee}\circ\tilde{\pi})^*\rightarrow \pi^{\vee}_!(\pi^{\vee})^*\cong\id\otimes\SL^{\rank V}
\]
whose sources are canonically identified by the equality $\pi^{\vee}\circ \tilde{\pi}=\pi\circ\tilde{\pi}^{\vee}$, and they differ by the sign $(-1)^{\rank(V)}$.
\end{proposition}
\begin{proof}
We can perform dimensional reduction along the fibers of $\tilde{\pi}^{\vee}$ or $\tilde{\pi}$. The dimensional reduction along the fibers of $\tilde{\pi}$ gives that the natural morphism
\[
 \tilde{\pi}_!\varphi_f\tilde{\pi}^*\rightarrow \tilde{\pi}_!\imath^{\vee}_*(\imath^{\vee})^*\tilde{\pi}^*
\]
is an isomorphism. By precomposing with $(\pi^{\vee})^*$ and postcomposing with $\pi^{\vee}_!$, we obtain the second dimensional reduction isomorphism of the proposition. The first isomorphism is obtained by exchanging $\tilde{\pi}$ and $\tilde{\pi}^{\vee}$.

The assertion is local on $X$. After trivializing $V$, both morphisms are obtained by tensoring the corresponding morphisms over a point with the identity on the base. We may therefore assume that $X=\pt$. We are then reduced to \cite[Appendix by Ben Davison, Lemma~4.1]{ren2017cohomological}. We obtain that the two morphisms indeed differ by the sign $(-1)^{\rank V}$.
\end{proof}

\subsection{Critical excess intersection for subbundles}
\begin{lemma}
\label{lemma:critical-subbundle-excess}
Let $X$ be the total space of an equivariant vector bundle over a smooth base, and let $U,Y\subset X$ be equivariant vector subbundles such that $Z=U\cap Y$ and $N=U+Y$ are subbundles. Let $f$ be an invariant regular function on $X$, and use its restrictions on the other spaces. Write $p\colon U\hookrightarrow X$, $i\colon Y\hookrightarrow X$, $c\colon Z\hookrightarrow U$ and $d\colon Z\hookrightarrow Y$. Then the critical pullback and Gysin maps of \S\ref{subsection:vanishingcyclesfunctoriality} satisfy
\[
 i^*p_*=d_*\bigl(\Eu(X/N)|_Z\cdot c^*(-)\bigr).
\]
Here $X/N$ denotes the quotient bundle on the common base, and its Euler class is pulled back to the indicated spaces.
\end{lemma}
\begin{proof}
Factor $p=ka$ and $i=kb$, where $a\colon U\hookrightarrow N$, $b\colon Y\hookrightarrow N$ and $k\colon N\hookrightarrow X$. The first two subbundles intersect transversely in $N$. Critical transverse base change and self-intersection give
\[
 b^*a_*=d_*c^*,\qquad k^*k_*=\Eu(X/N),
\]
by \cite[Corollary~2.15 and Proposition~2.16]{davison2017critical}. Hence
\[
 i^*p_*=b^*k^*k_*a_*
 =d_*\bigl(\Eu(X/N)|_Z\cdot c^*(-)\bigr)
\]
by the projection formula. The codimensions add, so the identity respects the Tate twists. Applying this argument on smooth equivariant approximations gives the statement for the quotient stacks used below.
\end{proof}

\subsection{Pushforward for affine fibrations}
We explain how the pushforward for a vector bundle is defined in critical cohomology. We let $\pi\colon X\rightarrow Y$ be a $G$-equivariant vector bundle and $f\colon Y\rightarrow\BoC$ a $G$-invariant function. Then, the natural adjunction morphism
\[
 \BoQ_{Y}\rightarrow\pi_*\BoQ_X
\]
is an isomorphism, and the same remains true after applying $\varphi_f$. This gives the pullback in vanishing cycle cohomology $\pi^*\colon \HO_G^*(Y,\varphi_f)\rightarrow\HO_G^*(X,\varphi_{f\circ \pi})$ for the vector bundle $\pi$. We let $\Eu_{\pi}\in\HO^*_G(Y)$ be the $G$-equivariant Euler class of $\pi$. The pushforward by $\pi$ is defined as
\[
 \Eu_{\pi}^{-1}(\pi^*)^{-1}\colon\HO_G^*(X,\varphi_{f\circ \pi})\rightarrow\HO_G^*(Y,\varphi_f)[\Eu_{\pi}^{-1}]\,,
\]
that is, we divide the inverse pullback map $(\pi^*)^{-1}$ by the Euler class $\Eu_{\pi}$.

\section{The Coxeter complex}
\label{section:Coxetercomplex}

\subsection{Weights}
\label{subsection:weights}

Let $G$ be a reductive algebraic group and $V$ a representation of $G$. We let $T\subset G$ be a maximal torus. The action of $T$ on $V$ is diagonalizable and so we may write $V=\bigoplus_{\alpha\in\rmX^*(T)}V_{\alpha}$ where $V_{\alpha}\coloneqq \{v\in V\mid t\cdot v=\alpha(t)v\text{ for any $t\in T$}\}$ is the \emph{weight space of weight $\alpha$}. A character $\alpha\in\rmX^*(T)$ is called a \emph{weight} of $V$ if $V_{\alpha}\neq 0$. We let $\CW(V)\subset\rmX^*(T)$ be the collection of weights of $V$, counted with multiplicities. We let $\CW'(V)$ be the set of weights of $V$, i.e. the set obtained from $\CW(V)$ by removing the multiplicities.

Two elements $\alpha,\beta\in\rmX^*(T)$ are called \emph{equivalent} if there exists a positive rational number $r\in\BoQ_+^*$ such that $r\alpha=\beta$. Given a collection $\CW\subset \rmX^*(T)$ of characters, we let $\overline{\CW}$ be the collection of equivalence classes of characters (with multiplicities taken into account).

\begin{definition}
 Let $V$ be a representation of a reductive group $G$. We say that $V$ is \emph{weakly symmetric} if $\overline{\CW}(V)=\overline{\CW}(V^*)$. We say that $V$ is \emph{symmetric} if $\CW(V)=\CW(V^*)$.
\end{definition}
From the representation theory of reductive groups, two representations are isomorphic if and only if they have the same character. Therefore, $V$ is symmetric if and only if it is self-dual, that is $V\cong V^*$ as representations of $G$. We do not know any similar interpretation for the notion of weak symmetricity.

\begin{example}
 The adjoint representation $\Fg$ of a reductive algebraic group $G$ is a symmetric representation of $G$. We refer to \cite[Example~2.4]{hennecart2024cohomological} for more examples of symmetric representations.
\end{example}

\subsection{Walls and chambers}
\label{subsection:wallchamber}
We let $\Fh\coloneqq\Lie(T)$ be the Lie algebra of $T$ and define $\Fh_{\BoR}\coloneqq \rmX_*(T)\otimes_{\BoZ}\BoR$ the real form of $\Fh$. For each non-zero character $\alpha\in \rmX^*(T)$, there is a hyperplane $H_{\alpha}\coloneqq \{\alpha=0\}\subset \Fh_{\BoR}$. We let $\bfH\coloneqq \cup_{\alpha\in(\CW(V)\cup\CW(\Fg))\setminus\{0\}}H_{\alpha}$ be the hyperplane arrangement associated to $\CW(V)\cup\CW(\Fg)$. We let $\FH\coloneqq\{H_{\alpha}\colon\alpha\in(\CW(V)\cup\CW(\Fg))\setminus\{0\}\}$ be the set of hyperplanes. An intersection of hyperplanes in $\FH$ is called a \emph{flat}. The empty intersection is understood to be $\Fh_{\BoR}$. The set of flats is denoted by $\FF$. A connected component of a subset of $\Fh_{\BoR}$ of the form
\[
 \bigcap_{H\in \FH_1}H\setminus\bigcup_{H\in\FH_2}H
\]
for some decomposition $\FH=\FH_1\sqcup\FH_2$ is called a \emph{cell}. The set of cells is denoted by $\FC$. The linear span of a cell $C\in\FC$ is a flat of $\bfH$, which is denoted by $\langle C\rangle$.

\subsection{Positive weights}
\label{subsection:positiveweights}
Let $F\in\FF$. We define $\CW^F(V)\subset \CW(V)$ as the subset of weights of $V$ which vanish on $F$. For $C\in\FC$, we define $\CW^{C\geq 0}(V)$ as the subset of weights of $V$ which are nonnegative on $C$, $\CW^{C>0}(V)$ the subset of weights which are positive on $C$, $\CW^{C\geq 0,F}(V)$ the weights which are nonnegative on $C$ and vanishing on $F$, etc. We define the analogous sets $\CW^{\lambda}(V)$, $\CW^{\lambda\geq 0}(V)$, etc for a cocharacter $\lambda\in\rmX_*(T)$.

\subsection{Partial order}
\label{subsection:partialorder}
We order the set of flats by reverse inclusion: for any $F,F'\in\FF$, $F\preceq F'$ if $F'\subset F$. We order the set of cells by reverse inclusion of the closure: $C\preceq C'$ if $C'\subset\overline{C}$. We may also compare pairs $(C,F)\in\FC\times\FF$ of a cell and a flat via $C\preceq F$ if $F\subset\langle C\rangle$ and $F\preceq C$ if $C\subset F$.

\subsection{Tits product of cells}
\label{subsection:Titsproduct}
The Tits product is an operation on the set of cells of the hyperplane arrangement which associates a third cell to two given cells. We recall its definition following \cite{kapranov2024langlands}. It is defined as follows. Let $C,C'\in\FC$. Then, we let $\lambda\in C$ and $\mu\in C'$ and we draw the line segment $[\lambda,\mu]\subset\Fh_{\BoR}$. The \emph{Tits product} $C\circ C'$ is defined to be the cell containing $(1-\varepsilon)\lambda+\varepsilon\mu$ for all sufficiently small $\varepsilon>0$. It is well-defined, and $C\circ C'\neq C'\circ C$ in general.

\begin{lemma}
 For any $C,C'\in\FC$, we have $C\circ C'\preceq C$.
\end{lemma}
\begin{proof}
 From the definition of the Tits product, for any $\lambda\in C$, a small ball around $\lambda$ for the Euclidean topology intersects $C\circ C'$. Since all subsets of $\Fh_{\BoR}$ we consider are $\BoR_+^*$-invariant, this implies that $C\subset \overline{C\circ C'}$ and so $C\circ C'\preceq C$.
\end{proof}

\subsection{Equivariance of the Tits product}
\label{subsection:equivarianceTitsproduct}
The Weyl group $\sfW$ of the pair $(G,T)$ acts on the sets of weights $\CW(V)$, $\CW(\Fg)$. Therefore, it also acts on $\bfH$, the set of flats $\FF$ and the set of cells. It acts on $\Fh_{\BoR}$ and since this action is linear, it preserves lines. We immediately obtain the equivariance of the Tits product.

\begin{proposition}
 \label{proposition:equivarianceTits}
 For any $C,C'\in\FC$ and $w\in \sfW$, we have $(w\cdot C)\circ (w\cdot C')=w\cdot(C\circ C')$.\qed
\end{proposition}

\subsection{Reformulation of the Tits product}
\label{subsection:reformulationTitsproduct}
We have defined the Tits product in an abstract way. We may give here a more concrete description. For a character $\alpha\in\rmX^*(T)\subset\Fh_{\BoR}^*$ and $C\in\FC$ a cell, we denote by $\alpha_C$ the restriction of $\alpha$ to $C$. If $\alpha\in\CW(V)\cup\CW(\Fg)$ and $C\in\FC$, $\alpha$ has a constant sign on $C$ and we may write $\alpha_C>0$ (resp. $\alpha_C<0$, $\alpha_C=0$) if this sign is $+$ (resp. $-$, $0$).

\begin{proposition}
 \label{proposition:reformulationTits}
 Let $\alpha\in \CW(V)\cup\CW(\Fg)$. Then,
 \[
  \alpha_{C\circ C'}>0\iff (\alpha_C>0 \text{ or }(\alpha_C=0\text{ and }\alpha_{C'}>0))\,.
 \]
\end{proposition}
\begin{proof}
 We let $\lambda\in C$ and $\mu\in C'$. We let $[\lambda,\mu]\subset\Fh_{\BoR}$ be the line segment joining $\lambda$ to $\mu$. We let $\nu\in[\lambda,\mu]\cap C\circ C'$. We have to show the equivalence
 \[
  \langle \alpha,\nu\rangle>0\iff \langle \alpha,\lambda\rangle>0 \text{ or } (\langle \alpha,\lambda\rangle=0 \text{ and } \langle \alpha,\mu\rangle>0)\,.
 \]
If $\langle\alpha,\nu\rangle>0$, then $\langle\alpha,\lambda\rangle\geq 0$, since $C\subset\overline{C\circ C'}$. If $\langle\alpha,\lambda\rangle=0$, then the line segment meets the hyperplane $\{\alpha=0\}$ only in $\lambda$. The sign of $\alpha$ on the half-open segment $(\lambda,\mu]$ is constant and so $\langle\alpha,\mu\rangle>0$ since $\langle \alpha,\nu\rangle>0$. This proves the direct implication.

For the reverse implication, we first assume that $\alpha_C>0$. Then, $C\subset \overline{C\circ C'}$ implies that $\alpha>0$ in a neighbourhood of $C$ which intersects $C\circ C'$. Therefore, $\langle\alpha,\nu\rangle>0$. If $\alpha_C=0$, then the line segment $(\lambda,\mu]$ is entirely contained in a half-space determined by $\alpha=0$ and so $\langle \alpha,\nu\rangle$ and $\langle \alpha,\mu\rangle$ have the same sign. Since $\langle\alpha,\mu\rangle>0$, we have $\langle\alpha,\nu\rangle>0$ and this concludes.
\end{proof}

\section{The critical induction system}
\label{section:inductionsystem}

\subsection{Induction diagram}
\label{subsection:inductiondiagram}
We let $V$ be a representation of a reductive group $G$, $\FC$ and $\FF$ be respectively the sets of cells and flats of the Coxeter system associated to $(G,V)$. We let $T'$ be an auxiliary torus acting on $V$ such that the actions of $G$ and $T'$ on $V$ commute.

For $\lambda\in\rmX_*(T)$, we define $V^{\lambda}, V^{\lambda\geq 0}, G^{\lambda}, G^{\lambda\geq 0}$ as in \cite[\S2.1]{hennecart2024cohomological}. More precisely, $V^{\lambda}$ is the fixed locus for the action of $\BoG_{\rmm}$ on $V$ induced by $\lambda$, $G^{\lambda}$ is the Levi subgroup of $G$ which is the fixed locus for the $\BoG_{\rmm}$-action induced by the adjoint action of $G$, $V^{\lambda\geq 0}$ and $G^{\lambda\geq 0}$ are the attracting loci. We may define $V^{\lambda\leq 0}, V^{\lambda>0},\dots$ analogously.

For any $F\in \FF$ and $C\in\FC$, we define
\begin{enumerate}
 \item $V^F\coloneqq V^{\lambda}$ for $\lambda\in F$ a general cocharacter (more precisely, $\lambda$ is in an open cell contained in $F$ or equivalently, $\lambda\in F\setminus\bigcup_{\substack{H\in \FH\\F\not\subset H}}H)$,
 \item $V^{C\geq 0,F}\coloneqq V^{\lambda\geq 0}\cap V^F$ for some cocharacter $\lambda\in C$. We let $V^{C\geq 0}=V^{\lambda\geq 0}$,
 \item $G^F\coloneqq G^{\lambda}$ for $\lambda\in F$ a general cocharacter,
 \item $G^{C\geq0,F}\coloneqq G^{\lambda\geq 0}\cap G^F$ for $\lambda\in C$ a cocharacter. We also let $G^{C\geq 0}\coloneqq G^{\lambda\geq 0}$.
\end{enumerate}
We may also define $G^{C\leq 0,F}, V^{C\leq 0,F},\dots$ in a similar fashion. We write $\sfW^F$ for the Weyl group of $G^F$, and occasionally write $\sfW^C=\sfW^{\langle C\rangle}$, $V^C=V^{\langle C\rangle}$, $\Ind_D^C=\Ind_D^{\langle C\rangle}$ and $\Res_C^D=\Res_{\langle C\rangle}^D$.

When $C\preceq F$, there is a $(G^{C\geq 0,F}\times T',G^F\times T')$-equivariant closed immersion $p_{C,F}\colon V^{C\geq 0,F}\rightarrow V^F$ and a $(G^{C\geq 0,F}\times T',G^{\langle C\rangle}\times T')$-equivariant projection $q_{C,F}\colon V^{C\geq 0,F}\rightarrow V^{\langle C\rangle}$.

We obtain the induction diagram
\begin{equation}
\label{equation:inductiondiagramcritical}
 V^{\langle C\rangle}/(G^{\langle C\rangle}\times T')\xleftarrow{q_{C,F}}V^{C\geq0,F}/(G^{C\geq 0,F}\times T')\xrightarrow{p_{C,F}} V^F/(G^F\times T')\,.
\end{equation}

For any $C\in\FC$ and $F\in\FF$ such that $C\preceq F$, we define $d_{C,F}$ as the codimension of $p_{C,F}$, that is $d_{C,F}\coloneqq \dim V^{C<0,F}-\dim\Fg^{C<0,F}$.

\subsection{Potential}
\label{subsection:potential}
Let $V$ be a representation of a reductive group $G$. Let $f$ be a $G$-invariant function on $V$. For any $F\in \FF$, we let $f_F$ be the restriction of $f$ to $V^F\subset V$. It is a $G^F$-invariant function.
\begin{proposition}
\label{proposition:potential}
 Let $f$ be a $G$-invariant function on $V$. Then, for any $C\in \FC$ and $F\in\FF$ such that $C\preceq F$, $f_F\circ p_{C,F}=f_{\langle C\rangle}\circ q_{C,F}$.
\end{proposition}
\begin{proof}
We let $\lambda\in C$ be a cocharacter of $T$. Then, for any $x\in V^{C\geq 0,F}$ the function $t\in\BoC^*\mapsto f_F(\lambda(t)\cdot x)$ is constant by $G$-invariance of $f$. Therefore, $f_F(x)=f_F(\lim_{t\rightarrow 0}\lambda(t)\cdot x)=f_{\langle C\rangle}(q_{C,F}(x))$.
\end{proof}

\subsection{Critical cohomology}
\label{subsection:criticalcohomology}
For any $F\in\FF$, we define $\CH_{G,V,f,F}^{T'}\coloneqq \HO^*_{G^F\times T'}(V^F,\varphi_{f_F}\BoQ)$. We also define $\CH_{T,V,f,F}^{T'}\coloneqq \HO^*_{T\times T'}(V^F,\varphi_{f_F}\BoQ)$ by restricting the action of $G$ to the maximal torus $T\subset G$. It admits an action of the Weyl group $\sfW^F$ of $G^F$ (see Lemma~\ref{lemma:Winvariantpart} for the relation between these critical cohomologies for different groups).

\subsection{Parabolic induction}
\label{subsection:parabolic_induction}
We consider the induction diagram \eqref{equation:inductiondiagramcritical} which we enrich with good moduli spaces $\CV_{F}^{T'}\coloneqq (V^F\cms G^F)/T'$ (the quotient stack of the scheme $V^F\cms G^F$ by the torus $T'$) as follows.

% https://q.uiver.app/#q=WzAsNSxbMCwxLCJWXntcXGxhbmdsZSBDXFxyYW5nbGV9LyhHXntcXGxhbmdsZSBDXFxyYW5nbGV9XFx0aW1lcyBUKSJdLFsxLDAsIlZee0NcXGdlcSAwLEZ9LyhHXntDXFxnZXEgMCxGfVxcdGltZXMgVCkiXSxbMiwxLCJWXkYvKEdeRlxcdGltZXMgVCkiXSxbMCwyLCJcXENWX3tcXGxhbmdsZSBDXFxyYW5nbGV9XlQiXSxbMiwyLCJcXENWX3tGfV57VH0iXSxbMSwwLCJxX3tDLEZ9IiwyXSxbMSwyLCJwX3tDLEZ9Il0sWzAsMywiXFxwaV97XFxsYW5nbGUgQ1xccmFuZ2xlfSIsMl0sWzMsNCwiXFxpbWF0aF97XFxsYW5nbGUgQ1xccmFuZ2xlLEZ9Il0sWzIsNCwiXFxwaV9GIl1d
\[\begin{tikzcd}
	& {V^{C\geq 0,F}/(G^{C\geq 0,F}\times T')} \\
	{V^{\langle C\rangle}/(G^{\langle C\rangle}\times T')} && {V^F/(G^F\times T')} \\
	{\CV_{\langle C\rangle}^{T'}} && {\CV_{F}^{T'}}
	\arrow["{q_{C,F}}"', from=1-2, to=2-1]
	\arrow["{p_{C,F}}", from=1-2, to=2-3]
	\arrow["{\pi_{\langle C\rangle}}"', from=2-1, to=3-1]
	\arrow["{\pi_F}", from=2-3, to=3-3]
	\arrow["{\imath_{\langle C\rangle,F}}", from=3-1, to=3-3]
\end{tikzcd}\]
The map $\imath_{\langle C\rangle,F}$ is finite by \cite[Lemma~2.2]{hennecart2024cohomological} since it is induced by the $T'$-equivariant finite map $V^{\langle C\rangle}\cms G^{\langle C\rangle}\rightarrow V^F\cms G^F$.

The pullback by $q_{C,F}$ is induced by a morphism of constructible sheaves or mixed Hodge modules
\[
 \BoQ_{V^{\langle C\rangle}/(G^{\langle C\rangle}\times T')}\rightarrow(q_{C,F})_*\BoQ_{V^{C\geq 0,F}/(G^{C\geq 0,F}\times T')}\,.
\]
By applying the vanishing cycle sheaf functor $\varphi_{f_{\langle C\rangle}}$, we obtain the morphism
\begin{equation}
\label{equation:pbcritical}
 \varphi_{f_{\langle C\rangle}}\BoQ_{V^{\langle C\rangle}/(G^{\langle C\rangle}\times T')}\rightarrow\varphi_{f_{\langle C\rangle}}(q_{C,F})_*\BoQ_{V^{C\geq 0,F}/(G^{C\geq 0,F}\times T')}\,.
\end{equation}
The pushforward by the proper and representable morphism $p_{C,F}$ is induced by a morphism of constructible sheaves
\[
 (p_{C,F})_*\BoQ_{V^{C\geq 0,F}/(G^{C\geq 0,F}\times T')}\rightarrow\BoQ_{V^F/(G^F\times T')}\otimes\SL^{-d_{C,F}}\,.
\]
By applying the vanishing cycle sheaf functor $\varphi_{f_F}$, we obtain the morphism
\begin{equation}
\label{equation:pfcritical}
  \varphi_{f_F}(p_{C,F})_*\BoQ_{V^{C\geq 0,F}/(G^{C\geq 0,F}\times T')}\rightarrow\varphi_{f_F}\BoQ_{V^F/(G^F\times T')}\otimes\SL^{-d_{C,F}}
\end{equation}
The good moduli pushforward $\pi_F$ commutes with vanishing cycles by \cite[Lemma~9.1]{hennecart2024cohomological2}. The same holds with $T'$-equivariance: the exchange morphism is equivariant, and it is an isomorphism after forgetting the $T'$-action, which is conservative. This uses the commutation statement itself, rather than asserting that $\pi_F$ is approachable by proper maps. By smoothness of $q_{C,F}$, properness of $p_{C,F}$, finiteness of $\imath_{\langle C\rangle,F}$ and Proposition~\ref{proposition:potential}, one can therefore identify canonically
\[
 (\imath_{\langle C\rangle,F})_*(\pi_{\langle C\rangle})_*{\varphi}_{f_{\langle C\rangle}}(q_{C,F})_*\BoQ_{V^{C\geq 0,F}/(G^{C\geq 0,F}\times T')}\quad\text{and}\quad (\pi_F)_*\varphi_{f_F}(p_{C,F})_*\BoQ_{V^{C\geq 0,F}/(G^{C\geq 0,F}\times T')}\,.
\]
Therefore, by composing $(\imath_{\langle C\rangle,F})_*(\pi_{\langle C\rangle})_*$ applied to \eqref{equation:pbcritical} and $(\pi_F)_*$ applied to \eqref{equation:pfcritical}, we obtain a morphism of constructible complexes
\begin{equation}
\label{equation:untwistedcriticalinduction}
(\imath_{\langle C\rangle,F})_*(\pi_{\langle C\rangle})_*{\varphi}_{f_{\langle C\rangle}}\BoQ_{V^{\langle C\rangle}/(G^{\langle C\rangle}\times T')}\rightarrow(\pi_F)_*\varphi_{f_F}\BoQ_{V^F/(G^F\times T')}\otimes\SL^{-d_{C,F}}\,.
\end{equation}
It is customary to give $\BoQ_{V^{\langle C\rangle}/(G^{\langle C\rangle}\times T')}$ and $\BoQ_{V^F/(G^F\times T')}$ the virtual twist relative to $BT'$
\[
 \BoQ_{V^F/(G^F\times T')}^{\vir}\coloneqq\BoQ_{V^F/(G^F\times T')}\otimes \SL^{-d_F/2}
\]
where $d_F\coloneqq \dim V^F-\dim G^F$.

We can then rewrite \eqref{equation:untwistedcriticalinduction} as
\begin{equation}
\label{equation:twistedcriticalinduction}
 (\imath_{\langle C\rangle,F})_*(\pi_{\langle C\rangle})_*{\varphi}_{f_{\langle C\rangle}}\BoQ_{V^{\langle C\rangle}/(G^{\langle C\rangle}\times T')}^{\vir}\rightarrow(\pi_F)_*{\varphi}_{f_F}\BoQ_{V^F/(G^F\times T')}^{\vir}\otimes\SL^{\frac{d_F}{2}-\frac{d_{\langle C\rangle}}{2}-d_{C,F}}\,.
\end{equation}

This is the sheafified version of the induction morphism. By taking derived global sections, we obtain the induction morphism
\begin{equation}
\label{equation:inductionabsolute}
 \Ind_{C}^F\colon \CH_{G,V,f,\langle C\rangle}^{T'}\rightarrow \CH_{G,V,f,F}^{T'}\,.
\end{equation}

\begin{proposition}
 If $V$ is a weakly symmetric representation of $G$, then the induction morphisms $\Ind_{C}^F$ respect the virtual cohomological grading.
\end{proposition}
\begin{proof}
If $V$ is weakly symmetric, then $d_F-d_{\langle C\rangle}-2d_{C,F}=0$, and so the Tate twist in the r.h.s. of \eqref{equation:twistedcriticalinduction} disappears. Therefore, the induction morphisms respect the virtual cohomological grading.
\end{proof}

\subsection{Induction morphisms and mixed Hodge structures}
\label{subsection:inductionmhs}
As explained in \S\ref{subsection:conventions}, our convention is to hide the Tate twists when we take derived global sections. All cohomology spaces come with natural mixed Hodge structures. We explain in this section how to write the induction morphisms as morphisms of mixed Hodge structures. The morphism \eqref{equation:twistedcriticalinduction} is a morphism of monodromic mixed Hodge modules. By taking derived global sections, we obtain the morphism
\[
 \HO^*(V^{\langle C\rangle}/(G^{\langle C\rangle}\times T'),\varphi_{f_{\langle C\rangle}}\BoQ^{\vir})\rightarrow\HO^*(V^F/(G^F\times T'),\varphi_{f_F}\BoQ^{\vir})\otimes\SL^{\frac{d_F}{2}-\frac{d_{\langle C\rangle}}{2}-d_{C,F}}\,.
\]
Unraveling the respective formulas for $\BoQ^{\vir}$, we obtain
\[
 \HO^*(V^{\langle C\rangle}/(G^{\langle C\rangle}\times T'),\varphi_{f_{\langle C\rangle}}\BoQ)\rightarrow\HO^*(V^F/(G^F\times T'),\varphi_{f_F}\BoQ)\otimes\SL^{-d_{C,F}}\,.
\]
From the definition of $\CH_{G,V,f,F}^{T'}$ at the beginning of \S\ref{subsection:criticalcohomology}, this is a morphism
\[
 \CH_{G,V,f,\langle C\rangle}^{T'}\rightarrow\CH_{G,V,f,F}^{T'}\otimes\SL^{-d_{C,F}}\,.
\]
Written in this way, with the Tate twist, the induction morphism respects the cohomological grading on both sides (recall that $\SL$ is in cohomological degree $2$), and respects the mixed Hodge structures.

\subsection{Associativity}
\label{subsection:associativity}

\begin{proposition}
\label{proposition:associativity3d}
For any cells $C,C'\in\FC$ and $F\in\FF$ such that $C\preceq C'\preceq F$, we have $\Ind_{C}^F=\Ind_{C'}^{F}\circ\Ind_{C}^{\langle C'\rangle}$.
\end{proposition}
\begin{proof}
The proof is rather classical, and follows from base-change in the diagram of equivariant morphisms
% https://q.uiver.app/#q=WzAsNyxbMiwwLCJWXntDPjAsRn0iXSxbMSwxLCJWXntDPjAsXFxsYW5nbGUgQydcXHJhbmdsZX0iXSxbMCwyLCJWXntcXGxhbmdsZSBDXFxyYW5nbGV9Il0sWzMsMSwiVl57Vic+MCxGfSJdLFs0LDIsIlZeRiJdLFsyLDIsIlZee1xcbGFuZ2xlIEMnXFxyYW5nbGV9Il0sWzAsMV0sWzAsMV0sWzEsMiwicV97QyxcXGxhbmdsZSBDJ1xccmFuZ2xlfSIsMl0sWzAsM10sWzMsNCwicF97QycsRn0iXSxbMSw1LCJwX3tDLFxcbGFuZ2xlIEMnXFxyYW5nbGV9Il0sWzMsNSwicV97QycsRn0iLDJdLFswLDUsIiIsMSx7InN0eWxlIjp7Im5hbWUiOiJjb3JuZXIifX1dLFswLDIsInFfe0MnLEZ9IiwxXSxbMCw0LCJwX3tDLEZ9IiwxXV0=
\[\begin{tikzcd}
	&& {V^{C\geq 0,F}} \\
	{} & {V^{C\geq0,\langle C'\rangle}} && {V^{C'\geq0,F}} \\
	{V^{\langle C\rangle}} && {V^{\langle C'\rangle}} && {V^F}
	\arrow[from=1-3, to=2-2]
	\arrow[from=1-3, to=2-4]
	\arrow["{q_{C,F}}"{description}, from=1-3, to=3-1,bend right=55]
	\arrow["\lrcorner"{anchor=center, pos=0.125, rotate=-45}, draw=none, from=1-3, to=3-3]
	\arrow["{p_{C,F}}"{description}, from=1-3, to=3-5,bend left=55]
	\arrow["{q_{C,\langle C'\rangle}}"', from=2-2, to=3-1]
	\arrow["{p_{C,\langle C'\rangle}}", from=2-2, to=3-3]
	\arrow["{q_{C',F}}"', from=2-4, to=3-3]
	\arrow["{p_{C',F}}", from=2-4, to=3-5]
\end{tikzcd}\]
using the properness of the maps $p_{\dots}$, the smoothness of the maps $q_{\dots}$ and the fact that vanishing cycle functors commute with smooth pullback and proper pushforward.
\end{proof}

\subsection{Torus-equivariant vanishing cycles}
Let $V$ be a representation of a reductive group $G$ and $f\colon V\rightarrow \BoC$ a $G$-invariant function on $V$. We fix a maximal torus $T\subset G$ and denote by $\sfW$ the corresponding Weyl group.

\begin{lemma}
\label{lemma:Winvariantpart}
 We have $\HO_{G\times T'}^*(V,\varphi_f\BoQ)\cong\HO^*_{T\times T'}(V,\varphi_f\BoQ)^{\sfW}$.
\end{lemma}
\begin{proof}
We consider the diagram
% https://q.uiver.app/#q=WzAsMyxbMCwwLCJWL1QiXSxbMSwwLCJWL05fRyhUKSJdLFsyLDAsIlYvRyJdLFswLDEsInUiXSxbMSwyLCJ2Il0sWzAsMiwidyIsMl1d
\[\begin{tikzcd}
	{V/(T\times T')} & {V/(N_G(T)\times T')} & {V/(G\times T')}
	\arrow["u", from=1-1, to=1-2]
	\arrow["w", from=1-1, to=1-3,bend left=30]
	\arrow["v", from=1-2, to=1-3]\,.
\end{tikzcd}\]
The morphism $u$ is a $\sfW$-torsor. Therefore, we have an isomorphism
\[
 \HO^*_{T\times T'}(V,\varphi_f\BoQ)^{\sfW}\cong\HO^*_{N_G(T)\times T'}(V,\varphi_f\BoQ)\,.
\]
The fibers of $v$ are $G/N_G(T)\cong (G/T)/\sfW$ and so, the cohomology of the fibers of $v$ is $\BoQ\cong \HO^*(G/T,\BoQ)^{\sfW}$ by Proposition~\ref{proposition:cohflagvarieties}. Therefore, the adjunction morphism $\id\rightarrow v_*v^*$ is an isomorphism. When applied to $\varphi_f\BoQ_{V/(G\times T')}$, this gives $\HO^*_{N_G(T)\times T'}(V,\varphi_f\BoQ)\cong\HO^*_{G\times T'}(V,\varphi_f\BoQ)$. This concludes.
\end{proof}

Similarly, we have the following lemma whose proof is almost identical.
\begin{lemma}
 \label{lemma:Winvariantparabolic}
We let $\lambda\colon\BoG_{\rmm}\rightarrow G$ be a one-parameter subgroup and $\sfW^{\lambda}$ be the Weyl group of $G^{\lambda}$. Then, $\HO^*_{G^{\lambda\geq 0}\times T'}(V^{\lambda\geq 0},\varphi_f\BoQ)\cong\HO^*_{T\times T'}(V^{\lambda\geq 0},\varphi_f\BoQ)^{\sfW^{\lambda}}$.
\end{lemma}
\begin{proof}
The inclusion $G^{\lambda}\subset G^{\lambda\geq 0}$ induces a smooth morphism of quotient stacks with fiber $G^{\lambda\geq 0}/G^{\lambda}$, an affine space. Its pullback is therefore an isomorphism on critical cohomology. Applying Lemma~\ref{lemma:Winvariantpart} to the reductive group $G^{\lambda}$ gives the assertion.
\end{proof}

\subsection{Torus-equivariant induction}
\label{subsection:torusequinduction}

\subsubsection{Multiplication by Euler classes}
\label{subsection:multEulerclasses}
We let $G$ be a reductive group and $T'$ an auxiliary torus. Let $W$ be a representation of $G\times T'$. For $C\in\FC$ and $F\in\FF$ a cell and flat in the Coxeter complex of $(G,W)$, we define the $T'$-equivariant Euler class
\[
 \Eu_{W,C,F}^{T'}\coloneqq \prod_{\alpha\in \CW^{F}(W)\cap\CW^{C<0}(W)}\alpha\in\HO^*_{T\times T'}(\pt)\,,
\]
where the weights of $W$ are taken with respect to the $T\times T'$ action.
\begin{proposition}
\label{proposition:multEulerinjective}
 For any $C\in\FC$ and $F\in\FF$ such that $C\preceq F$, the multiplications by $\Eu_{V,C,F}^{T'}$ and $\Eu_{\Fg,C,F}^{T'}$ are injective on $\CH_{G,V,f,\langle C\rangle}^{T'}$.
\end{proposition}
\begin{proof}
This is the usual Atiyah--Bott argument; compare \cite[Proposition~4.1]{davison2017critical}. Choose a cocharacter $\lambda\in C$. If $\lambda=0$, both Euler classes are $1$ and there is nothing to prove. Otherwise, its image $S\subset G^{\langle C\rangle}$ is a central one-dimensional torus acting trivially on $V^{\langle C\rangle}$.

There is a connected reductive subgroup $H\subset G^{\langle C\rangle}$ such that multiplication $S\times H\rightarrow G^{\langle C\rangle}$ is an isogeny. With rational coefficients, restriction along this isogeny gives
\[
 \CH_{G,V,f,\langle C\rangle}^{T'}
 \cong \HO^*_{H\times T'}(V^{\langle C\rangle},\varphi_f\BoQ)\otimes\BoQ[u],
 \qquad \deg u=2.
\]
For $W=V$ or $\Fg$, every $S$-weight on $W^{C<0,F}$ is nonzero. The Euler class $\Eu_{W,C,F}^{T'}$, regarded as a polynomial in $u$, therefore has degree $\dim W^{C<0,F}$ and nonzero scalar leading coefficient. Multiplication by such a polynomial is injective on the displayed polynomial module, by comparing the highest powers of $u$. This proves both assertions.
\end{proof}

\subsubsection{Torus induction}
In this section, we define the induction using the action of the maximal torus $T\subset G$ and summing over the Weyl group $\sfW$. This allows a \emph{shuffle description} of the induction \eqref{equation:inductionabsolute}. For any $C\in\FC$ and $F\in\FF$ such that $C\preceq F$, we have an induction diagram
\begin{equation}
\label{equation:torusdiagram}
 V^{\langle C\rangle}/(T\times T')\xleftarrow{q_{C,F}} V^{C\geq 0,F}/(T\times T')\xrightarrow{p_{C,F}} V^F/(T\times T').
\end{equation}
By the same arguments as in \S\ref{subsection:parabolic_induction}, we obtain the induction morphism
\[
 \overline{\Ind}_{C}^F\colon \CH_{T,V,f,\langle C\rangle}^{T'}\rightarrow \CH_{T,V,f,F}^{T'}\,.
\]
We define the torus induction morphism
\begin{equation}
\label{equation:torusinductionmorphism}
 \tilde{\Ind}_{C}^{F}\colon \CH_{T,V,f,\langle C\rangle}^{T'}\rightarrow \CH_{T,V,f,F}^{T'}[\Eu_{\Fg,C,F}^{T',-1}]
\end{equation}
as $\frac{1}{\Eu_{\Fg,C,F}^{T'}}\overline{\Ind}_{C}^F$, using the $\HO^*_{T\times T'}(\pt)$-module structure on $\CH_{T,V,f,F}^{T'}$. Last, we define
\[
 \Ind_{C}'^F\colon \CH_{G,V,f,\langle C\rangle}^{T'}\rightarrow\CH_{G,V,f,F}^{T',\loc}
\]
as the Weyl sum $\Ind_C'^F\coloneqq\left(\sum_{w\in \sfW^F/\sfW^{\langle C\rangle}}w\right)(\tilde{\Ind}_{C}^F)_{|\CH_{G,V,f,\langle C\rangle}^{T'}}$, where $(\tilde{\Ind}_{C}^F)_{|\CH_{G,V,f,\langle C\rangle}^{T'}}$ denotes the restriction of $\tilde{\Ind}_{C}^F$ to the $\sfW^{\langle C\rangle}$-invariant part in $\CH_{T,V,f,\langle C\rangle}^{T'}$, which coincides with $\CH_{G,V,f,\langle C\rangle}^{T'}$ using Lemma~\ref{lemma:Winvariantpart}, and $\CH_{G,V,f,F}^{T',\loc}$ denotes the localization of $\CH_{G,V,f,F}^{T'}$ at the invariant product $\prod_{w\in\sfW^F}w\cdot\Eu_{\Fg,C,F}^{T'}$. This localization is required by the denominators appearing in the torus induction morphism \eqref{equation:torusinductionmorphism}. We let $\xi_4'\colon \CH_{G,V,f,F}^{T'}\rightarrow\CH_{G,V,f,F}^{T',\loc}$ be the localization morphism.

\begin{proposition}
\label{proposition:toruscomparison3d}

 For any $C\in\FC$ and $F\in\FF$ such that $C\preceq F$, the morphisms $\xi_4'\circ \Ind_{C}^F,\Ind_{C}'^{F}\colon \CH_{G,V,f,\langle C\rangle}^{T'}\rightarrow\CH_{G,V,f,F}^{T',\loc}$ coincide.
\end{proposition}
\begin{proof}
 The situation is analogous to that of \cite[Proposition~4.6]{davison2017critical}. There is a commutative diagram
% https://q.uiver.app/#q=WzAsOCxbMCwxLCJcXEhPXiooVl5GLyhHXntDXFxnZXEgMCxGfVxcdGltZXMgVCcpLFxcdmFycGhpX2YpIl0sWzAsMywiXFxDSF97RyxWLGYsXFxsYW5nbGUgQ1xccmFuZ2xlfV57VCd9Il0sWzEsMywiKFxcQ0hfe1QsRixmLFxcbGFuZ2xlIENcXHJhbmdsZX1ee1QnfSlee1dee1xcbGFuZ2xlIENcXHJhbmdsZX19Il0sWzAsMiwiXFxIT14qKFZee0NcXGdlcSAwLEZ9LyhHXntDXFxnZXEgMCxGfVxcdGltZXMgVCcpLFxcdmFycGhpX2YpIl0sWzAsMCwiXFxDSF97RyxWLGYsRn1ee1QnfSJdLFsxLDIsIlxcSE9eKihWXntDXFxnZXEgMCxGfS8oVFxcdGltZXMgVCcpLFxcdmFycGhpX2YpXntXXntcXGxhbmdsZSBDXFxyYW5nbGV9fSJdLFsxLDEsIlxcSE9eKihWXkYvKFRcXHRpbWVzIFQnKSxcXHZhcnBoaV9mKV57V157XFxsYW5nbGUgQ1xccmFuZ2xlfX0iXSxbMSwwLCIoXFxDSF97VCxWLGYsRn1ee1QnLFxcbG9jfSlee1deRn0iXSxbMSwyLCJcXHhpXzEiXSxbMSwzXSxbMywwXSxbMCw0XSxbNiw3XSxbNSw2XSxbMiw1XSxbMyw1LCJcXHhpXzIiXSxbMCw2LCJcXHhpXzMiXSxbNCw3LCJcXHhpXzQiXSxbMSw0LCJcXEluZF97Q31eRiJdLFsyLDcsIlxcSW5kX0MnXkYiXV0=
\[\begin{tikzcd}[row sep = 4em, column sep=10em]
	{\CH_{G,V,f,F}^{T'}} & {(\CH_{T,V,f,F}^{T',\loc})^{\sfW^F}} \\
	{\HO^*(V^F/(G^{C\geq 0,F}\times T'),\varphi_f)} & {\HO^*(V^F/(T\times T'),\varphi_f)^{\sfW^{\langle C\rangle}}} \\
	{\HO^*(V^{C\geq 0,F}/(G^{C\geq 0,F}\times T'),\varphi_f)} & {\HO^*(V^{C\geq 0,F}/(T\times T'),\varphi_f)^{\sfW^{\langle C\rangle}}} \\
	{\CH_{G,V,f,\langle C\rangle}^{T'}} & {(\CH_{T,V,f,\langle C\rangle}^{T'})^{\sfW^{\langle C\rangle}}}
	\arrow["{\xi_4}", from=1-1, to=1-2]
	\arrow[from=2-1, to=1-1]
	\arrow["{\xi_3}", from=2-1, to=2-2]
	\arrow["\sum_{w\in\sfW^F/\sfW^{\langle C\rangle}}w\left(\frac{1}{\Eu_{\Fg,C,F}^{T'}}\cdot -\right)",from=2-2, to=1-2]
	\arrow[from=3-1, to=2-1]
	\arrow["{\xi_2}", from=3-1, to=3-2]
	\arrow[from=3-2, to=2-2]
	\arrow["{\Ind_{C}^F}", from=4-1, to=1-1, bend left=90]
	\arrow[from=4-1, to=3-1]
	\arrow["{\xi_1}", from=4-1, to=4-2]
	\arrow["{\Ind_C'^F}"', from=4-2, to=1-2,bend right = 90]
	\arrow[from=4-2, to=3-2]
\end{tikzcd}\]
where the horizontal arrow $\xi_1$ is defined as in Lemma~\ref{lemma:Winvariantpart}, the morphism $\xi_4$ is the postcomposition of the morphism obtained by Lemma~\ref{lemma:Winvariantpart} with the localization morphism and $\xi_2,\xi_3$ are defined analogously, see Lemma~\ref{lemma:Winvariantparabolic}. The vertical columns are the induction maps. We study the commutativity of the three squares one by one from the bottom to the top.

\textbf{Square 1:}
% https://q.uiver.app/#q=WzAsNSxbMCwzLCJcXENIX3tHLFYsZixcXGxhbmdsZSBDXFxyYW5nbGV9XntUJ30iXSxbMSwzLCIoXFxDSF97VCxGLGYsXFxsYW5nbGUgQ1xccmFuZ2xlfV57VCd9KV57V157XFxsYW5nbGUgQ1xccmFuZ2xlfX0iXSxbMCwyLCJcXEhPXiooVl57Q1xcZ2VxIDAsRn0vKEdee0NcXGdlcSAwLEZ9XFx0aW1lcyBUJyksXFx2YXJwaGlfZikiXSxbMCwwXSxbMSwyLCJcXEhPXiooVl57Q1xcZ2VxIDAsRn0vKFRcXHRpbWVzIFQnKSxcXHZhcnBoaV9mKV57V157XFxsYW5nbGUgQ1xccmFuZ2xlfX0iXSxbMCwxLCJcXHhpXzEiXSxbMCwyLCJhIl0sWzEsNCwiYiIsMl0sWzIsNCwiXFx4aV8yIl1d
\[\begin{tikzcd}
	{\HO^*(V^{C\geq 0,F}/(G^{C\geq 0,F}\times T'),\varphi_f)} & {\HO^*(V^{C\geq 0,F}/(T\times T'),\varphi_f)^{\sfW^{\langle C\rangle}}} \\
	{\CH_{G,V,f,\langle C\rangle}^{T'}} & {(\CH_{T,V,f,\langle C\rangle}^{T'})^{\sfW^{\langle C\rangle}}}
	\arrow["{\xi_2}", from=1-1, to=1-2]
	\arrow["a", from=2-1, to=1-1]
	\arrow["{\xi_1}", from=2-1, to=2-2]
	\arrow["b"', from=2-2, to=1-2]
\end{tikzcd}\]
This square is induced by smooth pullbacks in the commutative square
\[
 % https://q.uiver.app/#q=WzAsNCxbMCwwLCJWXntDXFxnZXEgMCxGfS8oR157Q1xcZ2VxIDAsRn1cXHRpbWVzIFQnKSJdLFswLDEsIlZee1xcbGFuZ2xlIENcXHJhbmdsZX0vKEdee1xcbGFuZ2xlIENcXHJhbmdsZX1cXHRpbWVzIFQnKSJdLFsxLDEsIlZee1xcbGFuZ2xlIENcXHJhbmdsZX0vKFRcXHRpbWVzIFQnKSJdLFsxLDAsIlZee0NcXGdlcSAwLEZ9LyhUXFx0aW1lcyBUJykiXSxbMiwxXSxbMCwxXSxbMywwXSxbMywyXV0=
\begin{tikzcd}
	{V^{C\geq 0,F}/(G^{C\geq 0,F}\times T')} & {V^{C\geq 0,F}/(T\times T')} \\
	{V^{\langle C\rangle}/(G^{\langle C\rangle}\times T')} & {V^{\langle C\rangle}/(T\times T')}
	\arrow[from=1-1, to=2-1]
	\arrow[from=1-2, to=1-1]
	\arrow[from=1-2, to=2-2]
	\arrow[from=2-2, to=2-1]
\end{tikzcd}
\]
and so commutes by compatibility of vanishing cycles with smooth pullbacks.

\textbf{Square 2:}
\[
 % https://q.uiver.app/#q=WzAsNCxbMCwwLCJcXEhPXiooVl5GLyhHXntDXFxnZXEgMCxGfVxcdGltZXMgVCcpLFxcdmFycGhpX2YpIl0sWzAsMSwiXFxIT14qKFZee0NcXGdlcSAwLEZ9LyhHXntDXFxnZXEgMCxGfVxcdGltZXMgVCcpLFxcdmFycGhpX2YpIl0sWzEsMSwiXFxIT14qKFZee0NcXGdlcSAwLEZ9LyhUXFx0aW1lcyBUJyksXFx2YXJwaGlfZilee1dee1xcbGFuZ2xlIENcXHJhbmdsZX19Il0sWzEsMCwiXFxIT14qKFZeRi8oVFxcdGltZXMgVCcpLFxcdmFycGhpX2YpXntXXntcXGxhbmdsZSBDXFxyYW5nbGV9fSJdLFsxLDAsImMiXSxbMiwzLCJkIiwyXSxbMSwyLCJcXHhpXzIiXSxbMCwzLCJcXHhpXzMiXV0=
\begin{tikzcd}
	{\HO^*(V^F/(G^{C\geq 0,F}\times T'),\varphi_f)} & {\HO^*(V^F/(T\times T'),\varphi_f)^{\sfW^{\langle C\rangle}}} \\
	{\HO^*(V^{C\geq 0,F}/(G^{C\geq 0,F}\times T'),\varphi_f)} & {\HO^*(V^{C\geq 0,F}/(T\times T'),\varphi_f)^{\sfW^{\langle C\rangle}}}
	\arrow["{\xi_3}", from=1-1, to=1-2]
	\arrow["c", from=2-1, to=1-1]
	\arrow["{\xi_2}", from=2-1, to=2-2]
	\arrow["d"', from=2-2, to=1-2]
\end{tikzcd}
\]
This square is induced by base-change in the following Cartesian square in which vertical maps are closed immersions and horizontal maps smooth
\[
 % https://q.uiver.app/#q=WzAsNCxbMCwxLCJWXntDXFxnZXEgMCxGfS8oR157Q1xcZ2VxIDAsIEZ9XFx0aW1lcyBUJykiXSxbMSwxLCJWXntDXFxnZXEgMCxGfS8oVFxcdGltZXMgVCcpIl0sWzAsMCwiVl5GLyhHXntDXFxnZXEgMCxGfVxcdGltZXMgVCcpIl0sWzEsMCwiVl5GLyhUXFx0aW1lcyBUJykiXSxbMCwyXSxbMywyXSxbMSwwXSxbMSwzXSxbMSwyLCIiLDEseyJzdHlsZSI6eyJuYW1lIjoiY29ybmVyLWludmVyc2UifX1dXQ==
\begin{tikzcd}
	{V^F/(G^{C\geq 0,F}\times T')} & {V^F/(T\times T')} \\
	{V^{C\geq 0,F}/(G^{C\geq 0, F}\times T')} & {V^{C\geq 0,F}/(T\times T')}
	\arrow[from=1-2, to=1-1]
	\arrow[from=2-1, to=1-1]
	\arrow["\ulcorner"{anchor=center, pos=0.125, rotate=0}, draw=none, from=2-2, to=1-1]
	\arrow[from=2-2, to=1-2]
	\arrow[from=2-2, to=2-1]
\end{tikzcd}\,.
\]

\textbf{Square 3:}
\[
% https://q.uiver.app/#q=WzAsNCxbMCwxLCJcXEhPXiooVl5GLyhHXntDXFxnZXEgMCxGfVxcdGltZXMgVCcpLFxcdmFycGhpX2YpIl0sWzAsMCwiXFxDSF97RyxWLGYsRn1ee1QnfSJdLFsxLDEsIlxcSE9eKihWXkYvKFRcXHRpbWVzIFQnKSxcXHZhcnBoaV9mKV57V157XFxsYW5nbGUgQ1xccmFuZ2xlfX0iXSxbMSwwLCIoXFxDSF97VCxWLGYsRn1ee1QnLFxcbG9jfSlee1deRn0iXSxbMCwxLCJlIl0sWzIsMywiZiIsMl0sWzAsMiwiXFx4aV8zIl0sWzEsMywiXFx4aV80Il1d
\begin{tikzcd}
	{\CH_{G,V,f,F}^{T'}} & {(\CH_{T,V,f,F}^{T',\loc})^{\sfW^F}} \\
	{\HO^*(V^F/(G^{C\geq 0,F}\times T'),\varphi_f)} & {\HO^*(V^F/(T\times T'),\varphi_f)^{\sfW^{\langle C\rangle}}}
	\arrow["{\xi_4}", from=1-1, to=1-2]
	\arrow["e", from=2-1, to=1-1]
	\arrow["{\xi_3}", from=2-1, to=2-2]
	\arrow["\sum_{w\in \sfW^{F}/\sfW^{\langle C\rangle}}w\cdot\left(\frac{1}{\Eu_{\Fg,C,F}^{T'}}\cdot -\right)"', from=2-2, to=1-2]
\end{tikzcd}
\]
We consider the diagram
\[
 % https://q.uiver.app/#q=WzAsNSxbMSwxLCJcXEZYIl0sWzIsMSwiVl5GLyhHXntDXFxnZXEgMH1cXHRpbWVzIFQnKSJdLFsyLDIsIlZeRi8oR15GXFx0aW1lcyBUJykiXSxbMSwyLCJWXkYvKE5fRyhUKVxcdGltZXMgVCcpIl0sWzAsMCwiVl5GLyhOX3tHXntcXGxhbmdsZSBDXFxyYW5nbGV9fVxcdGltZXMgVCcpIl0sWzAsMywieiJdLFszLDIsInkiXSxbMSwyLCJ4Il0sWzAsMSwidyIsMl0sWzAsMiwiIiwxLHsic3R5bGUiOnsibmFtZSI6ImNvcm5lciJ9fV0sWzQsMywidiIsMl0sWzQsMSwidSJdLFs0LDAsImkiLDFdXQ==
\begin{tikzcd}
	{V^F/(N_{G^{\langle C\rangle}}(T)\times T')} \\
	& \FX & {V^F/(G^{C\geq 0,F}\times T')} \\
	& {V^F/(N_{G^F}(T)\times T')} & {V^F/(G^F\times T')}
	\arrow["\imath"{description}, from=1-1, to=2-2]
	\arrow["u", from=1-1, to=2-3,bend left=20]
	\arrow["v"', from=1-1, to=3-2,bend right=20]
	\arrow["t", from=2-2, to=2-3]
	\arrow["z", from=2-2, to=3-2]
	\arrow["\lrcorner"{anchor=center, pos=0.125}, draw=none, from=2-2, to=3-3]
	\arrow["x", from=2-3, to=3-3]
	\arrow["y", from=3-2, to=3-3]
\end{tikzcd}
\]
We need to compare the operations $y^*x_*$ and $v_*u^*$ in vanishing cycle cohomology. We introduced the stack $\FX$ making the square Cartesian.

The morphism $\imath$ is a closed embedding with normal bundle of Euler class $e_C\coloneqq\Eu_{\Fg,C,F}^{T'}$, modelled on the morphism $\frac{1}{N_{G^{\langle C\rangle}}(T)\times T'}\rightarrow (N_{G^F}(T)\times T')\backslash (G^F\times T')/(G^{C\geq 0,F}\times T')$. Set
\[
 \delta\coloneqq\prod_{w\in\sfW^F}w\cdot e_C.
\]
After restricting equivariance to $T\times T'$, the embedding $\imath$ is the inclusion of the product of $V^F$ with the $T$-fixed flags in $G^F/G^{C\geq0,F}$. The coefficient complex is pulled back from $V^F$. On each simple factor with a nontrivial flag variety, the roots outside the Levi meet every Weyl orbit of roots: each such orbit spans the irreducible root space, whereas the roots of a proper Levi lie in a proper subspace. Thus $\delta$ inverts all roots on the factors contributing to the flag variety. Equivariant localization with constructible coefficients, \cite[Theorem~6.2]{goresky1998equivariant}, makes $\imath^*$ an isomorphism after inverting $\delta$. Passing to $\sfW^F$-invariants gives the same conclusion on the normalizer quotients. By critical self-intersection, its inverse is $\imath_*\circ(e_C^{-1}\cdot-)$.

Since $e_C$ is a class on the source of $\imath$, multiplication by $e_C^{-1}$ takes place before the pushforward. Therefore
\[
 \begin{aligned}
y^*x_*&=z_*t^*\\
&=z_*\imath_*\left(\frac{1}{\Eu_{\Fg,C,F}^{T'}}\imath^*t^*\right)\\
&=v_*\left(\frac{1}{\Eu_{\Fg,C,F}^{T'}}u^*\right)\\
&=\sum_{w\in \sfW^F/\sfW^{\langle C\rangle}}w\cdot\left(\frac{1}{\Eu_{\Fg,C,F}^{T'}}u^*\right)\,.
 \end{aligned}
\]
This concludes.

\end{proof}

\subsection{Explicit formula for the induction for vanishing potential}
\label{subsection:explicitformula}
For any $C\in\FC$ and $F\in\FF$ such that $C\preceq F$, we define the induction kernel
\[
 k_{C,F}\coloneqq\frac{\Eu_{V,C,F}^{T'}}{\Eu_{\Fg,C,F}^{T'}}\,.
\]
\begin{proposition}
 The induction morphism $\Ind_{C}^F\colon \CH_{G,V,\langle C\rangle}^{T'}\cong(\HO^*_{T\times T'})^{\sfW^{\langle C\rangle}}\rightarrow\CH_{G,V,F}^{T'}\cong (\HO^*_{T\times T'})^{\sfW^F}$ is given by the formula
 \[
  \Ind_{C}^F(f)=\sum_{w\in \sfW^{F}/\sfW^{\langle C\rangle}}w\cdot(fk_{C,F})\,.
 \]
\end{proposition}
\begin{proof}
 This is a standard calculation of Euler classes of normal bundles.
\end{proof}

\subsection{Equivariance for the Weyl group}
\begin{lemma}
\label{lemma:Weylgroupequivariance3d}
 For any $C\in\FC$ and $F\in\FF$ such that $C\preceq F$, $w\Ind_{C}^{F}=\Ind_{w\cdot C}^{w\cdot F}w$.
\end{lemma}
\begin{proof}
 We let $\dot{w}\in N_G(T)$ be a lift of $w$. Then, $\dot{w}$ gives an isomorphism of induction diagrams
 \[
  % https://q.uiver.app/#q=WzAsNixbMCwwLCJWXntcXGxhbmdsZSBDXFxyYW5nbGV9LyhHXntcXGxhbmdsZSBDXFxyYW5nbGV9XFx0aW1lcyBUJykiXSxbMSwwLCJWXntDXFxnZXEgMCxGfS8oR157Q1xcZ2VxIDAsRn1cXHRpbWVzIFQnKSJdLFsyLDAsIlZeRi97KEdeRlxcdGltZXMgVCcpfSJdLFswLDEsIlZee1xcbGFuZ2xlIHdcXGNkb3QgQ1xccmFuZ2xlfS8oR157XFxsYW5nbGUgd1xcY2RvdCBDXFxyYW5nbGV9XFx0aW1lcyBUJykiXSxbMSwxLCJWXnt3XFxjZG90IENcXGdlcSAwLHdcXGNkb3QgRn0vKEdee3dcXGNkb3QgQ1xcZ2VxIDAsd1xcY2RvdCBGfVxcdGltZXMgVCcpIl0sWzIsMSwiVl57d1xcY2RvdCBGfS97KEdee3dcXGNkb3QgRn1cXHRpbWVzIFQnKX0iXSxbMiw1LCJcXGRvdHt3fSJdLFsxLDQsIlxcZG90e3d9Il0sWzAsMywiXFxkb3R7d30iLDJdLFsxLDAsInFfe0MsRn0iLDJdLFsxLDIsInBfe0MsRn0iXSxbNCwzLCJxX3t3XFxjZG90IEMsd1xcY2RvdCBGfSJdLFs0LDUsInBfe3dcXGNkb3QgQyx3XFxjZG90IEZ9IiwyXV0=
\begin{tikzcd}
	{V^{\langle C\rangle}/(G^{\langle C\rangle}\times T')} & {V^{C\geq 0,F}/(G^{C\geq 0,F}\times T')} & {V^F/{(G^F\times T')}} \\
	{V^{\langle w\cdot C\rangle}/(G^{\langle w\cdot C\rangle}\times T')} & {V^{w\cdot C\geq 0,w\cdot F}/(G^{w\cdot C\geq 0,w\cdot F}\times T')} & {V^{w\cdot F}/{(G^{w\cdot F}\times T')}}
	\arrow["{\dot{w}}"', from=1-1, to=2-1]
	\arrow["{q_{C,F}}"', from=1-2, to=1-1]
	\arrow["{p_{C,F}}", from=1-2, to=1-3]
	\arrow["{\dot{w}}", from=1-2, to=2-2]
	\arrow["{\dot{w}}", from=1-3, to=2-3]
	\arrow["{q_{w\cdot C,w\cdot F}}", from=2-2, to=2-1]
	\arrow["{p_{w\cdot C,w\cdot F}}"', from=2-2, to=2-3]
\end{tikzcd}
 \]
It follows immediately that $w\Ind_{C}^F=\Ind_{w\cdot C}^{w\cdot F}w$.
\end{proof}

\section{The \texorpdfstring{$2d$}{2d} induction system}
\label{section:2dinductionsystem}
\subsection{Zero-locus and stacks}

\subsubsection{Function}

We let $V,W$ be representations of a reductive group $G$ and $\mu\colon V\rightarrow W^*$ be a $G$-equivariant map. We use the Coxeter complex associated to $V\oplus W$. For any $C\in\FC$ and $F\in\FF$ such that $C\preceq F$, we denote by $\mu_F\colon V^F\rightarrow (W^*)^F$ the induced function. We write $W^{*,F}=W^{*,C<0,F}\oplus W^{*,\langle C\rangle,F}\oplus W^{*,C>0,F}$. We decompose $\mu_F=\mu_{C<0,F}+\mu_{\langle C\rangle,F}+\mu_{C>0,F}$ accordingly.

\begin{proposition}
\label{proposition:admissiblefunction}
We have $\mu_F\circ p_{C,F}=\mu_{\langle C\rangle}\circ q_{C,F}+\mu_{C>0,F}\circ p_{C,F}$.
\end{proposition}
\begin{proof}
 Choose a cocharacter $\lambda\in C$. If $x\in V^{C\geq0,F}$, equivariance of $\mu$ gives $\mu_F(x)\in W^{*,C\geq0,F}$ and
\[
 \lim_{t\to0}\lambda(t)\cdot\mu_F(x)=\mu_F\bigl(\lim_{t\to0}\lambda(t)\cdot x\bigr)=\mu_{\langle C\rangle}(q_{C,F}(x)).
\]
The negative-weight component is therefore zero, and the zero-weight component is $\mu_{\langle C\rangle}\circ q_{C,F}$. This proves the formula.
\end{proof}

\subsubsection{Stacks and moduli spaces}
For any $F\in \FF$, we define $\FM_{G,V,W,\mu,F}^{T'}\coloneqq \mu_F^{-1}(0)/(G^F\times T')$. For any $C\in \FC$ and $F\in\FF$ such that $C\preceq F$, we define $\FM_{G,V,W,\mu,C,F}^{T'}\coloneqq (\mu_F^{-1}(0)\cap V^{C\geq 0,F})/(G^{C\geq 0,F}\times T')$. We define $\CM_{G,V,W,\mu,F}^{T'}\coloneqq (\mu_F^{-1}(0)\cms G^F)/ T'$ and we let $\pi_F\colon \FM_{G,V,W,\mu,F}^{T'}\rightarrow\CM_{G,V,W,\mu,F}^{T'}$ be the good moduli space map.

\subsection{Parabolic induction}
\label{subsection:parabolicinduction2d}
We let $V,W$ be representations of a reductive group $G$, and let $\mu\colon V\to W^*$ be $G\times T'$-equivariant. Fix $C\preceq F$, and abbreviate
\[
 L=G^{\langle C\rangle},\quad P=G^{C\geq0,F},\quad
 U=V^{C\geq0,F},\quad V_0=V^{\langle C\rangle},\quad E=(W^*)^{C>0,F},
\]
\[
 Z_0=\mu_{\langle C\rangle}^{-1}(0),\qquad
 Z_+=\mu_F^{-1}(0)\cap U,\qquad Z=\mu_F^{-1}(0).
\]
The map
\[
 h_{C,F}\colon U\longrightarrow V_0\times E,
 \qquad v\longmapsto(q_{C,F}(v),\mu_{C>0,F}(v))
\]
is $L\times T'$-equivariant. The zero-weight component of $\mu_F|_U$ equals $\mu_{\langle C\rangle}\circ q_{C,F}$, and its negative-weight component vanishes. Therefore the following square, with $z\mapsto(z,0)$ on the left, is Cartesian:
\begin{equation}
\label{equation:inductiondiagram2d}
\begin{tikzcd}
 {Z_0/(L\times T')} & {Z_+/(L\times T')}\\
 {(V_0\times E)/(L\times T')} & {U/(L\times T')}
 \arrow[from=1-1,to=2-1]
 \arrow["q^L_{C,F}"',from=1-2,to=1-1]
 \arrow[from=1-2,to=2-2]
 \arrow["h_{C,F}"',from=2-2,to=2-1]
 \arrow["\lrcorner"{anchor=center,pos=0.125,rotate=-90},draw=none,from=1-2,to=2-1]
\end{tikzcd}
\end{equation}
Both lower stacks are smooth. Set
\[
 a_{C,F}=\dim V^{C>0,F}-\dim(W^*)^{C>0,F},
 \qquad u_{C,F}=\dim(P/L)=\dim\Fg^{C>0,F}.
\]
The refined pullback of \S\ref{subsection:virtualpullbackBM} is a morphism
\[
 \BD\BoQ_{Z_0/(L\times T')}
 \longrightarrow(q^L_{C,F})_*\BD\BoQ_{Z_+/(L\times T')}
 \otimes\SL^{a_{C,F}}.
\]

We next use the morphisms
\[
 r_{C,F}\colon Z_+/(L\times T')\longrightarrow Z_+/(P\times T'),
 \qquad p_{C,F}\colon Z_+/(P\times T')\longrightarrow Z/(G^F\times T').
\]
The first is smooth with fiber $P/L\cong\BoA^{u_{C,F}}$, and the second is proper and representable. The projection $q_{C,F}\colon Z_+/(P\times T')\to Z_0/(L\times T')$ is well-defined and $q^L_{C,F}=q_{C,F}r_{C,F}$. Smooth duality and the contractibility of the fibers give the canonical isomorphism
\[
 (r_{C,F})_*\BD\BoQ_{Z_+/(L\times T')}
 \cong\BD\BoQ_{Z_+/(P\times T')}\otimes\SL^{-u_{C,F}}.
\]
Thus the refined pullback descends to
\begin{equation}
\label{equation:pullback2d}
 \BD\BoQ_{\FM_{G,V,W,\mu,\langle C\rangle}^{T'}}
 \longrightarrow(q_{C,F})_*\BD\BoQ_{\FM_{G,V,W,\mu,C,F}^{T'}}
 \otimes\SL^{d'_{C,F}},
\end{equation}
where
\[
 d'_{C,F}=a_{C,F}-u_{C,F}
 =\dim V^{C>0,F}-\dim\Fg^{C>0,F}-\dim(W^*)^{C>0,F}.
\]
Proper pushforward supplies
\begin{equation}
\label{equation:pushforward2d}
 (p_{C,F})_*\BD\BoQ_{\FM_{G,V,W,\mu,C,F}^{T'}}
 \longrightarrow\BD\BoQ_{\FM_{G,V,W,\mu,F}^{T'}}.
\end{equation}
On Borel--Moore homology, with twists suppressed, the resulting induction is
\begin{equation}
\label{equation:2d-levi-induction}
 \Ind_C^F=(p_{C,F})_*\bigl(r_{C,F}^{!}\bigr)^{-1}h_{C,F,\mathrm{ref}}^{!}.
\end{equation}
Here $r_{C,F}^{!}$ denotes smooth pullback, including its dimension shift.

The ambient square is deliberately formed with Levi equivariance. In general $h_{C,F}$ is not $P$-equivariant for the action on its target through $P\to L$: the positive component of $\mu$ can change under the unipotent radical through its zero component. The quotient by $P$ enters only after restricting to $Z_+$, and it is essential for the properness of $p_{C,F}$.

We also consider the commutative diagram obtained from the zero-locus correspondence defined by $q_{C,F},p_{C,F}$ and the finite map $\imath_{\langle C\rangle,F}\colon\CM_{G,V,W,\mu,\langle C\rangle}^{T'}\coloneqq (\mu_{\langle C\rangle}^{-1}(0)\cms G^{\langle C\rangle})/T'\rightarrow\CM_{G,V,W,\mu,F}^{T'}$ between the good moduli spaces.
 % https://q.uiver.app/#q=WzAsNSxbMCwyLCJcXENNX3tHLFYsVyxcXG11LFxcbGFuZ2xlIENcXHJhbmdsZX1ee1QnfSJdLFswLDEsIlxcRk1fe0csVixXLFxcbXUsXFxsYW5nbGUgQ1xccmFuZ2xlfV57VCd9Il0sWzIsMSwiXFxGTV97RyxWLFcsXFxtdSxGfV57VCd9Il0sWzEsMCwiXFxGTV97RyxWLFcsXFxtdSxDLEZ9XntUJ30iXSxbMiwyLCJcXENNX3tHLFYsVyxcXG11LEZ9XntUJ30iXSxbMywxLCJxX3tDLEZ9IiwyXSxbMywyLCJwX3tDLEZ9Il0sWzAsNCwiXFxpbWF0aF97XFxsYW5nbGUgQ1xccmFuZ2xlLEZ9Il0sWzEsMCwiXFxwaV97XFxsYW5nbGUgQ1xccmFuZ2xlfSIsMl0sWzIsNCwiXFxwaV9GIl1d
\[\begin{tikzcd}
	& {\FM_{G,V,W,\mu,C,F}^{T'}} \\
	{\FM_{G,V,W,\mu,\langle C\rangle}^{T'}} && {\FM_{G,V,W,\mu,F}^{T'}} \\
	{\CM_{G,V,W,\mu,\langle C\rangle}^{T'}} && {\CM_{G,V,W,\mu,F}^{T'}}
	\arrow["{q_{C,F}}"', from=1-2, to=2-1]
	\arrow["{p_{C,F}}", from=1-2, to=2-3]
	\arrow["{\pi_{\langle C\rangle}}"', from=2-1, to=3-1]
	\arrow["{\pi_F}", from=2-3, to=3-3]
	\arrow["{\imath_{\langle C\rangle,F}}", from=3-1, to=3-3]
\end{tikzcd}\]

By composing $(\imath_*)(\pi_{\langle C\rangle})_*$ applied to \eqref{equation:pullback2d} with $(\pi_F)_*\otimes\SL^{d'_{C,F}}$ applied to \eqref{equation:pushforward2d}, we obtain the relative induction morphism
\[
 \Ind_{C}^F\colon (\imath_{\langle C\rangle,F})_*(\pi_{\langle C\rangle})_*\BD\BoQ_{\FM_{G,V,W,\mu,\langle C\rangle}^{T'}}\rightarrow (\pi_F)_*\BD\BoQ_{\FM_{G,V,W,\mu,F}^{T'}}\otimes\SL^{d'_{C,F}}\,.
\]
This is the sheafified cohomological Hall induction morphism. It is customary to give the constant sheaf on $\FM_{G,V,W,\mu,F}^{T'}$ the virtual twist relative to $BT'$, that is $\BoQ_{\FM_{G,V,W,\mu,F}^{T'}}^{\vir}\coloneqq\BoQ_{\FM_{G,V,W,\mu,F}^{T'}}\otimes\SL^{-d''_F/2}$ where
\[
 d''_F=\dim V^F-\dim W^F-\dim G^F\,.
\]
The sheafified cohomological Hall induction can then be rewritten
\begin{equation}
\label{equation:virtualtwisted2d}
 \Ind_{C}^F\colon (\imath_{\langle C\rangle,F})_*(\pi_{\langle C\rangle})_*\BD\BoQ_{\FM_{G,V,W,\mu,\langle C\rangle}^{T'}}^{\vir}\rightarrow (\pi_F)_*\BD\BoQ_{\FM_{G,V,W,\mu,F}^{T'}}^{\vir}\otimes\SL^{\frac{d''_{\langle C\rangle}}{2}+d'_{C,F}-\frac{d_F''}{2}}\,.
\end{equation}

By taking derived global sections, we obtain the induction morphism in Borel--Moore homology
\[
  \Ind_{C}^F\colon \CH_{G,V,W,\mu,\langle C\rangle}^{T'}\rightarrow \CH_{G,V,W,\mu,F}^{T'}\,.
\]

\begin{proposition}
 \label{proposition:preservationcohdegrees2d}
 When $V,W$ are both weakly symmetric representations of $G$, the induction morphisms preserve cohomological degrees when virtual twists are taken into account.
\end{proposition}
\begin{proof}
 When $V,W$ are weakly symmetric representations of $G$, then
 \[
  \frac{d''_{\langle C\rangle}}{2}+d'_{C,F}=\frac{d''_F}{2}
 \]
 and so the Tate twist in the r.h.s. of \eqref{equation:virtualtwisted2d} is trivial. This concludes.

\end{proof}

\subsection{Associativity}
These induction morphisms also satisfy a version of associativity.

\begin{proposition}
\label{proposition:associativity2d}
 For any $C,C'\in \FC$ and $F\in \FF$ such that $C\preceq C'\preceq F$, one has
 \[
  \Ind_{C}^F=\Ind_{C'}^{F}\circ\Ind_{C}^{\langle C'\rangle}\,.
 \]
\end{proposition}
\begin{proof}
We use the dimensional-reduction comparison of Proposition~\ref{proposition:comparison2d3dmultiplications}, whose proof below uses only the individual induction maps. Write $\varepsilon_{C,F}=(-1)^{\dim(W^*)^{C>0,F}}$. For $C\preceq C'\preceq F$, the weight decomposition gives
\[
 (W^*)^{C>0,F}=(W^*)^{C>0,\langle C'\rangle}\oplus(W^*)^{C'>0,F},
\]
and hence $\varepsilon_{C,F}=\varepsilon_{C,\langle C'\rangle}\varepsilon_{C',F}$. Applying dimensional reduction to the composite of $2d$ induction maps, the comparison proposition and critical associativity give
\[
\begin{aligned}
 \mathsf{dr}_F\Ind_{C'}^{F,2d}\Ind_C^{\langle C'\rangle,2d}
 &=\varepsilon_{C',F}\varepsilon_{C,\langle C'\rangle}
 \Ind_{C'}^{F,\mathrm{crit}}\Ind_C^{\langle C'\rangle,\mathrm{crit}}
 \mathsf{dr}_{\langle C\rangle}\\
 &=\varepsilon_{C,F}\Ind_C^{F,\mathrm{crit}}\mathsf{dr}_{\langle C\rangle}
 =\mathsf{dr}_F\Ind_C^{F,2d}.
\end{aligned}
\]
Since $\mathsf{dr}_F$ is an isomorphism, the assertion follows. The dimensions $a_{C,F}$ and $u_{C,F}$ in the Levi construction are additive in the same way, so the argument also respects the unsuppressed twists.
\end{proof}

\subsection{Equivariance for the Weyl group}
\begin{proposition}
 Let $C\in\FC$ and $F\in\FF$ be such that $C\preceq F$. Let $w\in \sfW$. Then, $w\cdot \Ind_{C}^F=\Ind_{w\cdot C}^{w\cdot F}w$.
\end{proposition}
\begin{proof}
It is immediate to see that if $\dot{w}\in N_G(T)$ is a lift of $w$, then the action of $\dot{w}$ transforms the induction diagram \eqref{equation:inductiondiagram2d} to the induction diagram for $w\cdot C, w\cdot F$, exactly as for the Weyl group equivariance of the critical induction (Lemma~\ref{lemma:Weylgroupequivariance3d}). Conjugation also identifies the smooth comparison $r_{C,F}$ and the proper map $p_{C,F}$ with their counterparts for $w\cdot C,w\cdot F$. Their pullbacks, inverse comparison isomorphisms and proper pushforwards are therefore intertwined. The result follows.
\end{proof}

\subsection{Torus-equivariant induction}
In the $2d$-situation, we can also recover the parabolic induction by first restricting the group $G$ to its maximal torus $T$, building the induction for $T$, and then summing the formula obtained over the Weyl group $\sfW$ of $G$. We present this approach here.

For any $F\in\FF$, we define $\CH_{T,V,W,\mu,F}^{T'}\coloneqq \HO^{\rmBM}_{T\times T'}(\mu_F^{-1}(0),\BoQ)$. The following proposition is the analogue in Borel--Moore homology of Lemma~\ref{lemma:Winvariantpart}.

\begin{proposition}
 We have $\CH_{G,V,W,\mu,F}^{T'}\cong(\CH_{T,V,W,\mu,F}^{T'})^{\sfW^F}$.
\end{proposition}
\begin{proof}
 This follows immediately from Proposition~\ref{proposition:weylgroupBorelMoore}.
\end{proof}

We consider the torus-equivariant induction diagram. Here $q'_{C,F}=h_{C,F}$ with equivariance restricted to $T\times T'$, $p'_{C,F}$ is the inclusion, and the vertical maps are closed immersions
\[
 % https://q.uiver.app/#q=WzAsNixbMiwxLCJWXkYvKEdeRlxcdGltZXMgVCcpIl0sWzEsMSwiVl57Q1xcZ2VxIDAsRn0vKEdee0NcXGdlcSAwLEZ9XFx0aW1lcyBUJykiXSxbMCwxLCJWXntcXGxhbmdsZSBDXFxyYW5nbGUsRn0vKEdee1xcbGFuZ2xlIENcXHJhbmdsZX1cXHRpbWVzIFQnKVxcdGltZXMgKFdeKilee0M+MCxGfSJdLFswLDAsIlxcbXVfe1xcbGFuZ2xlIENcXHJhbmdsZX1eey0xfSgwKS8oR157XFxsYW5nbGUgQ1xccmFuZ2xlfVxcdGltZXMgVCcpIl0sWzEsMCwiXFxtdV9GXnstMX0oMClcXGNhcCBWXntDXFxnZXEgMCxGfS8oR157Q1xcZ2VxIDB9XFx0aW1lcyBUJykiXSxbMiwwLCJcXG11X0Zeey0xfSgwKS8oR15GXFx0aW1lcyBUJykiXSxbMSwwLCJwJ197QyxGfSJdLFsxLDIsInEnX3tDLEZ9IiwyXSxbMywyLCJcXGltYXRoX3tcXGxhbmdsZSBDXFxyYW5nbGV9IiwyXSxbNCw1LCJwX3tDLEZ9Il0sWzUsMCwiXFxpbWF0aF9GIl0sWzQsMSwiXFxpbWF0aF97Q1xcZ2VxIDAsRn0iLDJdLFs0LDMsInFfe0MsRn0iLDJdLFs0LDIsIiIsMSx7InN0eWxlIjp7Im5hbWUiOiJjb3JuZXIifX1dXQ==
\begin{tikzcd}
	{\mu_{\langle C\rangle}^{-1}(0)/(T\times T')} & {(\mu_F^{-1}(0)\cap V^{C\geq 0,F})/(T\times T')} & {\mu_F^{-1}(0)/(T\times T')} \\
	{(V^{\langle C\rangle}\times (W^*)^{C>0,F})/(T\times T')} & {V^{C\geq 0,F}/(T\times T')} & {V^F/(T\times T')}
	\arrow["{\imath_{\langle C\rangle}}"', from=1-1, to=2-1]
	\arrow["{q_{C,F}}"', from=1-2, to=1-1]
	\arrow["{p_{C,F}}", from=1-2, to=1-3]
	\arrow["\lrcorner"{anchor=center, pos=0.125, rotate=-90}, draw=none, from=1-2, to=2-1]
	\arrow["{\imath_{C\geq 0,F}}"', from=1-2, to=2-2]
	\arrow["{\imath_F}", from=1-3, to=2-3]
	\arrow["{q'_{C,F}}"', from=2-2, to=2-1]
	\arrow["{p'_{C,F}}", from=2-2, to=2-3]
\end{tikzcd}
\]
Using the same formalism and procedure of proper pushforward by $p_{C,F}$ and virtual pullback by $q_{C,F}$ as in \S\ref{subsection:parabolicinduction2d}, we obtain an induction morphism
\[
 \overline{\Ind}_{C}^F\colon \CH_{T,V,W,\mu,\langle C\rangle}^{T'}\rightarrow\CH_{T,V,W,\mu,F}^{T'}\,.
\]
We define
\[
 \tilde{\Ind}_{C}^F\coloneqq \frac{1}{\Eu_{\Fg,C,F}^{T'}}\overline{\Ind}_C^F\colon \CH_{T,V,W,\mu,\langle C\rangle}^{T'}\rightarrow\CH_{T,V,W,\mu,F}^{T'}[\Eu_{\Fg,C,F}^{T',-1}]\,.
\]
Last, we define $\Ind'^{F}_{C}\colon \CH_{G,V,W,\mu,\langle C\rangle}^{T'}\rightarrow\CH_{G,V,W,\mu,F}^{T',\loc}$ as the restriction to $\CH_{G,V,W,\mu,\langle C\rangle}^{T'}\cong (\CH_{T,V,W,\mu,\langle C\rangle}^{T'})^{\sfW^{\langle C\rangle}}\subset \CH_{T,V,W,\mu,\langle C\rangle}^{T'}$ of the Weyl sum $\sum_{w\in \sfW^F/\sfW^{\langle C\rangle}}w\cdot \tilde{\Ind}^F_{C}$. Here, the superscript $\loc$ denotes the localization of $\CH_{G,V,W,\mu,F}^{T'}$ at the invariant product $\prod_{w\in\sfW^F}w\cdot\Eu_{\Fg,C,F}^{T'}$.

\subsection{Comparison of induction and torus-equivariant induction}
We can compare the induction and torus induction as follows.
\begin{proposition}
 \label{proposition:comparisontorus2d}
 For any $C\in\FC$ and $F\in\FF$ such that $C\preceq F$, we have a commutative diagram
 \[
  \begin{tikzcd}
   \CH_{G,V,W,\mu,\langle C\rangle}^{T'}&\CH_{G,V,W,\mu,F}^{T'}\\
   &\CH_{G,V,W,\mu,F}^{T',\loc}
   \arrow["\Ind_{C}^F",from=1-1, to = 1-2]
   \arrow["\Ind'^F_{C}"',from=1-1, to = 2-2]
   \arrow[from =1-2, to = 2-2]
  \end{tikzcd}
 \]
where the vertical morphism is the natural localization morphism.
\end{proposition}
\begin{proof}
Apply the dimensional-reduction isomorphisms of Proposition~\ref{proposition:comparison2d3dmultiplications} for $G$ and for $T$. They are equivariant for the Weyl group and linear over the equivariant coefficient rings, so they commute with the indicated localizations and Euler factors. Both the geometric $G$-induction and each translated torus induction acquire the same sign $(-1)^{\dim(W^*)^{C>0,F}}$. The proposed square consequently identifies with the critical torus-comparison square of Proposition~\ref{proposition:toruscomparison3d}, multiplied by this common sign. This proves the assertion in the stated auxiliary localization.
\end{proof}

\subsection{Dimensional reduction and comparison with critical induction systems}
\label{subsection:comparisoninductionDimRed}

\subsubsection{The $2d$-potential}
\label{subsubsection:the2dpotential}
 Let $V,W$ be representations of a reductive group $G$, $T'$ an auxiliary torus acting on $V,W$ such that the actions of $G$ and $T'$ commute, and $\mu\colon V\rightarrow W^*$ a $G\times T'$-equivariant function. Then, the contraction of $\mu$ with $W$, $f\colon V\times W\rightarrow \BoC$, $f(v,w)\coloneqq \mu(v)(w)$, is a $G\times T'$-invariant function on $V\times W$. We may refer to it as the \emph{2d potential associated to $\mu$}.

\subsubsection{Comparison}
\label{subsubsection:comparison}

\begin{proposition}
\label{proposition:comparison2d3dmultiplications}
The dimensional reduction isomorphisms
\[
 \mathsf{dr}_F\colon\CH_{G,V,W,\mu,F}^{T'}\rightarrow\CH_{G,V\times W,f,F}^{T'}
\]
from Corollary~\ref{corollary:cohdimrediso} commute with the induction morphisms $\Ind_{C}^F$ on $\CH_{G,V,W,\mu,F}^{T'}$ and $(-1)^{\dim W^{*,C>0,F}}\Ind_{C}^F$ on $\CH_{G,V\times W,f,F}^{T'}$, that is dimensional reduction is compatible with induction up to an explicit sign.
\end{proposition}
\begin{proof}
Fix $C\preceq F$. We use the notation $L,P,U,V_0,E,Z_0,Z_+$ and $r_{C,F}$ of \S\ref{subsection:parabolicinduction2d}, and write
\[
 W_0=W^{\langle C\rangle},\quad W_+=W^{C>0,F},\quad
 W_-=W^{C<0,F},\quad W_F=W^F.
\]
Thus $E=W_-^*$, and we set $m=\dim E$. All the ambient maps in the first part of the proof are $L\times T'$-equivariant.

Put $B=V_0\times E$. On the bundle $B\times W_F\to B$, consider the section
\[
 s_B(v_0,\ell)=\mu_{\langle C\rangle}(v_0)+\ell\in W_F^*.
\]
Its zero locus is $Z_0\times\{0\}$. The associated function is
\[
 F_B(v_0,\ell,w)=f_{\langle C\rangle}(v_0,w_0)+\ell(w_-).
\]
The pullback of $s_B$ along $h_{C,F}\colon U\to B$ is $\mu_F|_U$. In particular, Proposition~\ref{proposition:dimredpullback}, applied to this map and this bundle, gives
\begin{equation}
\label{equation:comparison-dr-refined}
 \mathsf{dr}_U\,h_{C,F,\mathrm{ref}}^!
 =(h_{C,F}\times\id_{W_F})^*\mathsf{dr}_B,
\end{equation}
where the dimensional-reduction maps use the zero loci $Z_+$ and $Z_0\times\{0\}$, respectively.

Consider the Cartesian square
\begin{equation}
\label{equation:comparisoninductions}
\begin{tikzcd}
 {U\times(W_0\oplus W_+)} & {B\times(W_0\oplus W_+)}\\
 {U\times W_F} & {B\times W_F}
 \arrow["h_+",from=1-1,to=1-2]
 \arrow["j_U"',from=1-1,to=2-1]
 \arrow["j_B",from=1-2,to=2-2]
 \arrow["h_{C,F}\times\id_{W_F}"',from=2-1,to=2-2]
 \arrow["\ulcorner"{anchor=center,pos=0.125,rotate=180},draw=none,from=1-1,to=2-2]
\end{tikzcd}
\end{equation}
The vertical maps impose $w_-=0$. They are transverse to the horizontal maps, since the latter are the identity on $W_F$. Let
\[
 a\colon B\times(W_0\oplus W_+)\longrightarrow V_0\times W_0
\]
be the projection. The composite $a h_+$ is the attracting projection for $V\oplus W$, with Levi equivariance. Critical transverse base change therefore gives
\begin{equation}
\label{equation:comparison-transverse}
 (h_{C,F}\times\id_{W_F})^*(j_B)_*a^*
 =(j_U)_*h_+^*a^*.
\end{equation}
To apply \cite[Corollary~2.15]{davison2017critical} to the possibly nonsmooth map $h_{C,F}$, factor it as its graph embedding followed by the smooth projection $U\times B\to B$. After taking the product with $W_F$, the graph embedding is transverse to $w_-=0$, and the projection is smooth. The transverse-embedding and smooth cases of that corollary, composed on smooth equivariant approximations, give the displayed identity.

We identify the remaining normalization explicitly. The function on $B\times W_F$ is the box sum of $f_{\langle C\rangle}$ and the pairing on $E\times W_-$, with the spectator factor $W_+$. The map $(j_B)_*a^*$ uses the isotropic subbundle $E\times\{0\}$ of $E\times W_-$. Dimensional reduction for $s_B$ uses the opposite polarization, with base $E$ and fiber $W_-$. By Proposition~\ref{proposition:vanishingcyclesquadraticfunction}, the two normalizations differ by $(-1)^m$. Applying the Thom identifications for $W_+$ and the dimensional reduction for $W_0$ gives
\begin{equation}
\label{equation:comparison-quadratic-sign}
 \mathsf{dr}_B=(-1)^m(j_B)_*a^*\mathsf{dr}_{\langle C\rangle}.
\end{equation}
In this formula the source of $\mathsf{dr}_B$ is identified with the homology of $Z_0$; its extra twist $\SL^{\dim E}$ agrees with the codimension of $j_B$. This keeps the equality homogeneous before suppressing twists. The assertion for the pairing is compatible with an arbitrary critical coefficient on $V_0\times W_0$ by Thom--Sebastiani.

Combining \eqref{equation:comparison-dr-refined}--\eqref{equation:comparison-quadratic-sign} yields, with Levi equivariance,
\[
 \mathsf{dr}_U h_{C,F,\mathrm{ref}}^!
 =(-1)^m(j_U)_*h_+^*a^*\mathsf{dr}_{\langle C\rangle}.
\]
We now pass from $L$ to $P$. On $Z_+$, on $U\times W_F$, and on $U\times(W_0\oplus W_+)$, the change-of-group map from the $L\times T'$ quotient to the $P\times T'$ quotient is smooth with fiber $P/L$. Its pullback is an isomorphism. Smooth naturality of dimensional reduction identifies the Borel--Moore pullback $r_{C,F}^!$ with the critical pullback, with the dimension twists of \eqref{equation:pullback2d}. Proper base change for $j_U$ intertwines these comparison isomorphisms. Thus their inverses turn the last displayed formula into the corresponding identity with parabolic equivariance.

Finally, the map
\[
 (U\times W_F)/(P\times T')\longrightarrow
 (V^F\times W_F)/(G^F\times T')
\]
is proper, and is obtained by pulling the bundle $W_F$ back along $U/P\to V^F/G^F$. Proposition~\ref{proposition:dimredpropermorphism} identifies its critical pushforward with the proper pushforward $p_{C,F}$ on zero loci. Formula \eqref{equation:2d-levi-induction} now gives
\[
 \mathsf{dr}_F\Ind_C^{F,2d}
 =(-1)^{\dim(W^*)^{C>0,F}}\Ind_C^{F,\mathrm{crit}}
 \mathsf{dr}_{\langle C\rangle},
\]
as asserted. For completeness, the unshifted dimensional-reduction map at a flat $F$ has source twist $\SL^{\dim V^F-\dim G^F-\dim T'}$. If $d^{\mathrm{crit}}_{C,F}$ denotes the codimension of critical induction for $V\oplus W$, then
\[
\begin{aligned}
 d'_{C,F}+d^{\mathrm{crit}}_{C,F}
 &=\dim V^{C>0,F}+\dim V^{C<0,F}
   -\dim\Fg^{C>0,F}-\dim\Fg^{C<0,F}\\
 &=(\dim V^F-\dim G^F)-(\dim V_0-\dim L).
\end{aligned}
\]
This is precisely the equality of twists required in the comparison square. The proof uses only the individual induction maps and the functorialities of Section~\ref{section:equcoh-mhm}; it does not use associativity of $2d$ induction.
\end{proof}

\paragraph{Virtual normalizations and the auxiliary torus.}
The virtual dimensions used in \S\S\ref{subsection:parabolic_induction}, \ref{subsection:parabolicinduction2d} are relative to $BT'$, whereas $\BD$ denotes absolute Verdier duality. We make the resulting common twist in dimensional reduction explicit. Put
\[
 \mathcal Z_F\coloneqq\mu_F^{-1}(0)/(G^F\times T'),\qquad
 \mathcal X_F\coloneqq(V^F\times W^F)/(G^F\times T'),
\]
and write $x_F=\dim V^F-\dim G^F$, $r_F=\dim W^F$ and $t=\dim T'$. Thus
\[
 \BoQ_{\mathcal Z_F}^{\vir}=\BoQ_{\mathcal Z_F}\otimes\SL^{-(x_F-r_F)/2},
 \qquad
 \BoQ_{\mathcal X_F}^{\vir}=\BoQ_{\mathcal X_F}\otimes\SL^{-(x_F+r_F)/2}.
\]
For these normalizations, the degree-preserving dimensional-reduction isomorphism is
\begin{equation}
\label{equation:dimred-relative-virtual-twist}
 \HO^*(\mathcal Z_F,\BD\BoQ_{\mathcal Z_F}^{\vir})\otimes\SL^{-t}
 \xrightarrow{\mathsf{dr}_F}
 \HO^*(\mathcal X_F,\varphi_{f_F}\BoQ_{\mathcal X_F}^{\vir}).
\end{equation}
Indeed, the unshifted comparison has source
$\HO^*(\mathcal Z_F,\BD\BoQ_{\mathcal Z_F})\otimes\SL^{x_F-t}$.
Tensoring it with $\SL^{-(x_F+r_F)/2}$ gives the source exponent
$(x_F-r_F)/2-t$, which is exactly that in \eqref{equation:dimred-relative-virtual-twist}. The extra factor $\SL^{-\dim T'}$ is independent of $F$ and $C$; it does not change the induction compatibility or its sign. As elsewhere, this factor is suppressed in displays where Tate twists are omitted.

\section{The restriction morphisms}
\label{section:restrictionmorphism}

\subsection{Localized cohomology}
\label{subsection:localizedcohomology}
For reasons that will become clear later, we need to localize the $\HO^*_{G^F\times T'}$-modules $\CH_{G,V,f,F}^{T'}$ in order to define the restriction and braiding morphisms.

For $F'\preceq F$, we let
\[
 \Eu_{V,F',F}^{T'}\coloneqq
 \prod_{\substack{\alpha\in\rmX^*(T\times T')\\
 \alpha|_F=0,\ \alpha|_{F'}\neq0}}
 \alpha^{\dim V_\alpha},
 \qquad
 \tilde{\CH}_{G,V,f,F',F}^{T'}
 \coloneqq\CH_{G,V,f,F'}^{T'}[(\Eu_{V,F',F}^{T'})^{-1}].
\]
Here, restrictions to flats concern the $T$-component of the character. The Euler class is $\sfW^{F'}$-invariant. For a cell $C$ with $\langle C\rangle=F'$, it is the product
\[
 \Eu_{V,F',F}^{T'}=\Eu_{V,C,F}^{T'}\Eu_{V,-C,F}^{T'}.
\]
We also write $\tilde{\CH}_{G,V,f,C,F}^{T'}$ for this same localization. If the last flat is omitted, it is understood to be the intersection of all hyperplanes. This convention makes the localization independent of the choice of a cell spanning $F'$. The restriction to $C$ itself only requires inversion of $\Eu_{V,C,F}^{T'}$, with the negative-weight convention of \S\ref{subsection:multEulerclasses}.

\subsection{The restriction morphism}
\label{subsection:restrictionmorphim}

\subsubsection{The restriction morphism via pullback-pushforward}
\label{subsubsection:restviapullbackpushforward}
We let $C\in\FC$ and $F\in\FF$ such that $C\preceq F$. We consider the induction diagram \eqref{equation:inductiondiagramcritical} associated to the opposite cell $-C\in\FC$ and $F\in\FF$. Pullback by the morphism
$p_{-C,F}$ gives a morphism
\begin{equation}
 \label{equation:pbp}
 \CH_{G,V,f,F}^{T'}\rightarrow \HO^*(V^{C\leq 0,F}/(G^{C\leq 0,F}\times T'),\varphi_f\BoQ)\,.
\end{equation}

Since $q_{-C,F}$ has contractible fibers and virtual relative Euler class $\frac{\Eu_{V,C,F}^{T'}}{\Eu_{\Fg,C,F}^{T'}}$, we have the pushforward morphism
\begin{equation}
\label{equation:pfq}
 \HO^*(V^{C\leq 0,F}/(G^{C\leq 0,F}\times T'),\varphi_f\BoQ)\rightarrow\CH_{G,V,f,\langle C\rangle}^{T'}[\Eu_{V,C,F}^{T',-1}]\,,
\end{equation}
which is given by the multiplication by $\frac{\Eu_{\Fg,C,F}^{T'}}{\Eu_{V,C,F}^{T'}}$ of the inverse pullback morphism.

The restriction morphism is given by the composition of \eqref{equation:pfq} and \eqref{equation:pbp}, and is denoted by
\[
\Res_F^C\colon \CH_{G,V,f,F}^{T'}\rightarrow\CH_{G,V,f,\langle C\rangle}^{T'}[\Eu_{V,C,F}^{T',-1}]\,.
\]

\subsubsection{The restriction morphism via pullback}
\label{subsubsection:restviapullback}
We consider the morphism of stacks
\[
 i_{\langle C\rangle,F}\colon V^{\langle C\rangle}/(G^{\langle C\rangle}\times T')\rightarrow V^F/(G^F\times T')\,.
\]
induced by the $(G^{\langle C\rangle},G^F)$-equivariant closed immersion $V^{\langle C\rangle}\rightarrow V^F$. It factors as $p_{-C,F}\circ i'_{-C,F}$ where
\[
 i'_{-C,F}\colon V^{\langle C\rangle}/(G^{\langle C\rangle}\times T')\rightarrow V^{C\leq 0, F}/(G^{C \leq 0, F}\times T')
\]
is a section of $q_{-C,F}$. In particular, the inverse of $q_{-C,F}^*$ is $i'^*_{-C,F}$. This means that
\[
\begin{aligned}
 \Res_{F}^C&=\frac{\Eu_{\Fg,C,F}^{T'}}{\Eu_{V,C,F}^{T'}}i'^*_{-C,F}\circ p_{-C,F}^*\\
 &=\frac{\Eu_{\Fg,C,F}^{T'}}{\Eu_{V,C,F}^{T'}}i^*_{\langle C\rangle ,F}\,.
 \end{aligned}
\]

\subsection{The restriction morphisms via torus-equivariant cohomology}
In the same way that the induction morphisms can be defined using torus-equivariant cohomology and summing over the Weyl group, the restriction morphisms also admit a description via torus-equivariant cohomology.
\subsubsection{The restriction morphism via pullback-pushforward}
\label{subsubsection:restrictionpbpf}

We consider the torus-equivariant diagram \eqref{equation:torusdiagram} for the opposite cell $-C$ and $F$. The morphism $q_{-C,F}$ is an affine fibration with Euler class $\Eu_{V,C,F}^{T'}$. The restriction morphism
\[
 \overline{\Res}'^C_F\colon \CH_{T,V,f,F}^{T'}\rightarrow\CH_{T,V,f,\langle C\rangle}^{T'}[\Eu_{V,C,F}^{T',-1}]
\]
is defined as $(q_{-C,F})_*(p_{-C,F})^*$. Here, $(q_{-C,F})_*$ is the pushforward with respect to a vector bundle with Euler class $\Eu_{V,C,F}^{T'}$. We let
\[
 \tilde{\Res}'^C_F\coloneqq \Eu_{\Fg,C,F}^{T'}\overline{\Res}'^C_F\colon \CH_{T,V,f,F}^{T'}\rightarrow\CH_{T,V,f,\langle C\rangle}^{T'}[\Eu_{V,C,F}^{T',-1}]\,.
\]
Last, $\Res'^C_F$ denotes the restriction of $\tilde{\Res}'^C_F$ to the space of $\sfW^F$-invariants. Its image is contained in $(\CH_{T,V,f,\langle C\rangle}^{T'})^{\sfW^{C}}[\Eu_{V,C,F}^{T',-1}]$.

\subsubsection{The restriction morphism via pullback}
We let $i'_{-C,F}\colon V^{\langle C\rangle}/(T\times T')\rightarrow V^{C\leq 0,F}/(T\times T')$ and $\tilde{i}_{\langle C\rangle,F}\colon V^{\langle C\rangle}/(T\times T')\rightarrow V^F/(T\times T')$ be the closed immersions. The map $i'_{-C,F}$ is a section of $q_{-C,F}$.

\begin{proposition}
 We have $\overline{\Res}'^C_F=\frac{1}{\Eu_{V,C,F}^{T'}}\tilde{i}_{\langle C\rangle,F}^*$ and $\tilde{\Res}'{}^C_F=\frac{\Eu_{\Fg,C,F}^{T'}}{\Eu_{V,C,F}^{T'}}\tilde{i}_{\langle C\rangle,F}^*$.
\end{proposition}
\begin{proof}
 The second equality follows immediately from the first and the first comes from the fact that $i'^*_{-C,F}$ is the inverse of $q_{-C,F}^*$.
\end{proof}

\subsection{Comparison of the restriction morphisms}

\begin{proposition}
 The restriction morphisms
 \[
  \Res_{F}^C,\Res'^C_{F}\colon\CH_{G,V,f,F}^{T'}\rightarrow\tilde{\CH}_{G,V,f,C,F}^{T'}
 \]
coincide, via the identification of Lemma~\ref{lemma:Winvariantpart}.
\end{proposition}
\begin{proof}
This is almost straightforward. We may just compare the pullbacks $i_{\langle C\rangle,F}^*$ and $\tilde{i}_{\langle C\rangle,F}^*$, since the relevant restrictions are obtained by multiplying the respective pullback by $\frac{\Eu_{\Fg,C,F}^{T'}}{\Eu_{V,C,F}^{T'}}$. The fact that these pullbacks coincide via the identifications given by Lemma~\ref{lemma:Winvariantpart} is a consequence of the following commutative diagram
 % https://q.uiver.app/#q=WzAsNCxbMCwxLCJWXntcXGxhbmdsZSBDXFxyYW5nbGV9LyhHXntcXGxhbmdsZSBDXFxyYW5nbGV9XFx0aW1lcyBUJykiXSxbMSwxLCJWXkYvKEdeRlxcdGltZXMgVCcpIl0sWzAsMCwiVl57XFxsYW5nbGUgQ1xccmFuZ2xlfS8oVFxcdGltZXMgVCcpIl0sWzEsMCwiVl5GLyhUXFx0aW1lcyBUJykiXSxbMiwzLCJcXHRpbGRle2l9X3tcXGxhbmdsZSBDXFxyYW5nbGUsRn0iXSxbMiwwXSxbMywxXSxbMCwxLCJpX3tcXGxhbmdsZSBDXFxyYW5nbGUsRn0iLDJdXQ==
\[\begin{tikzcd}
	{V^{\langle C\rangle}/(T\times T')} & {V^F/(T\times T')} \\
	{V^{\langle C\rangle}/(G^{\langle C\rangle}\times T')} & {V^F/(G^F\times T')}
	\arrow["{\tilde{i}_{\langle C\rangle,F}}", from=1-1, to=1-2]
	\arrow[from=1-1, to=2-1]
	\arrow[from=1-2, to=2-2]
	\arrow["{i_{\langle C\rangle,F}}"', from=2-1, to=2-2]
\end{tikzcd}\]
in which the vertical arrows are used in the identification of Lemma~\ref{lemma:Winvariantpart}.

\end{proof}

\subsection{Explicit formula for the restriction for vanishing potential}
When the potential vanishes, the cohomology spaces under consideration are polynomial rings and it is possible, and rather easy, to give a closed formula for the restriction morphisms.

\begin{proposition}
 Let $C\in\FC$ and $F\in\FF$ be such that $C\preceq F$. Then, $\Res_{F}^C$ is given by $\frac{\Eu_{\Fg,C,F}^{T'}}{\Eu_{V,C,F}^{T'}}$ times the inclusion
 \[
  (\HO^*_{T\times T'})^{\sfW^{F}}\cong\CH_{G,V,F}^{T'}\rightarrow\CH_{G,V,\langle C\rangle}^{T'}\cong (\HO^*_{T\times T'})^{\sfW^{\langle C\rangle}}\,.
 \]
\end{proposition}
\begin{proof}
 This is immediate from the definitions.
\end{proof}

\subsection{Coassociativity}
The following proposition is the coassociativity constraint on the restriction morphisms, or the transitivity of the restriction.
\begin{proposition}
\label{proposition:coassociativityrestriction}
 For any $C,C'\in\FC$ and $F\in\FF$ such that $C\preceq C'\preceq F$, the restriction morphism $\Res_{\langle C'\rangle}^C$ extends canonically to $\CH_{G,V,f,\langle C'\rangle}^{T'}[\Eu_{V,C',F}^{T',-1}]$ and we have
 \[
  \Res_{F}^C=\Res_{\langle C'\rangle}^C\circ\Res_F^{C'}\,.
 \]
\end{proposition}
\begin{proof}
The restriction $\Res_{\langle C'\rangle}^C$ is $\HO^*_{G^{\langle C'\rangle}\times T'}$-linear and $\Eu_{V,C',F}^{T'}$ is $\sfW^{\langle C'\rangle}$-invariant. Therefore, $\Res_{\langle C'\rangle}^C$ extends to $\CH_{G,V,f,\langle C'\rangle}^{T'}[\Eu_{V,C',F}^{T',-1}]$. Moreover,
\[
 \Res_{F}^C=\frac{\Eu_{\Fg,C,F}^{T'}}{\Eu_{V,C,F}^{T'}}i_{\langle C\rangle,F}^*
\]
and
\[
 \Res_{\langle C'\rangle}^C\circ\Res_{F}^{C'}=\frac{\Eu_{\Fg,C,\langle C'\rangle}^{T'}}{\Eu_{V,C,\langle C'\rangle}^{T'}}i^*_{\langle C\rangle,\langle C'\rangle}\frac{\Eu_{\Fg,C',F}^{T'}}{\Eu_{V,C',F}^{T'}}i^*_{\langle C'\rangle,F}\,.
\]
The identification of $\Res_{F}^C$ with $\Res_{\langle C'\rangle}^C\circ\Res_{F}^{C'}$ now comes from the fact that $\frac{\Eu_{\Fg,C,\langle C'\rangle}^{T'}}{\Eu_{V,C,\langle C'\rangle}^{T'}}\frac{\Eu_{\Fg,C',F}^{T'}}{\Eu_{V,C',F}^{T'}}=\frac{\Eu_{\Fg,C,F}^{T'}}{\Eu_{V,C,F}^{T'}}$ and $i_{\langle C'\rangle,F}\circ i_{\langle C\rangle,\langle C'\rangle}=i_{\langle C\rangle,F}$.
\end{proof}

\subsection{Weyl group invariance}
\begin{lemma}
 For any $C \in\FC$ and $F\in\FF$ such that $C\preceq F$ and $w\in \sfW$, $w\Res_{F}^C=\Res_{w\cdot F}^{w\cdot C}w$.
\end{lemma}
\begin{proof}
 We let $\dot{w}\in N_G(T)$ be a lift of $w$. Then, the action of $\dot{w}$ gives a commutative diagram
  % https://q.uiver.app/#q=WzAsNCxbMCwwLCJWXntcXGxhbmdsZSBDXFxyYW5nbGV9LyhHXntcXGxhbmdsZSBDXFxyYW5nbGV9XFx0aW1lcyBUJykiXSxbMSwwLCJWXkYvKEdee0Z9XFx0aW1lcyBUJykiXSxbMCwxLCJWXntcXGxhbmdsZSB3XFxjZG90IENcXHJhbmdsZX0vKEdee1xcbGFuZ2xlIHdcXGNkb3QgQ1xccmFuZ2xlfVxcdGltZXMgVCcpIl0sWzEsMSwiVl57d1xcY2RvdCBGfS8oR157d1xcY2RvdCBGfVxcdGltZXMgVCcpIl0sWzAsMiwiXFxkb3R7d30iLDJdLFsxLDMsIlxcZG90e3d9Il0sWzAsMSwiaV97XFxsYW5nbGUgQ1xccmFuZ2xlLEZ9Il0sWzIsMywiaV97XFxsYW5nbGUgd1xcY2RvdCBDXFxyYW5nbGUsd1xcY2RvdCBGfSIsMl1d
\[\begin{tikzcd}
	{V^{\langle C\rangle}/(G^{\langle C\rangle}\times T')} & {V^F/(G^{F}\times T')} \\
	{V^{\langle w\cdot C\rangle}/(G^{\langle w\cdot C\rangle}\times T')} & {V^{w\cdot F}/(G^{w\cdot F}\times T')}
	\arrow["{i_{\langle C\rangle,F}}", from=1-1, to=1-2]
	\arrow["{\dot{w}}"', from=1-1, to=2-1]
	\arrow["{\dot{w}}", from=1-2, to=2-2]
	\arrow["{i_{\langle w\cdot C\rangle,w\cdot F}}"', from=2-1, to=2-2]
\end{tikzcd}\]
from which the equality $w\Res_{F}^C=\Res_{w\cdot F}^{w\cdot C}w$ follows, using the fact that $w\cdot\Eu_{W,C,F}^{T'}=\Eu_{W,w\cdot C,w\cdot F}^{T'}$ for $W=V,\Fg$.
\end{proof}

\section{Cohomological Mackey formula}
\label{section:cohomologicalMackeyformula}

\subsection{The localized system}
We let $G$ be a reductive group, $V$ a representation of $G$, $T'$ an auxiliary torus acting on $V$ such that the actions of $G$ and $T'$ commute, and $f\colon V\rightarrow\BoC$ a $G\times T'$-invariant regular function.

\begin{definition}
\label{definition:localized-induction-restriction-system}
For the Coxeter complex of $(G,V)$, a \emph{localized induction-restriction system} consists of graded modules $M_F$ over $R_F=\HO^*_{G^F\times T'}(\pt)$, their localizations
\[
 \widetilde M_{F',F}=M_{F'}[(\Eu_{V,F',F}^{T'})^{-1}]
 \qquad(F'\preceq F),
\]
Weyl-equivariant identifications of these modules, homogeneous induction and restriction maps
\[
 I_C^F\colon M_{\langle C\rangle}\to M_F,
 \qquad R_F^C\colon M_F\to\widetilde M_{\langle C\rangle,F},
\]
and braiding automorphisms $\tau_{C,F}^{D,F}$ of $\widetilde M_{\langle C\rangle,F}$ for $\langle C\rangle=\langle D\rangle$. The maps are linear over the coefficient rings of their targets for induction, and of their sources for restriction, using restriction of equivariance on coefficient rings. They satisfy the following identities, with all compositions extended to the indicated common localizations:
\begin{enumerate}
 \item The identities $I_C^{\langle C\rangle}=R_{\langle C\rangle}^C=\id$, and transitivity $I_C^F=I_{C'}^F I_C^{\langle C'\rangle}$ and $R_F^C=R_{\langle C'\rangle}^C R_F^{C'}$ when $C\preceq C'\preceq F$.
 \item Weyl equivariance of induction, restriction and braiding.
 \item The identities $\tau_{C,F}^{C,F}=\id$ and $\tau_{D,F}^{E,F}\tau_{C,F}^{D,F}=\tau_{C,F}^{E,F}$ for cells with the same span.
 \item The Mackey identity of Theorem~\ref{theorem:mackeyformula}, with $I,R$ in place of $\Ind,\Res$.
\end{enumerate}
A morphism of such systems is a family of coefficient-linear maps commuting with the Weyl actions, induction, restriction and braiding. All identities are understood with the grading shifts, or equivalently the Tate twists, of the maps involved.
\end{definition}
The preceding sections establish transitivity and equivariance for the critical cohomological system. The braiding operators below satisfy the third condition directly, and Theorem~\ref{theorem:mackeyformula} establishes the remaining condition.

\subsection{The braiding operators}
We occasionally omit the auxiliary torus $T'$ from the notation in this section.

\begin{definition}
\label{definition:braidingoperators}
For any $C,C'\in\FC$ such that $\langle C\rangle=\langle C'\rangle$ and $F\in \FF$ such that $C,C'\preceq F$, we define
\[
 \begin{matrix}
  \tau_{C,F}^{C',F}&\colon&\tilde{\CH}_{G,V,f,C,F}&\rightarrow&\tilde{\CH}_{G,V,f,C',F}\\
  &&a&\mapsto&\frac{k_{C,F}}{k_{C',F}}a\,.
 \end{matrix}
\]
It is an isomorphism. Indeed, the two cells have the same vanishing roots. The products of their negative roots therefore differ by a sign, while all moving weights of $V$ are invertible in the common localization defined in \S\ref{subsection:localizedcohomology}. When $F$ is the maximal flat (i.e. the intersection of all hyperplanes of the hyperplane arrangement $\bfH$), then we simply write $\tau_{C}^{C'}$.
\end{definition}

\begin{lemma}
 For any $C,C'\in\FC$ such that $\langle C\rangle=\langle C'\rangle$, $F\in\FF$ such that $C,C'\preceq F$, and $w\in \sfW$, one has $w\tau_{C,F}^{C',F}=\tau_{w\cdot C,w\cdot F}^{w\cdot C',w\cdot F}w$.
\end{lemma}
\begin{proof}
This comes directly from the fact that for any $w\in \sfW$, $w\cdot k_{C,F}=k_{w\cdot C,w\cdot F}$.
\end{proof}

\subsection{A lemma on double cosets}

\begin{lemma}
\label{lemma:doublequotientdescription}
 Let $C,C'\in\FC$ and $F\in\FF$ be such that $C,C'\preceq F$. For any $w\in \sfW^{\langle C'\rangle}\backslash \sfW^{F}/\sfW^{\langle C\rangle}$, we fix $\dot{w}\in \sfW^{F}$ a lift of $w$. For any $w'\in \sfW^{\langle C'\rangle}/\sfW^{\langle C'\circ \dot{w}\cdot C\rangle}$, we let $\dot{w}'\in \sfW^{\langle C'\rangle}$ be a lift. Then, the map
 \[
  \begin{matrix}
   \bigsqcup_{w\in \sfW^{\langle C'\rangle}\backslash \sfW^{F}/\sfW^{\langle C\rangle}}\sfW^{\langle C'\rangle}/\sfW^{\langle C'\circ \dot{w}\cdot C\rangle}&\rightarrow& \sfW^{F}/\sfW^{\langle C\rangle}\\
   (w,w')&\mapsto&\dot{w'}\dot{w}\sfW^{\langle C\rangle}
  \end{matrix}
 \]
 is a bijection. It does not depend on the chosen lift $\dot{w}'$.
\end{lemma}
\begin{proof}
First, $\sfW^{\langle C'\circ\dot w\cdot C\rangle}=\sfW^{\langle C'\rangle}\cap\sfW^{\langle\dot w\cdot C\rangle}$. Indeed, a root vanishes on the Tits product if and only if it vanishes on both cells. Equivalently, the stabilizer of a sufficiently small generic perturbation of a point of $C'$ towards $\dot w\cdot C$ is the intersection of their pointwise stabilizers.
 Let $w\in \sfW^{\langle C'\rangle}\backslash \sfW^{F}/\sfW^{\langle C\rangle}$ and $w'\in \sfW^{\langle C'\rangle}/\sfW^{\langle C'\circ \dot{w}\cdot C\rangle}$. We let $\dot{w}\in \sfW^{F}$ and $\dot{w}'\in \sfW^{\langle C'\rangle}$ be respective lifts. Let $\ddot{w}'\in \sfW^{\langle C'\rangle}$ be another lift of $w'$. Then, $\ddot{w}'^{-1}\dot{w}'\in \sfW^{\langle \dot{w}\cdot C\rangle}$. Therefore, $\dot{w}^{-1}\ddot{w}'^{-1}\dot{w}'\dot{w}\in \sfW^{\langle C\rangle}$, which precisely means that $\dot{w}'\dot{w}\sfW^{\langle C\rangle}=\ddot{w}'\dot{w}\sfW^{\langle C\rangle}$. Therefore, the map of the lemma does not depend on the chosen lift $\dot{w}'\in \sfW^{\langle C'\rangle}$ and so is well-defined.
 
 The map is surjective. Indeed, if $\tilde{w}\in \sfW^{F}/\sfW^{\langle C\rangle}$, there exists $\alpha\in \sfW^{\langle C'\rangle}$ such that $\alpha\tilde{w}=\dot{w}$ for some fixed representative $\dot{w}$ of some $w\in \sfW^{\langle C'\rangle}\backslash \sfW^{F}/\sfW^{\langle C\rangle}$. Then, $\tilde{w}$ is the image of $(w,\alpha^{-1}\sfW^{\langle C'\circ \dot{w}\cdot C\rangle})$.
 
 The map is injective. Indeed, if for some $w\in \sfW^{\langle C'\rangle}\backslash \sfW^{F}/\sfW^{\langle C\rangle}$ and $w'\in \sfW^{\langle C'\rangle}/\sfW^{\langle C'\circ \dot{w}\cdot C\rangle}$, $\tilde{w}\in \sfW^{\langle C'\rangle}\backslash \sfW^{F}/\sfW^{\langle C\rangle}$ and $\tilde{w}'\in \sfW^{\langle C'\rangle}/\sfW^{\langle C'\circ \dot{\tilde{w}}\cdot C\rangle}$, $\dot{w}'\dot{w}\sfW^{\langle C\rangle}=\dot{\tilde{w}}'\dot{\tilde{w}}\sfW^{\langle C\rangle}$, then by taking the quotient by $\sfW^{\langle C'\rangle}$ on the left, $\sfW^{\langle C'\rangle}\dot{w}\sfW^{\langle C \rangle}=\sfW^{\langle C'\rangle}\dot{\tilde{w}}\sfW^{\langle C\rangle}$ and so, $\dot{w}=\dot{\tilde{w}}$ since we have fixed representatives of the double quotient $\sfW^{\langle C'\rangle}\backslash \sfW^{F}/\sfW^{\langle C\rangle}$. Then, $\dot{w}^{-1}\dot{\tilde{w}}'^{-1}\dot{w}'\dot{w}\in \sfW^{\langle C\rangle}$. Therefore, $\dot{\tilde{w}}'^{-1}\dot{w}'\in \sfW^{\langle \dot{w}\cdot C\rangle}$. Since $\dot{\tilde{w}}'$ and $\dot{w}'$ belong to $\sfW^{\langle C'\rangle}$ by definition, $\dot{\tilde{w}}'^{-1}\dot{w}'\in \sfW^{\langle C'\rangle}$. Therefore, $\dot{\tilde{w}}'^{-1}\dot{w}'\in \sfW^{\langle C'\circ\dot{w}\cdot C\rangle}=\sfW^{\langle C'\rangle}\cap \sfW^{\langle \dot{w}\cdot C\rangle}$. This concludes.
\end{proof}

\subsection{Mackey formula}

\begin{lemma}
 For any $C,C'\in \FC$ and $F\in\FF$ such that $C,C'\preceq F$, one has $C\circ C'\preceq F$.
\end{lemma}
\begin{proof}
 By assumption, we have $F\subset\langle C\rangle, \langle C'\rangle$ and so $F\subset \langle C\rangle\subset \langle C\circ C'\rangle$. This concludes.
\end{proof}

\begin{lemma}
\label{lemma:composition_only_depends_on_class}
For any $C,C'\in \FC$ and $F\in\FF$ such that $C,C'\preceq F$ and $\dot{w}\in \sfW^F$, the composition $\Ind_{C'\circ \dot{w}\cdot C}^{C'}\circ\tau_{\dot{w}\cdot C\circ C',F}^{C'\circ \dot{w}\cdot C,F}\circ \Res_{\langle\dot w \cdot C\rangle}^{\dot{w}\cdot C\circ C'}\circ \dot{w}$ only depends on the class of $\dot{w}$ in the double quotient $\sfW^{\langle C'\rangle}\backslash \sfW^{F}/\sfW^{\langle C\rangle}$.
\end{lemma}
\begin{proof}
We first explain the extension of induction implicit in the statement. Set $D=\dot w\cdot C$, $A=D\circ C'$, $B=C'\circ D$, and $K=\langle A\rangle=\langle B\rangle$. We let
\[
 e\coloneqq\Eu_{V,\langle C'\rangle,F}^{T'}
 \in\HO^*_{G^{\langle C'\rangle}\times T'}(\pt).
\]
Every weight inverted by $\Res_{\langle D\rangle}^{A}$ vanishes on $D$ and is nonzero on $C'$, and hence divides $e$. In the quotient $k_{A,F}/k_{B,F}$, the root factors cancel up to a sign. Moreover, the factors corresponding to weights that vanish on $C'$ cancel, since such weights have the same sign on $A$ and $B$. Thus the remaining denominators also divide a power of $e$. It follows that the image to which induction is applied is contained in
\[
 \CH_{G,V,f,K}^{T'}[e^{-1}].
\]
By the projection formula, $\Ind_B^{\langle C'\rangle}$ is linear over $\HO^*_{G^{\langle C'\rangle}\times T'}(\pt)$, so it extends uniquely to this localization, with values in $\tilde{\CH}_{G,V,f,\langle C'\rangle,F}^{T'}$.

We now prove independence of the representative. We let $w'\in \sfW^{\langle C'\rangle}$, $w''\in \sfW^{\langle C\rangle}$ and $f\in \CH_{G,V,f,\langle C\rangle}^{T'}$. Then,
\[
\begin{aligned}
&\Ind_{C'\circ w'\dot{w}w''\cdot C}^{C'}\circ\tau_{w'\dot{w}w''\cdot C\circ C',F}^{C'\circ w'\dot{w}w''\cdot C,F}\circ \Res_{\langle w'\dot w w'' \cdot C\rangle}^{w'\dot{w}w''\cdot C\circ C'}\circ w'\dot{w}w''f\\&=\Ind_{C'\circ w'\dot{w}\cdot C}^{C'}\circ\tau_{w'\dot{w}\cdot C\circ C',F}^{C'\circ w'\dot{w}\cdot C,F}\circ \Res_{\langle w'\dot w \cdot C\rangle}^{w'\dot{w}\cdot C\circ C'}\circ w'\dot{w}f\quad\text{ since $f$ and $C$ are $\sfW^{\langle C\rangle}$-invariant}\\
  &=w'(\Ind_{C'\circ \dot{w}\cdot C}^{C'}\circ\tau_{\dot{w}\cdot C\circ C',F}^{C'\circ \dot{w}\cdot C,F}\circ \Res_{\langle\dot w \cdot C\rangle}^{\dot{w}\cdot C\circ C'}\circ \dot{w}f)\quad\text{ since $C'$ is $\sfW^{\langle C'\rangle}$-invariant}\\
  &=\Ind_{C'\circ \dot{w}\cdot C}^{C'}\circ\tau_{\dot{w}\cdot C\circ C',F}^{C'\circ \dot{w}\cdot C,F}\circ \Res_{\langle\dot w \cdot C\rangle}^{\dot{w}\cdot C\circ C'}\circ \dot{w}f\quad\text{since the image of $\Ind_{C'\circ\dot{w}\cdot C}^{C'}$ is $\sfW^{\langle C'\rangle}$-invariant}\,.
 \end{aligned}
\]
This concludes.
\end{proof}

\begin{theorem}[Mackey formula]
\label{theorem:mackeyformula}
 For any $C,C'\in\FC$ and $F\in\FF$ such that $C,C'\preceq F$, one has
 \[
  \Res_F^{C'}\circ \Ind_C^F=\sum_{w\in \sfW^{\langle C'\rangle}\backslash \sfW^F/\sfW^{\langle C\rangle}} \Ind_{C'\circ \dot{w}\cdot C}^{C'}\circ\tau_{\dot{w}\cdot C\circ C',F}^{C'\circ \dot{w}\cdot C,F}\circ \Res_{\langle\dot w \cdot C\rangle}^{\dot{w}\cdot C\circ C'}\circ \dot{w}\,.
 \]
\end{theorem}
\begin{proof}
We use fixed-component localization on a flag variety, retaining the induction maps in the target Levi as proper pushforwards. This will introduce only root factors moving under the cocharacter of $C'$.

\emph{The auxiliary localization.} Set
\[
 G_0=G^F,\quad X=V^F,\quad H=G^{\langle C'\rangle},\quad
 Y=V^{\langle C'\rangle},\quad
 M=\tilde{\CH}_{G,V,f,\langle C'\rangle,F}^{T'}.
\]
Choose an integral cocharacter $\lambda\in C'$ and put $S=\lambda(\BoC^*)$. Then $H=Z_{G_0}(S)$ and $Y=X^S$. Define
\[
 \Delta=\prod_{\substack{\alpha\in\Phi(G_0,T)\\\langle\alpha,\lambda\rangle\neq0}}\alpha
 =\Eu_H(\Lie(G_0)/\Lie(H))\in\HO^*_{H}(\pt).
\]
We claim that
\begin{equation}
\label{equation:mackey-safe-localization}
 M\longrightarrow M[\Delta^{-1}]
 \quad\text{is injective}.
\end{equation}
If $S$ is trivial, then $\Delta=1$. Otherwise, choose a connected reductive $H_0\subset H$ such that $S\times H_0\to H$ is an isogeny. Since $S$ acts trivially on $Y$, rational equivariant cohomology gives
\[
 \HO^*_{H\times T'}(Y,\varphi_{f_Y}\BoQ)
 \cong\HO^*_{H_0\times T'}(Y,\varphi_{f_Y}\BoQ)\otimes\BoQ[u].
\]
Every $S$-weight of $\Lie(G_0)/\Lie(H)$ is nonzero. Thus $\Delta$, as a polynomial in $u$, has a nonzero scalar leading coefficient and acts injectively. Localization at $\Eu_{V,\langle C'\rangle,F}^{T'}$ preserves this injectivity. This proves \eqref{equation:mackey-safe-localization}; it is the argument of Proposition~\ref{proposition:multEulerinjective} for the moving roots of $C'$.

\emph{The incidence bundle.} Put $P=G^{C\geq0,F}$, $U=V^{C\geq0,F}$ and $\mathcal F=G_0/P$. The space $\mathcal E=G_0\times_P U$ is a vector bundle on $\mathcal F$. We factor the proper induction map as
\[
 \mathcal E\xrightarrow{j}X\times\mathcal F\xrightarrow{\pi}X,
 \qquad j([g,v])=(gv,gP).
\]
Here $j$ is a closed embedding and $\pi$ is smooth and proper. For $a\in\CH_{G,V,f,\langle C\rangle}^{T'}$, pull $a$ back along the attracting projection and then restrict equivariance from $G_0$ to $H$; denote the resulting class on $\mathcal E$ by $\tilde a$. Let $i\colon Y\hookrightarrow X$, and put
\[
 \beta=(i\times\id_{\mathcal F})^*j_*\tilde a.
\]
Smooth proper base change gives
\begin{equation}
\label{equation:mackey-incidence-restriction}
 i_{\langle C'\rangle,F}^*\Ind_C^F(a)=(\pi_Y)_*\beta,
 \qquad\pi_Y\colon Y\times\mathcal F\to Y.
\end{equation}
The coefficient complex on $Y\times\mathcal F$ is
\[
 \varphi_{f_Y\circ\pi_Y}\BoQ_{Y\times\mathcal F}
 \cong\varphi_{f_Y}\BoQ_Y\boxtimes\BoQ_{\mathcal F}.
\]
In particular, restriction in the flag coordinate is ordinary restriction of this coefficient complex.

\emph{Fixed flags and their normal bundles.} We have
\[
 \mathcal F^S=\bigsqcup_{w\in\sfW^{\langle C'\rangle}\backslash\sfW^F/\sfW^{\langle C\rangle}}\mathcal F_w.
\]
For a chosen representative $\dot w$, set
\[
 D=\dot w\cdot C,\quad A=D\circ C',\quad B=C'\circ D,
 \quad K=\langle A\rangle=\langle B\rangle.
\]
Then $Q_w=H\cap G^{D\geq0,F}$ is a parabolic of $H$ with Levi $G^K$, and $\mathcal F_w\cong H/Q_w$. To see the decomposition, a parabolic fixed by $S$ contains $S$. A maximal torus in it containing $S$ is conjugate to $T$ by $Z_{G_0}(S)=H$. Hence every fixed flag is $H$-conjugate to a $T$-fixed flag, and the stated double cosets classify the resulting orbits.

The normal bundle of $\mathcal F_w$ in $\mathcal F$ is associated to the $Q_w$-representation
$\Lie(G_0)/(\Lie(G^{D\geq0,F})+\Lie(H))$. Its Euler class, viewed in $\HO^*_{Q_w}(\pt)\cong\HO^*_{G^K}(\pt)$, is
\begin{equation}
\label{equation:mackey-normal-root-factor}
 n_w=\prod_{\substack{\alpha\in\Phi(G_0,T)\\\alpha(D)<0,\ \alpha(C')\neq0}}\alpha.
\end{equation}
The polynomial $n_w$ divides the restriction of $\Delta$, so it is invertible after inverting $\Delta$.

We check that this particular localization suffices for the localization theorem. If $z\in\mathcal F\setminus\mathcal F^S$ and $R=(T_z)^\circ$, there is a root $\alpha$ with $\alpha|_R=0$ and $\alpha|_S\neq0$. Otherwise $Z_{G_0}(R)\subset H$, and the description of $\mathcal F^R$ as centralizer orbits of $T$-fixed flags would imply $z\in\mathcal F^S$. Thus $\Delta$ vanishes on every stabilizer Lie algebra in this complement. The same holds on $Y\times(\mathcal F\setminus\mathcal F^S)$; the auxiliary torus acts trivially on $\mathcal F$. The localization theorem with constructible coefficients, \cite[Theorem~6.2]{goresky1998equivariant}, therefore applies after inverting $\Delta$. It can first be applied for $T\times T'$ and then passed to $\sfW^{\langle C'\rangle}$-invariants. If $\ell_w\colon Y\times\mathcal F_w\hookrightarrow Y\times\mathcal F$ is the inclusion, its inverse is given by self-intersection:
\begin{equation}
\label{equation:mackey-flag-localization}
 \beta=\sum_w(\ell_w)_*\left(n_w^{-1}\ell_w^*\beta\right).
\end{equation}
This identity imposes no torsion-freeness assumption for the roots of $H$.

\emph{The excess factor.} Restricting the incidence bundle to $\mathcal F_w$ gives
\[
 H\times_{Q_w}V^{D\geq0,F}\subset X\times\mathcal F_w.
\]
Its intersection with $Y\times\mathcal F_w$ is
\[
 \mathcal E_w=H\times_{Q_w}V^{D\geq0,\langle C'\rangle}.
\]
The sum and intersection of these subbundles have constant rank. The excess quotient has fiber
\[
 X/(V^{D\geq0,F}+Y).
\]
Its Euler class is
\[
 e_w=\prod_{\substack{\chi\in\rmX^*(T\times T')\\
 \chi(D)<0,\ \chi(C')\neq0}}
 \chi^{\dim (V^F)_\chi}.
\]
The restriction to $X\times\mathcal F_w$ is transverse to the incidence bundle, since the latter projects smoothly to $\mathcal F$. Lemma~\ref{lemma:critical-subbundle-excess}, applied afterwards in $X\times\mathcal F_w$, therefore computes $\ell_w^*\beta$ by the excess class $e_w$.

More explicitly, write
\[
 i_w\colon V^K/(G^K\times T')\longrightarrow
 V^{\langle D\rangle}/(G^{\langle D\rangle}\times T').
\]
The quotient of $\mathcal E_w$ by $H\times T'$ is the attracting correspondence for $B$ inside $\langle C'\rangle$. Its projection to the source of $\tilde a$ factors as its attracting projection $q_B$ followed by $i_w$, and conjugation by $\dot w$ sends $a$ to $\dot w a$. Thus, if $j_w\colon\mathcal E_w\hookrightarrow Y\times\mathcal F_w$ is the inclusion,
\[
 \ell_w^*\beta=(j_w)_*\bigl(e_w q_B^*i_w^*(\dot w a)\bigr).
\]
Combining this with \eqref{equation:mackey-incidence-restriction} and \eqref{equation:mackey-flag-localization}, and using the projection formula, gives
\begin{equation}
\label{equation:mackeyreduction}
 i_{\langle C'\rangle,F}^*\Ind_C^F(a)
 =\sum_w\Ind_B^{\langle C'\rangle}
 \left(\frac{e_w}{n_w}i_w^*(\dot w a)\right)
\end{equation}
after inverting $\Delta$. Induction in $H$ has not been expanded into a Weyl sum.

\emph{Cancellation and descent.} The restriction map is $k_{C',F}^{-1}i_{\langle C'\rangle,F}^*$. Moving its coefficient inside each induction by the projection formula gives the scalar
\[
 r_w=\frac{e_w}{k_{C',F}n_w}.
\]
Pairing opposite roots gives the following expression with no root denominators:
\begin{equation}
\label{equation:mackey-cancelled-coefficient}
 r_w=(-1)^{m_w}\Eu\bigl(\Fg^{D=0,C'<0,F}\bigr)
 \frac{\Eu\bigl(V^{D<0,C'>0,F}\bigr)}
 {\Eu\bigl(V^{D\geq0,C'<0,F}\bigr)},
 \qquad m_w=\dim\Fg^{D<0,C'>0,F}.
\end{equation}
The notation in this formula means the indicated simultaneous weight conditions, and the Euler classes retain the $T'$-characters. Every denominator divides the restriction of $\Eu_{V,\langle C'\rangle,F}^{T'}$. Thus each resulting induction term is defined with values in $M$, by the projection formula as in Lemma~\ref{lemma:composition_only_depends_on_class}.

The Tits-product identities
\[
 k_{A,F}=k_{D,F}k_{A,\langle D\rangle},\qquad
 k_{B,F}=k_{C',F}k_{B,\langle C'\rangle}
\]
give the identity of rational coefficients
\[
 r_w=\frac{k_{D,F}}{k_{C',F}k_{B,\langle C'\rangle}}
 =\frac{k_{A,F}}{k_{B,F}}\frac1{k_{A,\langle D\rangle}}.
\]
By \eqref{equation:mackey-cancelled-coefficient}, this is an identity in the allowed representation-weight localization after cancellation; it does not require inverting an internal root of $H$. Since
$\Res_{\langle D\rangle}^{A}=k_{A,\langle D\rangle}^{-1}i_w^*$, equation \eqref{equation:mackeyreduction} gives exactly the asserted Mackey sum in $M[\Delta^{-1}]$. Both sides already take values in $M$. Injectivity of \eqref{equation:mackey-safe-localization} proves the equality in $M$.
\end{proof}

\begin{remark}
\label{remark:internal-root-torsion}
The restriction on the auxiliary root denominators is essential. For example, let $\mathrm{SL}_2$ act on $\operatorname{Mat}_2(\BoC)$ by left multiplication and take $f=(\det-1)^2$. The zero fiber is supported on $\{\det=1\}\cong\mathrm{SL}_2$, and the vanishing cycles are a constant rank-one normal Morse factor on this hypersurface. Thus its $\mathrm{SL}_2$-equivariant critical cohomology is $\BoQ$ up to its shift and monodromic factor, and the class $c_2\in\HO^*(B\mathrm{SL}_2)$ annihilates it. Inverting all roots would annihilate this nonzero module. The proof above only inverts the roots of $\Lie(G_0)/\Lie(H)$, whose Euler class is injective on the stated target by the central-cocharacter argument.
\end{remark}

\begin{corollary}
\label{corollary:Mackey3cells}
For any $C,C'\preceq C''\in\FC$, one has
\[
\Res_{\langle C''\rangle}^{C'}\circ \Ind_C^{\langle C''\rangle}=\sum_{w\in \sfW^{\langle C'\rangle}\backslash \sfW^{\langle C''\rangle}/\sfW^{\langle C\rangle}} \Ind_{C'\circ \dot{w}\cdot C}^{C'}\circ\tau_{\dot{w}\cdot C\circ C'}^{C'\circ \dot{w}\cdot C}\circ \Res_{\langle\dot w \cdot C\rangle}^{\dot{w}\cdot C\circ C'}\circ \dot{w}\,.
\]
\end{corollary}
\begin{proof}
Thanks to Theorem~\ref{theorem:mackeyformula}, the only thing to check is the equality
\[
 \frac{k_{\dot{w}\cdot C\circ C',\langle C''\rangle}}{k_{C'\circ\dot{w}\cdot C,\langle C''\rangle}}=\frac{k_{\dot{w}\cdot C\circ C'}}{k_{C'\circ \dot{w}\cdot C}}
\]
for any $\dot{w}\in \sfW^{\langle C''\rangle}$. It comes from the fact that $C'\circ \dot{w}\cdot C,\dot{w}\cdot C\circ C'\preceq C''$ and so
\[
 k_{\dot{w}\cdot C\circ C'}=k_{\dot{w}\cdot C\circ C',\langle C''\rangle}k_{C''}\quad k_{C'\circ\dot{w}\cdot C}=k_{C'\circ\dot{w}\cdot C,\langle C''\rangle}k_{C''}\,.
\]
The equality to check is then a consequence of the definition of the braiding operators (Definition~\ref{definition:braidingoperators}).
\end{proof}

\section{Comparison of critical and \texorpdfstring{$2d$}{2d} induction systems}
\label{section:comparisoninductionsystems}

\subsection{From critical to shuffle induction system}
We let $V$ be a representation of a reductive group $G$. We let $T'$ be an auxiliary torus acting on $V$ such that the actions of $G$ and $T'$ commute, and $f\colon V\rightarrow \BoC$ a $G\times T'$-invariant function on $V$. We denote by $f_F\colon V^F\rightarrow \BoC$ the $G^F\times T'$-invariant function induced. We assume that there is a $\BoC^*$-action on $V$ commuting with the $G\times T'$-action and such that all weights on $V$ are nonnegative and $f$ is homogeneous with strictly positive weight. We let $V'\coloneqq V^{\BoC^*}$. Then, for any $F\in\FF$, $f_F^{-1}(0)$ contracts $G^F\times T'$-equivariantly onto $V'^F$. Therefore, the restriction morphism
\[
 \HO^*(V^F/(G^F\times T'),\BoQ)\rightarrow\HO^*(f_F^{-1}(0)/(G^F\times T'),\BoQ)
\]
is an isomorphism. We let $\imath\colon f_F^{-1}(0)/(G^F\times T')\rightarrow V^{F}/(G^F\times T')$ be the inclusion.

There is a canonical morphism $\varphi_f\BoQ_{V^F/(G^F\times T')}\rightarrow \imath_*\imath^*\BoQ_{V^F/(G^F\times T')}$ coming from the distinguished triangle relating vanishing and nearby cycles \eqref{equation:distinguishedtrianglevanishingnearby}. There is also an adjunction morphism $\BoQ_{V^F/(G^F\times T')}\rightarrow\imath_*\imath^*\BoQ_{V^F/(G^F\times T')}$. By applying the derived global sections functor, we obtain the morphisms
\[
 \CH^{T'}_{G,V,f,F}\xrightarrow{s}\HO^*(f_F^{-1}(0)/(G^F\times T'),\imath^*\BoQ)\xleftarrow{r}\CH^{T'}_{G,V,F}\,.
\]
Since $r$ is an isomorphism, we may set $\xi_F\coloneqq r^{-1}s\colon\CH^{T'}_{G,V,f,F}\rightarrow\CH^{T'}_{G,V,F}$.

\begin{proposition}
The morphisms
 \[
  \xi_F\colon\CH_{G,V,f,F}^{T'}\rightarrow\CH_{G,V,F}^{T'}
 \]
for $F\in\FF$ give a morphism of induction-restriction systems.
\end{proposition}
\begin{proof}
We have to show that the $\xi_F$'s commute with both the induction and restriction morphisms.

We prove the commutation with the induction morphisms. We let $C\in\FC$ and $F\in\FF$ such that $C\preceq F$. We consider the induction diagram \eqref{equation:inductiondiagramcritical}
\[
 V^{\langle C\rangle}/(G^{\langle C\rangle}\times T')\xleftarrow{q}V^{C\geq 0,F}/(G^{C\geq 0,F}\times T')\xrightarrow{p} V^F/(G^F\times T')
\]
where $q=q_{C,F}$ and $p=p_{C,F}$.

We consider the diagrams below, in which complexes on zero loci are understood to be pushed forward by their closed immersions into the ambient stacks:
\[
% https://q.uiver.app/#q=WzAsNixbMCwwLCJcXEJvUV97Vl57XFxsYW5nbGUgQ1xccmFuZ2xlfS8oR157XFxsYW5nbGUgQ1xccmFuZ2xlfVxcdGltZXMgVCl9Il0sWzAsMSwiXFxCb1Ffe2Zfe1xcbGFuZ2xlIENcXHJhbmdsZX1eey0xfSgwKS8oR157XFxsYW5nbGUgQ1xccmFuZ2xlfVxcdGltZXMgVCl9Il0sWzAsMiwiXntcXEZwfVxcdmFycGhpX3tmX3tcXGxhbmdsZSBDXFxyYW5nbGV9fVxcQm9RX3tWXntcXGxhbmdsZSBDXFxyYW5nbGV9LyhHXntcXGxhbmdsZSBDXFxyYW5nbGV9XFx0aW1lcyBUKX0iXSxbMSwwLCJxXypcXEJvUV97Vl57Q1xcZ2VxIDB9LyhHXntDXFxnZXEgMH1cXHRpbWVzIFQpfSJdLFsxLDIsIl57XFxGcH1cXHZhcnBoaV97Zl97XFxsYW5nbGUgQ1xccmFuZ2xlfX1xXypcXEJvUV97Vl57Q1xcZ2VxIDB9LyhHXntDXFxnZXEgMH1cXHRpbWVzIFQpfSJdLFsxLDEsInFfKlxcQm9RX3soZl97XFxsYW5nbGUgQ1xccmFuZ2xlfVxcY2lyYyBxKV57LTF9KDApLyhHXntDXFxnZXEgMH1cXHRpbWVzIFQpfSJdLFswLDNdLFsxLDVdLFsyLDRdLFs0LDVdLFszLDVdLFswLDFdLFsyLDFdXQ==
\begin{tikzcd}
	{\BoQ_{V^{\langle C\rangle}/(G^{\langle C\rangle}\times T')}} & {q_*\BoQ_{V^{C\geq 0,F}/(G^{C\geq 0,F}\times T')}} \\
	{\BoQ_{f_{\langle C\rangle}^{-1}(0)/(G^{\langle C\rangle}\times T')}} & {q_*\BoQ_{(f_{\langle C\rangle}\circ q)^{-1}(0)/(G^{C\geq 0,F}\times T')}} \\
	{\varphi_{f_{\langle C\rangle}}\BoQ_{V^{\langle C\rangle}/(G^{\langle C\rangle}\times T')}} & {\varphi_{f_{\langle C\rangle}}q_*\BoQ_{V^{C\geq 0,F}/(G^{C\geq 0,F}\times T')}}
	\arrow[from=1-1, to=1-2]
	\arrow[from=1-1, to=2-1]
	\arrow[from=1-2, to=2-2]
	\arrow[from=2-1, to=2-2]
	\arrow[from=3-1, to=2-1]
	\arrow[from=3-1, to=3-2]
	\arrow[from=3-2, to=2-2]
\end{tikzcd}
\]
where vertical morphisms come from the triangle \eqref{equation:distinguishedtrianglevanishingnearby} and the unit of the adjunction $(\imath^*,\imath_*)$, and
\begin{equation}
\label{equation:compatibilitypb}
 % https://q.uiver.app/#q=WzAsNixbMCwwLCJcXEhPXiooVl57XFxsYW5nbGUgQ1xccmFuZ2xlfS8oR157XFxsYW5nbGUgQ1xccmFuZ2xlfVxcdGltZXMgVCksXFxCb1EpIl0sWzAsMSwiXFxIT14qKGZfe1xcbGFuZ2xlIENcXHJhbmdsZX1eey0xfSgwKS8oR157XFxsYW5nbGUgQ1xccmFuZ2xlfVxcdGltZXMgVCksXFxCb1EpIl0sWzAsMiwiXFxIT14qKFZee1xcbGFuZ2xlIENcXHJhbmdsZX0vKEdee1xcbGFuZ2xlIENcXHJhbmdsZX1cXHRpbWVzIFQpLF57XFxGcH1cXHZhcnBoaV97Zl97XFxsYW5nbGUgQ1xccmFuZ2xlfX1cXEJvUSkiXSxbMSwwLCJcXEhPXiooVl57Q1xcZ2VxIDB9LyhHXntDXFxnZXEgMH1cXHRpbWVzIFQpLFxcQm9RKSJdLFsxLDIsIlxcSE9eKihWXntDXFxnZXEgMH0vKEdee0NcXGdlcSAwfVxcdGltZXMgVCksXntcXEZwfVxcdmFycGhpX3tmX3tcXGxhbmdsZSBDXFxyYW5nbGV9XFxjaXJjIHF9XFxCb1EpIl0sWzEsMSwiXFxIT14qKChmX3tcXGxhbmdsZSBDXFxyYW5nbGV9XFxjaXJjIHEpXnstMX0oMCkvKEdee0NcXGdlcSAwfVxcdGltZXMgVCksXFxCb1EpIl0sWzAsM10sWzEsNV0sWzIsNF0sWzQsNV0sWzMsNV0sWzAsMV0sWzIsMV1d
\begin{tikzcd}
	{\HO^*(V^{\langle C\rangle}/(G^{\langle C\rangle}\times T'),\BoQ)} & {\HO^*(V^{C\geq 0,F}/(G^{C\geq 0,F}\times T'),\BoQ)} \\
	{\HO^*(f_{\langle C\rangle}^{-1}(0)/(G^{\langle C\rangle}\times T'),\BoQ)} & {\HO^*((f_{\langle C\rangle}\circ q)^{-1}(0)/(G^{C\geq 0,F}\times T'),\BoQ)} \\
	{\HO^*(V^{\langle C\rangle}/(G^{\langle C\rangle}\times T'),\varphi_{f_{\langle C\rangle}}\BoQ)} & {\HO^*(V^{C\geq 0,F}/(G^{C\geq 0,F}\times T'),\varphi_{f_{\langle C\rangle}\circ q}\BoQ)}
	\arrow[from=1-1, to=1-2]
	\arrow[from=1-1, to=2-1]
	\arrow[from=1-2, to=2-2]
	\arrow[from=2-1, to=2-2]
	\arrow[from=3-1, to=2-1]
	\arrow[from=3-1, to=3-2]
	\arrow[from=3-2, to=2-2]
	\arrow[from=3-1, to= 1-1, bend left=80,"r^{-1}s"]
\end{tikzcd}
\end{equation}
where the second diagram is obtained by taking derived global sections and using the exchange morphism $\varphi q_*\to q_*\varphi$. Naturality gives the squares on the source, and Lemma~\ref{lemma:exchange-zero-fiber} gives compatibility with this exchange morphism and restriction to the zero fiber. Hence the second diagram commutes.

The analogous diagrams for the pushforward by $p$ instead of pullback by $q$ give the following commutative diagrams, where for convenience, we set $\dim p\coloneqq \dim(V^{C\geq 0,F}/(G^{C\geq 0,F}\times T'))-\dim(V^{F}/(G^{F}\times T'))$.
\[
 % https://q.uiver.app/#q=WzAsNixbMCwwLCJwXypcXEJvUV97Vl57Q1xcZ2VxIDB9LyhHXntcXGdlcSAwfVxcdGltZXMgVCl9Il0sWzEsMCwiXFxCb1Ffe1ZeRi8oR15GXFx0aW1lcyBUKX1bLTJcXGRpbSBwXSJdLFswLDEsInBfKlxcQm9RX3soZlxcY2lyYyBwKV57LTF9KDApLyhHXntDXFxnZXEgMH1cXHRpbWVzIFQpfSJdLFsxLDEsIlxcQm9RX3tmXnstMX0oMCkvKEdeRlxcdGltZXMgVCl9Wy0yXFxkaW0gcF0iXSxbMSwyLCJee1xcRnB9XFx2YXJwaGlfe2ZfRn1cXEJvUV97Vl5GLyhHXkZcXHRpbWVzIFQpfVstMlxcZGltIHBdIl0sWzAsMiwiXntcXEZwfVxcdmFycGhpX3tmX3tGfX1wXypcXEJvUV97Vl57Q1xcZ2VxIDB9LyhHXntDXFxnZXEgMH1cXHRpbWVzIFQpfSJdLFs0LDNdLFsxLDNdLFswLDFdLFsyLDNdLFswLDJdLFs1LDJdLFs1LDRdXQ==
\begin{tikzcd}
	{p_*\BoQ_{V^{C\geq 0,F}/(G^{C\geq 0,F}\times T')}} & {\BoQ_{V^F/(G^F\times T')}\otimes\SL^{\dim p}} \\
	{p_*\BoQ_{(f\circ p)^{-1}(0)/(G^{C\geq 0,F}\times T')}} & {\BoQ_{f^{-1}(0)/(G^F\times T')}\otimes\SL^{\dim p}} \\
	{\varphi_{f_{F}}p_*\BoQ_{V^{C\geq 0,F}/(G^{C\geq 0,F}\times T')}} & {\varphi_{f_F}\BoQ_{V^F/(G^F\times T')}\otimes\SL^{\dim p}}
	\arrow[from=1-1, to=1-2]
	\arrow[from=1-1, to=2-1]
	\arrow[from=1-2, to=2-2]
	\arrow[from=2-1, to=2-2]
	\arrow[from=3-1, to=2-1]
	\arrow[from=3-1, to=3-2]
	\arrow[from=3-2, to=2-2]
\end{tikzcd}
\]
and
\begin{equation}
\label{equation:compatibilitypf}
 % https://q.uiver.app/#q=WzAsNixbMCwwLCJcXEhPXiooVl57Q1xcZ2VxIDB9LyhHXntcXGdlcSAwfVxcdGltZXMgVCksXFxCb1EpIl0sWzEsMCwiXFxIT14qKFZeRi8oR15GXFx0aW1lcyBUKSxcXEJvUSlbLTJcXGRpbSBwXSJdLFswLDEsIlxcSE9eKigoZlxcY2lyYyBwKV57LTF9KDApLyhHXntDXFxnZXEgMH1cXHRpbWVzIFQpLFxcQm9RKSJdLFsxLDEsIlxcSE9eKihmXnstMX0oMCkvKEdeRlxcdGltZXMgVCksXFxCb1EpWy0yXFxkaW0gcF0iXSxbMSwyLCJcXEhPXiooVl5GLyhHXkZcXHRpbWVzIFQpLF57XFxGcH1cXHZhcnBoaV97Zl9GfVxcQm9RKVstMlxcZGltIHBdIl0sWzAsMiwiXFxIT14qKFZee0NcXGdlcSAwfS8oR157Q1xcZ2VxIDB9XFx0aW1lcyBUKSxee1xcRnB9XFx2YXJwaGlfe2Zfe0Z9XFxjaXJjIHB9XFxCb1EpIl0sWzQsMywicyIsMl0sWzEsMywiciJdLFswLDFdLFsyLDNdLFswLDJdLFs1LDJdLFs1LDRdLFs0LDEsInJeey0xfXMiLDJdXQ==
\begin{tikzcd}
	{\HO^*(V^{C\geq 0,F}/(G^{C\geq 0,F}\times T'),\BoQ)} & {\HO^*(V^F/(G^F\times T'),\BoQ)\otimes\SL^{\dim p}} \\
	{\HO^*((f\circ p)^{-1}(0)/(G^{C\geq 0,F}\times T'),\BoQ)} & {\HO^*(f^{-1}(0)/(G^F\times T'),\BoQ)\otimes\SL^{\dim p}} \\
	{\HO^*(V^{C\geq 0,F}/(G^{C\geq 0,F}\times T'),\varphi_{f_{F}\circ p}\BoQ)} & {\HO^*(V^F/(G^F\times T'),\varphi_{f_F}\BoQ)\otimes\SL^{\dim p}}
	\arrow[from=1-1, to=1-2]
	\arrow[from=1-1, to=2-1]
	\arrow["r", from=1-2, to=2-2]
	\arrow[from=2-1, to=2-2]
	\arrow[from=3-1, to=2-1]
	\arrow[from=3-1, to=3-2]
	\arrow["{r^{-1}s}"', from=3-2, to=1-2,bend right=90]
	\arrow["s"', from=3-2, to=2-2]
\end{tikzcd}
\end{equation}
In the pushforward diagram, proper base change identifies the zero-fiber terms, and properness makes the vanishing-cycle exchange an isomorphism. Lemma~\ref{lemma:exchange-zero-fiber}, applied to the proper trace, proves compatibility of these identifications with the vertical arrows. By combining the outer squares of \eqref{equation:compatibilitypb} and \eqref{equation:compatibilitypf}, we obtain the commutative square
\[
 % https://q.uiver.app/#q=WzAsNCxbMCwxLCJcXENIXlRfe0csVixmLFxcbGFuZ2xlIENcXHJhbmdsZX0iXSxbMSwxLCJcXENIXlRfe0csVixmLEZ9Il0sWzAsMCwiXFxDSF5UX3tHLFYsXFxsYW5nbGUgQ1xccmFuZ2xlfSJdLFsxLDAsIlxcQ0heVF97RyxWLEZ9Il0sWzAsMSwiXFxJbmRfe0N9XkYiLDJdLFswLDIsIlxceGlfe1xcbGFuZ2xlIENcXHJhbmdsZX0iXSxbMSwzLCJcXHhpX0YiLDJdLFsyLDMsIlxcSW5kX3tDfV5GIl1d
\begin{tikzcd}
	{\CH^{T'}_{G,V,\langle C\rangle}} & {\CH^{T'}_{G,V,F}} \\
	{\CH^{T'}_{G,V,f,\langle C\rangle}} & {\CH^{T'}_{G,V,f,F}}
	\arrow["{\Ind_{C}^F}", from=1-1, to=1-2]
	\arrow["{\xi_{\langle C\rangle}}", from=2-1, to=1-1]
	\arrow["{\Ind_{C}^F}"', from=2-1, to=2-2]
	\arrow["{\xi_F}"', from=2-2, to=1-2]
\end{tikzcd}
\]
which is precisely the compatibility of the $\xi_F$'s with the induction morphisms.

For restriction, write $i=i_{\langle C\rangle,F}$. The critical pullback is induced by the composite
\[
 \varphi_f\BoQ\longrightarrow\varphi_f i_*\BoQ
 \longrightarrow i_*\varphi_{f\circ i}\BoQ,
\]
where the second arrow is the exchange transformation; it need not be an isomorphism since $i$ is not generally proper. The ordinary-restriction part of Lemma~\ref{lemma:exchange-zero-fiber} makes this composite compatible with $\epsilon\colon\varphi_f\to\imath_*\imath^*$ and with the ordinary pullback on the zero fibers. Consequently, on global cohomology,
\[
 s_{\langle C\rangle}i_{\mathrm{crit}}^*=i_0^*s_F,
 \qquad r_{\langle C\rangle}i_{\mathrm{ord}}^*=i_0^*r_F.
\]
Since both $r$ maps are isomorphisms, these identities imply
$\xi_{\langle C\rangle}i_{\mathrm{crit}}^*=i_{\mathrm{ord}}^*\xi_F$.
All maps are linear over the appropriate equivariant coefficient rings. Multiplication by the common factor $\Eu_{\Fg,C,F}^{T'}/\Eu_{V,C,F}^{T'}$ therefore gives compatibility with restriction in the localized target. Weyl equivariance follows from the same functoriality, and coefficient linearity also gives compatibility with the braiding operators. This proves the assertion in the sense of Definition~\ref{definition:localized-induction-restriction-system}.
\end{proof}

\subsection{From $2d$ to shuffle induction system}
\label{subsection:from2dtoshuffle}
Let $V$ be a representation of a reductive group $G$. We let $f\colon \Tan^*V\times\Fg\rightarrow\BoC$ be the function obtained by contracting the moment map (\S\ref{subsubsection:the2dpotential}). We let $T'$ be an auxiliary torus acting on $V$ such that the actions of $T'$ and $G$ commute and $f$ is $G\times T'$-invariant. If we let $\BoC^*$ act with weight $1$ on the whole of $\Tan^*V\times\Fg$, then the potential has weight $3$.

We let $p\colon (\Tan^*V\times\Fg)^F/(G^F\times T')\rightarrow\Tan^*V^F/(G^F\times T')$ be the projection. This is an affine fibration. We let $\FZ_F\coloneqq p^{-1}(\mu_F^{-1}(0)/(G^{F}\times T'))$. The restriction of $p$ to $\FZ_F$ is an affine fibration, which therefore induces an isomorphism
\begin{equation}
 \label{equation:pbaffinefibration}
\HO^{\rmBM}_*(\mu_F^{-1}(0)/(G^F\times T'),\BoQ)\rightarrow \HO^{\rmBM}_*(\FZ_F,\BoQ).
\end{equation}
The pushforward for the closed embedding
\[
 \FZ_F\rightarrow(\Tan^*V\times\Fg)^F/(G^F\times T')
\]
gives a morphism
\begin{equation}
 \label{equation:pfclosedembedding}
 \HO^{\rmBM}_*(\FZ_F,\BoQ)\rightarrow\HO^{*}((\Tan^*V\times\Fg)^F/(G^F\times T'),\BoQ)\,.
\end{equation}

By composing \eqref{equation:pbaffinefibration} and \eqref{equation:pfclosedembedding}, we obtain a morphism
\[
 \overline{\xi}_F\colon \CH_{G,\Tan^*V,\Fg,\mu,F}^{T'}\rightarrow \CH^{T'}_{G,\Tan^*V\times\Fg,F}\,.
\]

\begin{proposition}
\label{proposition:commutativeinductiondimred2d}
We set $W=\Fg$. We have a commutative diagram
 \[
  % https://q.uiver.app/#q=WzAsMyxbMCwwLCJcXENIX3tHLFxcVGFuXipWLFxcRmdeKixGfV5UIl0sWzEsMSwiXFxDSF5UX3tHLFxcVGFuXipWXFx0aW1lc1xcRmcsRn0iXSxbMSwwLCJcXENIX3tHLFxcVGFuXipWXFx0aW1lc1xcRmcsZixGfSJdLFswLDIsIlxcbWF0aHNme2RyfV9GIl0sWzAsMSwiXFxvdmVybGluZXtcXHhpfV9GIiwyXSxbMiwxLCJcXHhpX0YiXV0=
\begin{tikzcd}
	{\CH_{G,\Tan^*V,\Fg,\mu,F}^{T'}} & {\CH^{T'}_{G,\Tan^*V\times\Fg,f,F}} \\
	& {\CH^{T'}_{G,\Tan^*V\times\Fg,F}}
	\arrow["{\mathsf{dr}_F}", from=1-1, to=1-2]
	\arrow["{\overline{\xi}_F}"', from=1-1, to=2-2]
	\arrow["{\xi_F}", from=1-2, to=2-2]
\end{tikzcd}
 \]
 whose morphisms commute with the induction morphisms $\Ind_{C}^F$ up to the sign $(-1)^{\dim W^{C<0,F}}$, that is $\mathsf{dr}_F\Ind_{C}^F=(-1)^{\dim W^{C<0,F}}\Ind_C^F\mathsf{dr}_{\langle C\rangle}$, $\xi_F\Ind_{C}^F=\Ind_C^F\xi_{\langle C\rangle}$, $\overline{\xi}_F\Ind_{C}^F=(-1)^{\dim W^{C<0,F}}\Ind_{C}^F\overline{\xi}_{\langle C\rangle}$.
\end{proposition}
\begin{proof}
 Since we have already proven (Proposition~\ref{proposition:comparison2d3dmultiplications}) that the dimensional reduction isomorphisms $\mathsf{dr}_F$ commute with the induction morphisms $\Ind_{C}^F$ up to the sign $(-1)^{\dim W^{C<0,F}}$ and $\xi_F$ ($F\in\FF$) commute with the induction morphisms $\Ind_{C}^F$, the compatibility of $\overline{\xi}_F$ with induction morphisms up to the sign $(-1)^{\dim W^{C<0,F}}$ will follow from the commutativity of the diagram.
 
We now prove that the diagram commutes. To lighten the notation in the remainder of the proof, set $U=\Tan^*V$ and $W=\Fg$. We let $\FM_{G,U\times W,F}^{T'}\coloneqq (U^F\times W^F)/(G^F\times T')$ and we let $\overline{\imath}\colon\FZ_{G,U,W,\mu,F}^{T'}\rightarrow\FM_{G,U\times W,F}^{T'}$ be the closed immersion. There is a natural pushforward morphism
\[
 \overline{\imath}_*\BD\BoQ_{\FZ_{G,U,W,\mu,F}^{T'}}\rightarrow \BD\BoQ_{\FM_{G,U\times W,F}^{T'}}\cong\BoQ_{\FM_{G,U\times W,F}^{T'}}\otimes\SL^{-\dim\FM_{G,U\times W,F}^{T'}}\,.
\]
By applying the natural transformation $\varphi_f\rightarrow\imath_*\imath^*$ to it, we obtain the commutative square
\[
 % https://q.uiver.app/#q=WzAsNCxbMCwwLCJcXG92ZXJsaW5le1xcaW1hdGh9XypcXEJEXFxCb1Ffe1xcRlpfe0csVixXLFxcbXUsRn1ee1QnfX0iXSxbMCwxLCJcXG92ZXJsaW5le1xcaW1hdGh9XypcXEJEXFxCb1Ffe1xcRlpfe0csVixXLFxcbXUsRn1ee1QnfX0iXSxbMSwwLCJcXHZhcnBoaV9mXFxCb1Ffe1xcRk1fe0csVlxcdGltZXMgVyxGfV57VCd9fSJdLFsxLDEsIlxcaW1hdGhfKlxcQm9RX3tmX0Zeey0xfSgwKX0iXSxbMCwxLCJcXHNpbSIsMl0sWzAsMl0sWzIsM10sWzEsM11d
\begin{tikzcd}
	{\overline{\imath}_*\BD\BoQ_{\FZ_{G,U,W,\mu,F}^{T'}}} & {\varphi_f\BoQ_{\FM_{G,U\times W,F}^{T'}}}\otimes\SL^{-\dim\FM_{G,U\times W,F}^{T'}} \\
	{\overline{\imath}_*\BD\BoQ_{\FZ_{G,U,W,\mu,F}^{T'}}} & {\imath_*\BoQ_{f_F^{-1}(0)/(G^F\times T')}}\otimes\SL^{-\dim\FM_{G,U\times W,F}^{T'}}
	\arrow[from=1-1, to=1-2]
	\arrow["\sim"', from=1-1, to=2-1]
	\arrow[from=1-2, to=2-2]
	\arrow[from=2-1, to=2-2]
\end{tikzcd}
\]
while applying the natural transformation $\id\rightarrow\imath_*\imath^*$, we get the commutative square
\[
 % https://q.uiver.app/#q=WzAsNixbMCwwXSxbMCwxLCJcXG92ZXJsaW5le1xcaW1hdGh9XypcXEJEXFxCb1Ffe1xcRlpfe0csVixXLFxcbXUsRn1ee1QnfX0iXSxbMSwwXSxbMSwxLCJcXGltYXRoXypcXEJvUV97Zl9GXnstMX0oMCl9Il0sWzAsMiwiXFxvdmVybGluZXtcXGltYXRofV8qXFxCRFxcQm9RX3tcXEZaX3tHLFYsVyxcXG11LEZ9XntUJ319Il0sWzEsMiwiXFxCb1Ffe1xcRk1fe0csVlxcdGltZXMgVyxGfV57VCd9fSJdLFsxLDNdLFs0LDEsIlxcc2ltIl0sWzQsNV0sWzUsM11d
\begin{tikzcd}
	{} & {} \\
	{\overline{\imath}_*\BD\BoQ_{\FZ_{G,U,W,\mu,F}^{T'}}} & {\imath_*\BoQ_{f_F^{-1}(0)/(G^F\times T')}}\otimes\SL^{-\dim\FM_{G,U\times W,F}^{T'}} \\
	{\overline{\imath}_*\BD\BoQ_{\FZ_{G,U,W,\mu,F}^{T'}}} & {\BoQ_{\FM_{G,U\times W,F}^{T'}}}\otimes\SL^{-\dim\FM_{G,U\times W,F}^{T'}}
	\arrow[from=2-1, to=2-2]
	\arrow["\sim", from=3-1, to=2-1]
	\arrow[from=3-1, to=3-2]
	\arrow[from=3-2, to=2-2]
\end{tikzcd}
\]
Superposing these squares and taking derived global sections, we obtain the commutative diagram
\[
 % https://q.uiver.app/#q=WzAsNixbMCwwLCJcXENIX3tHLFYsVyxcXG11LEd9XntUJ30iXSxbMSwwLCJcXENIX3tHLFZcXHRpbWVzIFcsZixGfV57VCd9Il0sWzAsMSwiXFxDSF97RyxWLFcsXFxtdSxHfV57VCd9Il0sWzAsMiwiXFxDSF97RyxWLFcsXFxtdSxHfV57VCd9Il0sWzEsMiwiXFxDSF97RyxWXFx0aW1lcyBXLEZ9XntUJ30iXSxbMSwxLCJcXEhPXiooZl9GXnstMX0oMCksXFxCb1EpIl0sWzEsNSwicyJdLFswLDEsIlxcbWF0aHNme2RyfSJdLFs0LDUsInIiLDJdLFszLDQsIlxcb3ZlcmxpbmV7XFx4aX1fRiIsMl0sWzAsMiwiXFxzaW0iLDJdLFszLDIsIlxcc2ltIl0sWzIsNV0sWzEsNCwiXFx4aV9GPXJeey0xfXMiLDAseyJjdXJ2ZSI6LTV9XV0=
\begin{tikzcd}
	{\CH_{G,U,W,\mu,F}^{T'}} & {\CH_{G,U\times W,f,F}^{T'}} \\
	{\CH_{G,U,W,\mu,F}^{T'}} & {\HO^*(f_F^{-1}(0)/(G^F\times T'),\BoQ)} \\
	{\CH_{G,U,W,\mu,F}^{T'}} & {\CH_{G,U\times W,F}^{T'}}
	\arrow["{\mathsf{dr}}", from=1-1, to=1-2]
	\arrow["\sim"', from=1-1, to=2-1]
	\arrow["s", from=1-2, to=2-2]
	\arrow["{\xi_F=r^{-1}s}", bend left=110, from=1-2, to=3-2]
	\arrow[from=2-1, to=2-2]
	\arrow["\sim", from=3-1, to=2-1]
	\arrow["{\overline{\xi}_F}"', from=3-1, to=3-2]
	\arrow["r"', from=3-2, to=2-2]
\end{tikzcd}
\]
which concludes the proof.
\end{proof}

\begin{lemma}
\label{lemma:notorsionS1}
Let $\CO$ be a nilpotent orbit in the Lie algebra of a reductive group $G$. Let $T'$ act on $\CO$, commuting with $G$, and let $S=\lambda(\BoC^*)\subset T'$ be a one-dimensional subtorus acting trivially on $\CO$. Let $I$ be the kernel of restriction
\[
 \HO^*_{T'\times G}(\pt)\longrightarrow\HO^*_S(\pt).
\]
Then $\HO^*_{T'\times G}(\CO)$ is concentrated in even degrees and has no $\HO^*_{T'\times G}(\pt)\setminus I$-torsion.
\end{lemma}
\begin{proof}
The group $T'\times G$ acts transitively on $\CO$. Let $H$ be the stabilizer of a point. Since $S$ acts trivially, $S\subset H$, and
\[
 \HO^*_{T'\times G}(\CO)\cong\HO^*_H(\pt).
\]
The rational cohomology of $BH$ embeds in the polynomial ring of a maximal torus of a Levi subgroup of $H^\circ$, by first taking the Weyl group invariants and then the component group invariants. It is therefore a domain concentrated in even degrees. If $a\notin I$, its image in $\HO^*_H(\pt)$ is nonzero, since its further restriction to $S$ is nonzero. Multiplication by this image is injective, as required.
\end{proof}

\begin{proposition}
\label{proposition:changeflavourtorus}
Let $V$ be a representation of a reductive group $G$, $f\colon V\rightarrow\BoC$ a $G$-invariant regular function, and $T'$ an auxiliary torus whose action commutes with that of $G$ and preserves $f$. Assume that the mixed Hodge structure on $\CH_{G,V,f,F}$, with no auxiliary torus, is pure. Then, for any subtorus $\tilde T\subset T'$, the equivariant spectral sequence degenerates and gives
\[
 \operatorname{gr}\CH_{G,V,f,F}^{T'}
 \cong\CH_{G,V,f,F}^{\tilde T}\otimes\HO^*_{T'/\tilde T}.
\]
All terms are pure. After choosing a splitting, this gives an isomorphism of graded mixed Hodge structures
\[
 \CH_{G,V,f,F}^{T'}
 \cong\CH_{G,V,f,F}^{\tilde T}\otimes\HO^*_{T'/\tilde T}.
\]
After choosing a complementary torus up to isogeny, the cohomology on the left is free over the polynomial ring of the complementary torus.
\end{proposition}
\begin{proof}
Apply the equivariant spectral sequence first for $\tilde T$ and then for $T'/\tilde T$. Since the tori are connected, their actions on the cohomology of the fibers are trivial. Purity implies that the terms of total degree $n$ have weight $n$; all differentials therefore vanish. The resulting cohomology is pure, and the filtration splits degree by degree in the category of polarizable pure Hodge structures. Choosing homogeneous lifts of a basis of the fiber cohomology gives the asserted freeness. This is the usual purity argument; compare \cite[Theorem~14.1]{goresky1998equivariant} and \cite[Theorem~9.6]{davison2016integrality}.
\end{proof}

\begin{corollary}
\label{corollary:changeflavourtorus}
 Let $V$ be a representation of a reductive group $G$. We let $f\colon \Tan^*V\times\Fg\rightarrow\BoC$ be the contraction of the moment map of $\Tan^*V$. For any injective morphism of tori $\tilde{T}\rightarrow T'$ such that $T'$ acts on $\Tan^*V\times\Fg$, the actions of $T'$ and $G$ commute and $f$ is $T'$-invariant, we have an isomorphism of mixed Hodge structures
 \[
  \CH_{G,\Tan^*V\times\Fg,f,F}^{T'}\cong \CH^{\tilde{T}}_{G,\Tan^*V\times\Fg,f,F}\otimes \HO^*_{T'/\tilde{T}}\,.
 \]
The displayed isomorphism depends on a splitting of the equivariant filtration, as in Proposition~\ref{proposition:changeflavourtorus}; the associated-graded identification is canonical. Moreover, $\CH_{G,\Tan^*V\times\Fg,f,F}^{T'}$ and $\CH^{\tilde{T}}_{G,\Tan^*V\times\Fg,f,F}\otimes \HO^*_{T'/\tilde{T}}$ carry pure mixed Hodge structures.
\end{corollary}
\begin{proof}
 This is Proposition~\ref{proposition:changeflavourtorus}, using the fact that for the representation $\Tan^*V\times\Fg$ of $G$ and the function $f$, the mixed Hodge structure on $\CH_{G,\Tan^*V\times\Fg,f,F}$ (for trivial auxiliary torus $T'$) is pure by \cite[Corollary~1.11]{hennecart2024cohomological2}.
\end{proof}

\begin{proposition}
\label{proposition:localizedpushforward}
Let $V,W$ be representations of a reductive group $G$, let $T'$ be an auxiliary torus commuting with $G$, and let $\mu\colon V\rightarrow W^*$ be a $G\times T'$-equivariant regular map with $\mu(0)=0$. Assume that $V^{T\times T'}=\{0\}$, where $T\subset G$ is a maximal torus. For $F\in\FF$, set
\[
 R_F\coloneqq\HO^*_{G^F\times T'}(\pt),\qquad K_F\coloneqq\operatorname{Frac}(R_F),
\]
and let $\imath\colon\mu_F^{-1}(0)\hookrightarrow V^F$ be the inclusion. Then the localized pushforward
\[
 \imath_*\colon
 \HO^{\rmBM}_{G^F\times T'}(\mu_F^{-1}(0),\BoQ)\otimes_{R_F}K_F
 \longrightarrow
 \HO^{\rmBM}_{G^F\times T'}(V^F,\BoQ)\otimes_{R_F}K_F
\]
is an isomorphism.
\end{proposition}
\begin{proof}
Both $V^F$ and $\mu_F^{-1}(0)$ have $T\times T'$-fixed locus $\{0\}$. Equivariant localization therefore makes both inclusions of $\{0\}$ isomorphisms in torus-equivariant Borel--Moore homology; compare the argument for equivariant Chow groups in \cite[Theorem~1]{edidin1998localization}. The inclusion into $V^F$ factors through $\imath$, so $\imath_*$ is an isomorphism after localization as well. Passing to $\sfW^F$-invariants gives the stated result. Under the Thom identification for $V^F$, the composite pushforward from $\{0\}$ is multiplication by its Euler class, which becomes invertible after localization.
\end{proof}

Let $V$ be a representation of a reductive group $G$. We assume that an auxiliary torus $T'$ acts linearly on $\Tan^*V\times\Fg$, preserves its factors and the potential $f$, and commutes with $G$. Let $T_2\coloneqq(\BoC^*)^2$ act on $\Tan^*V\times\Fg$ with weight $(1,0)$ on $V$, $(0,1)$ on $V^*$ and $(-1,-1)$ on $\Fg$.

\begin{proposition}
\label{proposition:embeddinginductionsystems}

 We assume that there exists a morphism $(\BoC^*)^2\rightarrow T'$ such that the action of $(\BoC^*)^2$ on $\Tan^*V\times\Fg$ induced is that of $T_2$. Then, the morphisms $\overline{\xi}_F\colon \CH_{G,\Tan^*V,\Fg,\mu,F}^{T'}\rightarrow \CH_{G,\Tan^*V\times\Fg,F}^{T'}$ ($F\in\FF$) give an embedding of induction systems for the twisted induction $(-1)^{\dim\Fg^{C<0,F}}\Ind_{C}^F$ on either the source or the target.
\end{proposition}
\begin{proof}
We already know by Proposition~\ref{proposition:commutativeinductiondimred2d} that $\overline{\xi}_F$ commutes with the induction morphisms up to the given sign. We prove its injectivity. We may replace $G,V$ by $G^F,V^F$, so we omit $F$ from the notation.

We first reduce to $T'=T_2$. If $V=0$, the map is an isomorphism by the Thom identification. Otherwise, the prescribed action makes the morphism $T_2\rightarrow T'$ injective. Choose a complementary torus up to isogeny. By Corollary~\ref{corollary:changeflavourtorus} and dimensional reduction, restriction to $T_2$ identifies the source with its quotient by the positive-degree parameters of the complementary torus. The same is true for the target, which is a polynomial ring. Injectivity after this specialization implies injectivity before specialization: for one parameter $z$, an element in the kernel is divisible by $z$; since multiplication by $z$ is injective on the target, its quotient by $z$ is again in the kernel. Repeating this contradicts boundedness below of the cohomological grading unless the element is zero. We apply this argument successively to the complementary parameters.

We consider the isomorphism $T'_2\coloneqq(\BoC^*)^2\rightarrow T_2$, $(a,b)\mapsto(ab,a^{-1})$. Then, $T'_2$ acts on $V$, $V^*$ and $\Fg$ with weights $(1,1)$, $(-1,0)$ and $(0,-1)$, respectively. We write $T'_2=\BoC^*_1\times\BoC^*_2$, put $u_i=c_1(\mathrm{pr}_i)$, and let
\[
 R\coloneqq\HO^*_{G\times T'_2}(\pt),\qquad
 I_1\coloneqq\ker\bigl(R\longrightarrow\HO^*_{\BoC^*_1}(\pt)\bigr),
 \qquad S_1\coloneqq R\setminus I_1.
\]
The fixed locus of $\BoC^*_1$ on $\mu^{-1}(0)$, and on $\Tan^*V$, is $\{0\}$. Equivariant localization at the prime $I_1$ therefore makes the pushforward
\[
 \imath_*\colon\HO^{\rmBM}_{G\times T'_2}(\mu^{-1}(0),\BoQ)
 \longrightarrow\HO^{\rmBM}_{G\times T'_2}(\Tan^*V,\BoQ)
\]
an isomorphism. One may apply torus localization and then take Weyl invariants; compare \cite[Theorem~6.2]{goresky1998equivariant} and the Borel--Moore analogue of \cite[Theorem~1]{edidin1998localization}. Under the Thom identifications, this is $\overline\xi$. It suffices therefore to prove that its source has no $S_1$-torsion.

By Corollary~\ref{corollary:changeflavourtorus} and dimensional reduction, the source is free over $\BoQ[u_2]$. Its localization at $u_2$ is therefore injective. Applying dimensional reduction first along $\Fg$ and then along $V^*$ gives, with the usual Tate twists suppressed,
\[
\begin{aligned}
 \HO^{\rmBM}_{G\times T'_2}(\mu^{-1}(0),\BoQ)
 &\cong\HO^*_{G\times T'_2}(\Tan^*V\times\Fg,\varphi_f\BoQ)\\
 &\cong\HO^{\rmBM}_{G\times T'_2}(\CC(V,\Fg),\BoQ),
\end{aligned}
\]
where $\CC(V,\Fg)\coloneqq\{(v,\xi)\in V\times\Fg\mid\xi\cdot v=0\}$. We let $\CC^{\nil}(V,\Fg)$ be the closed subset on which $\xi$ is nilpotent.

We now justify the localization to this closed subset. Choose homogeneous generators $p_1,\ldots,p_r$ of $\BoC[\Fg]^G$ of positive degrees $d_1,\ldots,d_r$. Their common zero locus is the nilpotent cone. The complement of $\CC^{\nil}(V,\Fg)$ is therefore covered by the invariant open subsets on which $p_j(\xi)\neq0$. On each such open subset, $p_j$ is a nowhere-vanishing equivariant section of the line bundle with character $-d_j u_2$. This line bundle is equivariantly trivial, so its first Chern class $-d_j u_2$ acts by zero on equivariant Borel--Moore homology. After inverting $u_2$, the homology of this open subset, and of every finite intersection of these open subsets, vanishes. The Mayer--Vietoris sequence for the finite cover then gives vanishing on the whole complement. The closed-open exact sequence yields an isomorphism
\[
 \HO^{\rmBM}_{G\times T'_2}(\CC^{\nil}(V,\Fg),\BoQ)[u_2^{-1}]
 \xrightarrow{\sim}
 \HO^{\rmBM}_{G\times T'_2}(\CC(V,\Fg),\BoQ)[u_2^{-1}].
\]

It remains to prove that $\HO^{\rmBM}_{G\times T'_2}(\CC^{\nil}(V,\Fg),\BoQ)$ has no $S_1$-torsion. We stratify by the finitely many nilpotent orbits $\CO\subset\Fg$. Each stratum $\CC(V,\CO)$ is a vector bundle over $\CO$, with fiber $\ker(\xi\colon V\rightarrow V)$ over $\xi\in\CO$. Since $\BoC^*_1$ acts trivially on $\CO$, Lemma~\ref{lemma:notorsionS1} and the Thom isomorphism show that the homology of each stratum has no $S_1$-torsion and is concentrated in even degrees. A filtration by closed unions of strata therefore gives short exact sequences in Borel--Moore homology. Torsion-freeness for $S_1$ is preserved by extensions, so it holds for $\CC^{\nil}(V,\Fg)$. It also holds after inverting $u_2$. The preceding isomorphism and the injective localization of the source show that $\HO^{\rmBM}_{G\times T'_2}(\mu^{-1}(0),\BoQ)$ has no $S_1$-torsion, as required.
\end{proof}

Since $\CH_{G,\Tan^*V\times\Fg,F}^{T'}$ admits a very explicit description in terms of some invariant polynomials, an interesting question we leave for future investigations is the following.
\begin{question}
\label{question:wheel}
What is the image of the injective morphism $\overline{\xi}_F$? We expect a description in terms of symmetric polynomials satisfying some easily described condition, similar to the wheel relations \cite{neguct2023shuffle} in the context of $K$-theoretic Hall algebras of quivers. See also \cite{neguct2023generators} for a cohomological analogue of wheel relations for cohomological Hall algebras.
\end{question}

We can nevertheless partially answer Question~\ref{question:wheel}. We first formulate a couple of lemmas.

\begin{lemma}
\label{lemma:representationtheory}
 Let $V$ be a representation of a reductive algebraic group $G$. Then, for any $\alpha,\beta\in\rmX^*(T)$ such that $\alpha, \alpha+\beta$ are weights of $V$ and $\beta$ is a root of $G$, there exists $x\in V_{\alpha}$ and $\zeta\in\Fg_{\beta}$ such that $0\neq\zeta\cdot x\in V_{\alpha+\beta}$.
\end{lemma}
\begin{proof}
This is a question about the representation theory of $\mathfrak{sl}_2$. The space
\[
 V_{\alpha+\BoZ\beta}\coloneqq\bigoplus_{t\in\BoZ}V_{\alpha+t\beta}
\]
is stable under $\mathfrak{s}_\beta\coloneqq\Fg_{-\beta}\oplus[\Fg_\beta,\Fg_{-\beta}]\oplus\Fg_\beta\cong\mathfrak{sl}_2$. Its irreducible summands have weight strings of the same parity, all centered at zero. Since the two adjacent weights corresponding to $\alpha$ and $\alpha+\beta$ both occur, one of these strings contains both. The raising operator is nonzero between the corresponding weight spaces, which proves the assertion.
\end{proof}

\begin{lemma}
\label{lemma:equivariantcohomology}
 Let $T$ be a torus acting on $(\BoC^*)^2$ with weight $\alpha\in\rmX^*(T)$ on the first factor and weight $\beta\in\rmX^*(T)$ on the second factor. We assume that the action of $T$ on $(\BoC^*)^2$ is transitive. Then,
 \[
  \HO^*_{T}((\BoC^*)^2)\cong\frac{\Sym(\Ft^*)}{\langle\alpha,\beta\rangle}\,.
 \]
We identify $\rmX^*(T)\subset\Ft^*$. Without the transitivity assumption, the kernel of the natural map $\Sym(\Ft^*)\rightarrow\HO^*_T((\BoC^*)^2)$ is still the ideal $\langle\alpha,\beta\rangle$.
\end{lemma}
\begin{proof}
 By transitivity of the action, we have
 \[
  \HO^*_T((\BoC^*)^2)\cong \HO^*_{\ker(\alpha\times\beta)}(\pt),
 \]
where $\alpha\times\beta\colon T\rightarrow(\BoC^*)^2$. Moreover
\[
 \HO^*_{\ker(\alpha\times\beta)}(\pt)\cong\Sym(\Lie(\ker(\alpha\times\beta))^*)\,.
\]
The displayed isomorphism follows from $\Lie(\ker(\alpha\times\beta))^*\cong\Ft^*/\langle\alpha,\beta\rangle$. In general, the image of $T$ is a subtorus of $(\BoC^*)^2$. Splitting the tori up to isogeny identifies the equivariant cohomology with the cohomology of the classifying stack of its kernel tensored with the cohomology of the quotient torus. The map from $\Sym(\Ft^*)$ lands in the first factor, which proves the assertion about its kernel.
\end{proof}

\begin{proposition}
\label{proposition:wheelrelations}
 Let $V$ be a representation of a reductive group $G$. We let $T'$ be an auxiliary torus acting on $V$ such that the actions of $G$ and $T'$ commute. We may decompose $V\cong\bigoplus_{i=1}^rV_i$ where $V_i$ is an irreducible representation of $G$. We assume that $T'$ acts with a fixed weight on each simple summand $V_i$. We let $\chi$ be a character of $T'$ and let $T'$ act on $V^*$ by $\chi$-times the contragredient action. We let $T'$ act on $\Fg$ with character $\chi^{-1}$, so that the potential obtained by contraction is invariant. We let $\mu\colon\Tan^*V\rightarrow\Fg^*$ be the moment map for the action of $G$ on $\Tan^*V$. We still denote by $\chi\in\Ft'^*$ the linear form on $\Ft'\coloneqq\Lie(T')$ induced by $\chi$ and we let $t_{i}\in\Ft'^*$ be the linear form corresponding to the character of $T'$ acting on $V_{i}$. We let $\Fh\coloneqq\Lie(T)$ be the Lie algebra of a maximal torus $T\subset G$.
 
 Then, the image of
 \[
  \overline{\xi}\colon\CH_{G,\Tan^*V,\Fg,\mu}^{T'}\rightarrow\CH_{G,\Tan^*V\times\Fg}^{T'}\cong\Sym_{\HO^*_{T'}}(\Fh^*)^{\sfW}
 \]
 is contained in the subspace of polynomials $P\in \Sym(\Fh^*\oplus\Ft'^*)$ such that the restriction of $P$ to the subspaces
 \[
  t_{i}+\alpha=\chi-t_{i}-\beta=0\quad\quad\text{in }\Fh\oplus\Ft'
 \]
 for any weights $\alpha,\beta$ of $V_i$ such that $\beta-\alpha$ is a root of $G$ and $1\leq i\leq r$ vanishes.
\end{proposition}
\begin{proof}
 Let $\alpha,\beta$ be weights of $V_i$ such that $\beta-\alpha$ is a root of $G$. We let $x\in V_{i,\alpha}\setminus\{0\}$, $\zeta\in\Fg_{\beta-\alpha}\setminus\{0\}$ and $y\in V_{i,\beta}\setminus\{0\}$ be such that $\zeta\cdot x=y$ (Lemma~\ref{lemma:representationtheory}). We let $\ell\in (V_{i,\beta})^*=(V_i^*)_{-\beta}$ be such that $\ell(y)\neq 0$, and extend $\ell$ by zero on the other summands of $V$. Then, we define
 \[
  E_{x,\ell}\coloneqq \BoC^*x\times \BoC^*\ell\subset \Tan^*V\,.
 \]
We let $\jmath\colon E_{x,\ell}\rightarrow\Tan^*V$ be the inclusion. Then, $E_{x,\ell}\cap\mu^{-1}(0)=\emptyset$, as $\ell(\zeta\cdot x)=\ell(y)\neq0$. Therefore, $\jmath^*\imath_*=0$ where $\imath\colon\mu^{-1}(0)\rightarrow \Tan^*V$ is the natural inclusion. This implies that $\imath_*P\in\ker(\jmath^*)$ for any $P\in \CH_{G,\Tan^*V,\Fg,\mu}^{T'}$. After restriction to the maximal torus $T'\times T$ of $T'\times G$, the action on $E_{x,\ell}$ has weight $t_{i}+\alpha$ on $\BoC^*x$ and with weight $\chi-t_{i}-\beta$ on $\BoC^*\ell$. By the kernel statement of Lemma~\ref{lemma:equivariantcohomology}, the kernel of $\jmath^*$ is the kernel of the quotient
\[
\Sym(\Fh^*\oplus\Ft'^*)\rightarrow \frac{\Sym(\Fh^*\oplus\Ft'^*)}{\langle t_{i}+\alpha,\chi-t_{i}-\beta\rangle}\,.
\]
This concludes.
\end{proof}

When the auxiliary torus $T'$ is the torus $T_2$ defined in \S\ref{subsection:from2dtoshuffle}, we obtain the following corollary.
\begin{corollary}
\label{corollary:wheelrelations}
 We let $t_j\in\Ft_2^*$ be the character of $T_2$ which is the projection on the $j$th factor. The image of $\overline{\xi}$ is contained in the subset of polynomials $P\in\Sym_{\HO^*_{T_2}}(\Fh^*)$ whose restriction to the subspaces
 \[
  t_1+\alpha=t_2-\beta=0\quad\text{in }\Fh\oplus\Ft_2
 \]
for any weights $\alpha,\beta$ belonging to the same irreducible summand of $V$ such that $\beta-\alpha$ is a root of $G$ vanishes.
\end{corollary}
\begin{proof}
 This follows from Proposition~\ref{proposition:wheelrelations} by taking $T'=T_2$, which acts on $V$ with weight $t_1$ and on $V^*$ with weight $t_2=\chi-t_1$, where $\chi=t_1+t_2$ in additive notation. This concludes.
\end{proof}

\begin{question}
\label{conjecture:wheelrelations}
Which additional relations, beyond those of Corollary~\ref{corollary:wheelrelations}, characterize the image of $\overline\xi$?
\end{question}
The root conditions alone do not suffice in this generality, as the following example shows. It may be necessary to add relations coming from the Cartan part of the moment map, as well as higher order wheel relations.

\begin{example}
Let $G=\BoC^*$ and let $V=\BoC$ have weight $1$. Let $x\in\HO^2_G(\pt)$, and let $T_2=(\BoC^*)^2$ act with weights $t_1$ on $V$ and $t_2$ on $V^*$. The moment map is $\mu(v,\ell)=v\ell$, so its zero locus is the union of the two coordinate axes. Under the Thom identifications, $\overline\xi$ is the pushforward from this union to $V\oplus V^*$; the extra $\Fg$-factor does not change its image. We claim that
\[
 \operatorname{im}(\overline\xi)
 =(x+t_1,-x+t_2)\subset\BoQ[x,t_1,t_2].
\]
Indeed, the fundamental classes of the two axes push forward to their normal Euler classes $-x+t_2$ and $x+t_1$, so the displayed ideal is contained in the image. Conversely, restriction to the complement $(\BoC^*)^2$ annihilates the image. By Lemma~\ref{lemma:equivariantcohomology}, the kernel of this restriction is precisely the displayed ideal. Since $G$ has no roots, the conditions of Corollary~\ref{corollary:wheelrelations} are empty, whereas this ideal is proper.
\end{example}

\begin{remark}
A careful reader will have noticed the similarity between Proposition~\ref{proposition:wheelrelations}, Corollary~\ref{corollary:wheelrelations} and \cite[Proposition~2.11]{neguct2023shuffle}. While we work in cohomology, \emph{op. cit.} works in $K$-theory. The same disjoint-support argument gives necessary relations in $K$-theory, with Euler classes replaced by their multiplicative analogues. The shuffle descriptions and wheel relations in the quiver setting \cite{neguct2023shuffle,neguct2023generators} motivate Question~\ref{conjecture:wheelrelations}; their precise extension to arbitrary reductive-group representations remains to be determined.
\end{remark}

\section{Examples}
\label{section:examples}
We first illustrate the Mackey formula when the potential vanishes, using the explicit shuffle formulas of \S\ref{subsection:explicitformula}. The final example has a nonzero quadratic potential. To ease the reading of this section, we recall the formulas for the induction and restriction morphisms here. We let $V$ be a representation of a reductive group $G$. Then, for any cocharacters $\lambda,\mu\in\rmX_*(T)$, we define $k_{\lambda,\mu}=\frac{\prod_{\alpha\in\CW^{\lambda<0}(V)\cap\CW^{\mu}(V)}\alpha}{\prod_{\alpha\in\CW^{\lambda<0}(\Fg)\cap\CW^{\mu}(\Fg)}\alpha}$. The corresponding induction morphism is
\[
 \Ind_{\lambda,\mu}\colon\HO^*((V^{\lambda}\cap V^{\mu})/(G^{\lambda}\cap G^{\mu}),\BoQ)\rightarrow \HO^*(V^{\mu}/G^{\mu},\BoQ)\,.
\]
It is given by $\Ind_{\lambda,\mu}(f)=\sum_{w\in\sfW^{\mu}/(\sfW^{\lambda}\cap\sfW^{\mu})}w\cdot(fk_{\lambda,\mu})$. The corresponding restriction morphism is
\[
 \Res_{\mu,\lambda}\colon \HO^*(V^{\mu}/G^{\mu},\BoQ)\rightarrow \HO^*((V^{\lambda}\cap V^{\mu})/(G^{\lambda}\cap G^{\mu}),\BoQ)\left[\prod_{\alpha\in \CW^{\lambda<0}(V)\cap\CW^{\mu}(V)}\alpha^{-1}\right],
\]
and is given by $\Res_{\mu,\lambda}(f)=\frac{1}{k_{\lambda,\mu}}f$.

\subsection{The defining representation of \texorpdfstring{$\GL_2$: $\GL_2\curvearrowright\BoC^2$}{GL2}}
We consider the defining representation $V\coloneqq\BoC^2$ of $G\coloneqq\GL_2(\BoC)$. We fix the maximal torus $T=(\BoC^*)^{2}$ of diagonal matrices inside $\GL_2(\BoC)$. We denote by $\alpha_1,\alpha_2$ a basis of the weight lattice of $T$, so that the weights of $V$ are $\alpha_1,\alpha_2$ and the non-zero weights of the Lie algebra $\Fg$ of $G$ are $\pm(\alpha_1-\alpha_2)$.
\subsubsection{Coxeter system}
We give the Coxeter system of the action of $\GL_2(\BoC)$ on $\BoC^2$. We give a cocharacter inside each cell. For $a,b\in\BoZ$, we denote by $\lambda_{a,b}=(t^a,t^b)$ the cocharacter $t\mapsto \begin{pmatrix}
t^a&0\\0&t^b
\end{pmatrix}$
 of $T$.
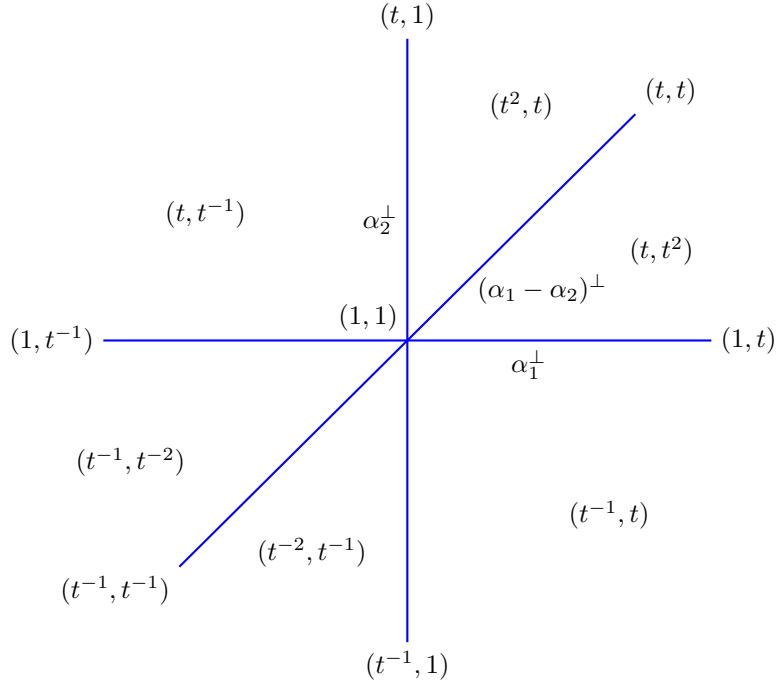
\[
 \begin{tikzpicture}[scale=1]
% Define roots as coordinates
\foreach \x/\y in {4/0, 0/4,-3/-3,3/3,0/-4,-4/0} {
    % Draw arrows for the roots
    \draw[-, thick, blue] (0,0) -- (\x,\y);
}
% Add labels for some roots
\node[anchor=north] at (1.6,0) {$\alpha_1^{\perp}$};
\node[anchor=east] at (0,1.6) {$\alpha_2^{\perp}$};
\node[anchor=north west] at (0.8,1) {$(\alpha_1-\alpha_2)^{\perp}$};
% Label for cocharacters
\node[anchor=south west] at (3,3) {$(t,t)$};
\node[anchor=north west] at (2.8,1.5) {$(t,t^2)$};
\node[anchor=west] at (4,0) {$(1,t)$};
\node[anchor=north west] at (2,-2) {$(t^{-1},t)$};
\node[anchor=north west] at (-2.1,-2.5) {$(t^{-2},t^{-1})$};
\node[anchor=north] at (0,-4) {$(t^{-1},1)$};
\node[anchor=north east] at (-3,-3) {$(t^{-1},t^{-1})$};
\node[anchor=north east] at (-2.8,-1.3) {$(t^{-1},t^{-2})$};
\node[anchor=north east] at (-2,2) {$(t,t^{-1})$};
\node[anchor=east] at (-4,0) {$(1,t^{-1})$};
\node[anchor=south] at (0,4) {$(t,1)$};
\node[anchor=south] at (1.5,2.8) {$(t^2,t)$};
\node[anchor=south east] at (0,0) {$(1,1)$};
\end{tikzpicture}
\]

\subsubsection{Induction and restriction}
We illustrate Mackey's formula between induction and restriction in the example of $\GL_2(\BoC)$ acting on $\BoC^2$. We write $x_i=c_1(\alpha_i)\in\HO_T^2(\pt)$, so that $\HO_T^*(\pt)=\BoQ[x_1,x_2]$. In the formulas below, a pair $(a,b)$ denotes the cell containing the cocharacter $\lambda_{a,b}$. For this, we have to calculate several induction morphisms, restriction morphisms and braiding morphisms $\tau_{\lambda}^{\mu}$. This involves the calculation of induction kernels $k_{\lambda}$.

We verify the Mackey formula
\begin{equation}
\label{equation:mackey1}
\Res_{(-1,-1)}^{(-1,-2)}\circ\Ind_{(-2,-1)}^{(-1,-1)}=\sum_{w\in\FS_2}\Ind_{(-1,-2)\circ w\cdot(-2,-1)}^{(-1,-2)}\circ\tau_{w\cdot(-2,-1)\circ (-1,-2)}^{(-1,-2)\circ w\cdot(-2,-1)}\circ\Res_{w\cdot(-2,-1)}^{w\cdot(-2,-1)\circ (-1,-2)}\circ (w\cdot-)\,.
\end{equation}
We let $w\in\FS_2$ be the non-trivial element. We have the following:
\begin{enumerate}
 \item $w\cdot(-2,-1)=(-1,-2)$,
 \item $(-1,-2)\circ w\cdot (-2,-1)=(-1,-2)=w\cdot (-2,-1)\circ (-1,-2)$,
 \item $(-1,-2)\circ (-2,-1)=(-1,-2)$ and $(-2,-1)\circ (-1,-2)=(-2,-1)$,
 \item $\tau_{w\cdot(-2,-1)\circ (-1,-2)}^{(-1,-2)\circ w\cdot (-2,-1)}=\id$,
 \item $\tau_{(-2,-1)}^{(-1,-2)}=\frac{\frac{x_1x_2}{x_1-x_2}}{\frac{x_1x_2}{x_2-x_1}}\cdot -=-\id$,
 \item $\Ind_{(-1,-2)\circ w\cdot(-2,-1)}^{(-1,-2)}=\id$,
 \item $\Res_{w\cdot(-2,-1)}^{w\cdot(-2,-1)\circ (-1,-2)}=\id$,
 \item $k_{(-2,-1),(-1,-1)}=\frac{1}{x_1-x_2}$,
 \item $k_{(-1,-2),(-1,-1)}=\frac{1}{x_2-x_1}$.
\end{enumerate}
The left-hand side of \eqref{equation:mackey1} applied to $f(x_1,x_2)$ is
\[
 (x_2-x_1)\frac{f(x_1,x_2)-f(x_2,x_1)}{x_1-x_2}=f(x_2,x_1)-f(x_1,x_2)\,.
\]
The term of the r.h.s. of \eqref{equation:mackey1} corresponding to $w$ is
\[
f(x_2,x_1)
\]
while the term corresponding to the trivial element of $\FS_2$ is
\[
 -f(x_1,x_2)\,.
\]
Mackey's formula is indeed satisfied in this case.

We verify Mackey's formula
\begin{equation}
\label{equation:mackey2}
 \Res_{(0,-1)}^{(1,-1)}\circ\Ind_{(-1,-2)}^{(0,-1)}=\tau_{(-1,-2)}^{(1,-1)}\,.
\end{equation}
We have
\begin{enumerate}
 \item $k_{(1,-1),(0,-1)}=1$,
 \item $k_{(-1,-2),(0,-1)}=x_1$,
 \item $\tau_{(-1,-2)}^{(1,-1)}=\frac{\frac{x_1x_2}{x_2-x_1}}{\frac{x_2}{x_2-x_1}}=x_1$.
\end{enumerate}
Therefore, the l.h.s. of \eqref{equation:mackey2} applied to $f(x_1,x_2)$ is $x_1f(x_1,x_2)$ which is also the r.h.s. of \eqref{equation:mackey2} given the formula for $\tau_{(-1,-2)}^{(1,-1)}$.

We verify Mackey's formula
\begin{equation}
 \label{equation:mackey3}
 \Res_{(0,0)}^{(1,-1)}\circ \Ind_{(-1,1)}^{(0,0)}=\sum_{w\in\FS_2}\Ind_{(1,-1)\circ w\cdot(-1,1)}^{(1,-1)}\circ \tau_{w\cdot (-1,1)\circ(1,-1)}^{(1,-1)\circ w\cdot (-1,1)}\circ\Res_{w\cdot(-1,1)}^{w\cdot(-1,1)\circ(1,-1)}\circ (w\cdot -)\,.
\end{equation}
The term of the r.h.s. of \eqref{equation:mackey3} corresponding to the trivial $w\in\FS_2$ is
\[
 \tau_{(-1,1)}^{(1,-1)}\,.
\]
The term of \eqref{equation:mackey3} corresponding to the non-trivial $w\in\FS_2$ is
\[
w\cdot -\,.
\]
We have
\begin{enumerate}
 \item $k_{(1,-1),(0,0)}=\frac{x_2}{x_2-x_1}$,
 \item $k_{(-1,1),(0,0)}=\frac{x_1}{x_1-x_2}$,
 \item $\tau_{(-1,1)}^{(1,-1)}=-\frac{x_1}{x_2}$\,.
\end{enumerate}
Therefore,
\[
 \Ind_{(-1,1)}^{(0,0)}f(x_1,x_2)=\frac{x_1f(x_1,x_2)-x_2f(x_2,x_1)}{x_1-x_2}
\]
and
\[
 \begin{aligned}
 \Res_{(0,0)}^{(1,-1)}\circ\Ind_{(-1,1)}^{(0,0)}f(x_1,x_2)&=\frac{x_1f(x_1,x_2)-x_2f(x_2,x_1)}{x_1-x_2}\frac{x_2-x_1}{x_2}\\
 &=f(x_2,x_1)-\frac{x_1}{x_2}f(x_1,x_2)\,.
 \end{aligned}
\]
On the other hand, for the nontrivial $w\in\FS_2$, one has
\[
 w\cdot f(x_1,x_2)=f(x_2,x_1)
\]
and
\[
 \tau_{(-1,1)}^{(1,-1)}f(x_1,x_2)=-\frac{x_1}{x_2}f(x_1,x_2)\,.
\]
By combining these formulas, one obtains \eqref{equation:mackey3}.

We now verify Mackey's formula
\begin{equation}
\label{equation:mackey4}
\Res_{(0,0)}^{(-2,-1)}\circ\Ind_{(1,0)}^{(0,0)}=\sum_{w\in\FS_2}\Ind_{(-2,-1)\circ w\cdot(1,0)}^{(-2,-1)}\circ\tau_{w\cdot(1,0)\circ (-2,-1)}^{(-2,-1)\circ w\cdot(1,0)}\circ\Res_{w\cdot(1,0)}^{w\cdot(1,0)\circ(-2,-1)}\circ (w\cdot -)\,.
\end{equation}
The term of the r.h.s. of \eqref{equation:mackey4} corresponding to the identity element $w\in\FS_2$ is
\[
 \tau_{(1,-1)}^{(-2,-1)}\circ\Res_{(1,0)}^{(1,-1)}\,.
\]
The term of the r.h.s. of \eqref{equation:mackey4} corresponding to the nontrivial element $w\in\FS_2$ is
\[
 \tau_{(-1,1)}^{(-2,-1)}\circ\Res_{(0,1)}^{(-1,1)}\circ (w\cdot-)\,.
\]
We have
\begin{enumerate}
 \item $k_{(-2,-1),(0,0)}=\frac{x_1x_2}{x_1-x_2}$
 \item $k_{(1,0),(0,0)}=\frac{1}{x_2-x_1}$
 \item $k_{(1,-1),(1,0)}=x_2$
 \item $k_{(-1,1),(0,1)}=x_1$
 \item $\tau_{(1,-1)}^{(-2,-1)}=\frac{\frac{x_2}{x_2-x_1}}{\frac{x_1x_2}{x_1-x_2}}=-\frac{1}{x_1}$
 \item $\tau_{(-1,1)}^{(-2,-1)}=\frac{\frac{x_1}{x_1-x_2}}{\frac{x_1x_2}{x_1-x_2}}=\frac{1}{x_2}$\,.
\end{enumerate}
We have
\[
\Ind_{(1,0)}^{(0,0)}f(x_1,x_2)=\frac{f(x_1,x_2)-f(x_2,x_1)}{x_2-x_1}
\]
so that the l.h.s. of \eqref{equation:mackey4} is
\[
 \Res_{(0,0)}^{(-2,-1)}\circ\Ind_{(1,0)}^{(0,0)}f(x_1,x_2)=\frac{f(x_2,x_1)-f(x_1,x_2)}{x_1x_2}\,.
\]
We have
\[
 \tau_{(1,-1)}^{(-2,-1)}\circ\Res_{(1,0)}^{(1,-1)}f(x_1,x_2)=-\frac{1}{x_1}\frac{1}{x_2}f(x_1,x_2)
\]
and
\[
 \tau_{(-1,1)}^{(-2,-1)}\circ\Res_{(0,1)}^{(-1,1)}\circ (w\cdot -)\bigl(f(x_1,x_2)\bigr)=\frac{1}{x_2}\frac{1}{x_1}f(x_2,x_1)\,.
\]
By combining all these equalities, we obtain Mackey's formula \eqref{equation:mackey4}\,.

We verify Mackey's formula on a fifth example.

\begin{equation}
 \label{equation:mackey5}
 \begin{aligned}
\Res_{(0,0)}^{(-1,-1)}\circ\Ind_{(1,0)}^{(0,0)}&=\Ind_{(-1,-1)\circ(1,0)}^{(-1,-1)}\circ\tau_{(1,0)\circ (-1,-1)}^{(-1,-1)\circ(1,0)}\circ \Res_{(1,0)}^{(1,0)\circ(-1,-1)}\\
&=\Ind_{(-1,-2)}^{(-1,-1)}\circ \tau_{(1,-1)}^{(-1,-2)}\circ \Res_{(1,0)}^{(1,-1)}\,.
 \end{aligned}
\end{equation}

We have
\begin{enumerate}
 \item $k_{(-1,-1),(0,0)}=x_1x_2$
 \item $k_{(1,0),(0,0)}=\frac{1}{x_2-x_1}$
 \item $k_{(-1,-2),(-1,-1)}=\frac{1}{x_2-x_1}$
 \item $k_{(1,-1),(1,0)}=x_2$
 \item $\tau_{(1,-1)}^{(-1,-2)}=\frac{\frac{x_2}{x_2-x_1}}{\frac{x_1x_2}{x_2-x_1}}=\frac{1}{x_1}$\,.
\end{enumerate}
We calculate:
\[
 \Res_{(0,0)}^{(-1,-1)}\circ\Ind_{(1,0)}^{(0,0)}f(x_1,x_2)=\frac{f(x_1,x_2)-f(x_2,x_1)}{x_2-x_1}\frac{1}{x_1x_2}
\]
and
\[
\Ind_{(-1,-2)}^{(-1,-1)}\circ \tau_{(1,-1)}^{(-1,-2)}\circ \Res_{(1,0)}^{(1,-1)}f(x_1,x_2)=\frac{f(x_1,x_2)}{x_1x_2(x_2-x_1)}+\frac{f(x_2,x_1)}{x_1x_2(x_1-x_2)}\,.
\]
We obtain exactly \eqref{equation:mackey5}\,.

\subsection{Torus-equivariant Landau--Ginzburg models}
The situation for torus-equivariant LG-models is much more elementary than in the general case, as there is no Weyl group. The Mackey formula therefore reduces to one term and is straightforward to prove when the function $f$ vanishes. We have to verify that for any $\lambda,\nu\preceq\mu$
\[
 \Res_{\mu}^{\nu}\circ\Ind_{\lambda}^{\mu}=\Ind_{\nu\circ\lambda}^{\nu}\circ\tau_{\lambda\circ\nu}^{\nu\circ\lambda}\circ\Res_{\lambda}^{\lambda\circ\nu}\,.
\]
In terms of formulas, we have to verify the equality
\[
 \frac{k_{\lambda,\mu}}{k_{\nu,\mu}}=k_{\nu\circ\lambda,\nu}\frac{k_{\lambda\circ\nu}}{k_{\nu\circ\lambda}}\frac{1}{k_{\lambda\circ\nu,\lambda}}\,.
\]
Since $k_{a,b}=\frac{k_a}{k_b}$ for $a\preceq b$, both sides simplify to $\frac{k_{\lambda}}{k_{\nu}}$. This concludes.

\subsection{A quadratic potential and its sign}
\label{subsection:quadratic-example}
Let $G=\BoC^*$ act on $\BoA^2$ by $t\cdot(x,y)=(tx,t^{-1}y)$, and take the invariant potential $f(x,y)=xy$. Write $u\in\HO^2(B\BoC^*)$ for the positive character, $C_+$ and $C_-$ for the positive and negative chambers, and $F=\{0\}$. The chamber fixed loci are points. Dimensional reduction for the section $\mu(x)=x$ of the dual of the $y$-line gives
\[
 \CH_{G,\BoA^2,f,F}\cong\BoQ[u]\otimes\SL.
\]
Let $\epsilon=\mathsf{dr}_F(1)$, including this Thom twist, so that $\epsilon$ has cohomological degree $2$. The zero locus of $\mu$ is a point and both $2d$ induction maps are the identity. Proposition~\ref{proposition:comparison2d3dmultiplications} therefore gives
\[
 \Ind_{C_+}^{F}(1)=-\epsilon,\qquad
 \Ind_{C_-}^{F}(1)=\epsilon.
\]
Indeed, $(W^*)^{C_+>0}$ is one-dimensional and $(W^*)^{C_->0}=0$, where $W$ is the $y$-line.

For $i\colon\{0\}\hookrightarrow\BoA^2$, critical self-intersection of the $x$-axis gives $i^*(-\epsilon)=-u$, hence $i^*\epsilon=u$. Since $k_{C_+,F}=-u$ and $k_{C_-,F}=u$, the restriction maps are
\[
 \Res_F^{C_+}(\epsilon)=-1,\qquad
 \Res_F^{C_-}(\epsilon)=1.
\]
The Weyl group is trivial. For equal chambers the Mackey composite is the identity, whereas for opposite chambers it is $-\id$. This agrees with $\tau_{C_+,F}^{C_-,F}=\tau_{C_-,F}^{C_+,F}=-1$. Thus this nonzero-potential example checks both the quadratic dimensional-reduction sign and the direction of the braiding.

\printbibliography
\end{document}